\documentclass[11pt]{amsart}
\usepackage{hyperref}
\usepackage{latexsym,color}
\title[On Measure Solutions of the Boltzmann Equation part II\dots]{On Measure
  Solutions of the Boltzmann Equation part II: Rate of Convergence to Equilibrium }

\author{Xuguang Lu and Cl\'{e}ment Mouhot}

\def\signxl{\bigskip \begin{center} {\sc Xuguang Lu\par\vspace{3mm}
Tsinghua University\par
Department of Mathematical Sciences\par
Beijing 100084, P.R.,
CHINA\par\vspace{5mm}
e-mail:} \tt{xglu@math.tsinghua.edu.cn} \end{center}}

\def\signcm{\bigskip \begin{center} {\sc
Cl\'ement Mouhot\par\vspace{5mm}
University of Cambridge\par
DPMMS, Centre for Mathematical Sciences\par
Wilberforce Road,
Cambridge CB3 0WA,
UK\par\vspace{5mm}
e-mail:} \tt{C.Mouhot@dpmms.cam.ac.uk} \end{center}}

\begin{document}
\maketitle                            

\newcommand{\p}{\partial}
\newcommand{\og}{\omega}
\newcommand{\Og}{\Omega}
\newcommand{\dt}{\delta}
\newcommand{\Dt}{\Delta}
\newcommand{\ld}{\lambda}
\newcommand{\Ld}{\Lambda}
\newcommand{\Gm}{\Gamma}
\newcommand{\gm}{\gamma}
\newcommand{\vp}{\varphi}
\newcommand{\vep}{\varepsilon}
\newcommand{\ep}{\epsilon}
\newcommand{\vh}{\varrho}
\newcommand{\vap}{\varphi}
\newcommand{\kp}{\eta}
\newcommand{\Sg}{\Sigma}
\newcommand{\fa}{\forall}
\newcommand{\fr}{\frac}
\newcommand{\sg}{\sigma}
\newcommand{\la}{\langle}
\newcommand{\ra}{\rangle}
\newcommand{\wt}{\widetilde}
\newcommand{\be}{\begin{equation}}
\newcommand{\ee}{\end{equation}}
\newcommand{\ba}{\begin{array}}
\newcommand{\ea}{\end{array}}
\newcommand{\bea}{\begin{eqnarray}}
\newcommand{\eea}{\end{eqnarray}}
\newcommand{\beas}{\begin{eqnarray*}}
\newcommand{\eeas}{\end{eqnarray*}}
\newcommand{\dpm}{\displaystyle }
\newcommand{\intt}{\int\!\!\!\!\int}
\newcommand{\inttt}{\int\!\!\!\!\int\!\!\!\!\int}
\newcommand{\intttt}{\int\!\!\!\!\int\!\!\!\!\int\!\!\!\!\int}
\newcommand{\bRN}{{\mathbb R}^N}
\newcommand{\bSN}{{\mathbb S}^{N-1}}
\newcommand{\bRRN}{{\mathbb R}^N\times{\mathbb R}^N}
\newcommand{\bRSN}{{\mathbb R}^N\times{\mathbb S}^{N-1}}
\newcommand{\bRRRN}{{\mathbb R}^N\times{\mathbb R}^N\times{\mathbb R}^N}
\newcommand{\bRRSN}{{\mathbb R}^N\times{\mathbb R}^N\times{\mathbb S}^{N-1}}

\theoremstyle{theorem}
\newtheorem{theorem}{Theorem}
\newtheorem{lemma}[theorem]{Lemma}
\newtheorem{conjecture}[theorem]{Conjecture}
\newtheorem{corollary}[theorem]{Corollary}
\newtheorem{proposition}[theorem]{Proposition}
\theoremstyle{definition}
\newtheorem{definition}[theorem]{Definition}

\def\theThm{{\arabic{section}.\arabic{theorem}}}
\numberwithin{equation}{section}
\numberwithin{theorem}{section}

\theoremstyle{remark}
\newtheorem{remark}[theorem]{Remark}
\newtheorem{remarks}[theorem]{Remarks}
\newtheorem{examples}[theorem]{Examples}
\newtheorem{example}[theorem]{Example}

\vskip 0.1in \baselineskip 17pt

\begin{abstract}
  The paper considers the convergence to equilibrium for measure
  solutions of the spatially homogeneous Boltzmann equation for hard
  potentials with angular cutoff. We prove the exponential sharp rate
  of strong convergence to equilibrium for conservative measure
  solutions having finite mass and energy. The proof is based on the
  regularizing property of the iterated collision operators,
  exponential moment production estimates, and some previous results
  on the exponential rate of strong convergence to equilibrium for
  square integrable initial data. We also obtain a lower bound of the
  convergence rate and deduce that no eternal solutions exist apart
  from the trivial stationary solutions given by the Maxwellian
  equilibrium. The constants in these convergence rates depend only on
  the collision kernel and conserved quantities (mass, momentum, and
  energy).  We finally use these convergence rates in order to deduce
  global-in-time strong stability of measure solutions.
\end{abstract}

\bigskip


\bigskip


\tableofcontents

\section {Introduction}
\label{sec1}

The Boltzmann equation describes evolution of a dilute
gas. Investigations of the spatially homogeneous Boltzmann equation
have made a lot of progresses in the last decades and it is hoped to
provide useful clues for the understanding of the complete (spatially
inhomogeneous) Boltzmann equation. The complete equation is more
realistic and interesting to physics and mathematics but remains still
largely out of reach mathematically and will most likely need long
term preparations and efforts.  For review and references of these
areas, the reader may consult for instance
\cite{Villani-handbook,MR2765747,partI}.

The present paper is a follow-up to our previous work \cite{partI} on
measure-valued solutions\footnote{As in our previous
  work~\cite{partI}, the ``measure-valued solutions" will be also
  called ``measure solutions''.}  to the spatially homogeneous
Boltzmann equation for hard potentials.  In this second part, we prove
that, under some angular cutoff assumptions (which include the hard
sphere model), solutions with measure-valued initial data having
finite mass and energy converge strongly to equilibrium in the
exponential rate $e^{-\ld t}$, where $\ld>0$ is the spectral gap of
the corresponding linearized collision operator. This sharp
exponential rate was first proved in \cite{Mcmp} for initial data with
bounded energy, and belonging to $L^1$ (for the hard sphere model) or
to $L^1\cap L^2$ (for all hard potentials with cutoff). The core idea
underlying our improvement of this result to measure solutions is that
instead of considering a one-step iteration of the collision integral
which produces the $L^1\cap L^2$ integrability for the hard sphere
model (as first observed by Abrahamsson \cite{MR1697495}, elaborating
upon an idea in \cite{MR1697562}), we consider a \emph{multi-steps}
iteration which produces the $L^1\cap L^{\infty}$ integrability for
all hard potentials with angular cutoff. This, together with
approximation by $L^1$ solutions through the Mehler transform, and the
property of the exponential moment production, enables us to apply the
results of \cite{Mcmp} and obtain the same convergence rate
$e^{-\ld t}$ for measure solutions.  We also obtain a lower bound of
the convergence rate and establish the global in time strong stability
estimate. As a consequence we prove that, for any hard potentials with
cutoff, there are no eternal measure solutions with finite and
non-zero temperature, apart from the Maxwellians.

\subsection{The spatially homogeneous Boltzmann equation}
The spatially homogeneous Boltzmann equation takes the form
\begin{equation}\label{(B)}
\frac{\p }{\p t}f_t(v)=Q(f_t,f_t)(v), \quad (v,t)\in
\mathbb{R}^N\times(0,\infty),\quad N\ge 2
\end{equation}
with some given initial data $f_t(v)|_{t=0}=f_0(v) \ge 0$, where $Q$ is the
\emph{collision integral} defined by
\begin{equation}\label{(1.1)}
Q(f,f)(v)=\intt_{\mathbb{R}^N\times
\mathbb{S}^{N-1}}B(v-v_*,\sg)\Big(f(v')f(v_*')-f(v)f(v_*)\Big)
\, {\rm d}\sigma \, {\rm d}v_*.
\end{equation}
In the latter expression, $v,v_*$ and $v', v_*'$ stand for velocities
of two particles after and before their collision, and the microscopic
conservation laws of an elastic collision
\begin{equation}\label{(1.3)}
  v'+v_*'=v+v_*,\quad |v'|^2+|v_*'|^2 = |v|^2 + |v_*|^2.
\end{equation}
induce the following relations:
\begin{equation}\label{(1.2)}
  v'=\frac{v+v_*}{2}+\frac{|v-v_*|}{2}\sg,\quad
  v_*'=\frac{v+v_*}{2}-\frac{|v-v_*|}{2}\sg
\end{equation}
for some unit vector $\sg \in \mathbb{S}^{N-1}$.

The collision kernel $B(z,\sg)$ under consideration is assumed to have
the following \emph{product form}
\begin{equation}\label{(1.4)}
B(z,\sg)=|z|^{\gm}b\left(\frac{z}{|z|}\cdot\sg\right)\, ,\quad \gamma>0
\end{equation}
where $b$ is a nonnegative Borel function on $[-1,1]$. This
corresponds to the so-called \emph{inverse power-law interaction
  potentials} between particles, and the condition $\gamma>0$
corresponds to the so-called \emph{hard potentials}. Throughout this paper we
assume that the function $b$ satisfies \emph{Grad's angular cutoff}:
\begin{equation}\label{(Grad)}
A_0:=\int_{{\mathbb S}^{N-1}}b\left(\frac{z}{|z|}\cdot\sg\right){\rm d}\sg=\left|{\mathbb S}^{N-2}\right|
  \int_{0}^{\pi}b(\cos\theta)\sin^{N-2}\theta\,{\rm d}\theta<\infty
\end{equation}
and it is always assumed that $A_0>0$,  where $|{\mathbb S}^{N-2}|$ denotes the Lebesgue measure of the
$(N-2)$-dimensional sphere ${\mathbb S}^{N-2}$ (recall that in the case $N=2$ we have ${\mathbb S}^{0}=\{-1,1\}$ and $|{\mathbb S}^{0}|=2$ ). This enables us to
split the collision integral as
$$
Q(f,g)=Q^{+}(f,g)-Q^{-}(f,g)
$$
with the two bilinear operators
\begin{eqnarray}\label{(Q+fg)}
  && Q^{+}(f,g)(v)=\intt_{{\mathbb
      R}^N\times{\mathbb S}^{N-1}} B(v-v_*,\sg)f(v')g(v_*') \, {\rm
    d}\sg \, {\rm d}v_*,
  \\
  &&\label{(Q-fg)} Q^{-}(f,g)(v)=A_0f(v)\int_{{\mathbb
      R}^N}|v-v_*|^{\gm} g(v_*) \, {\rm d}v_*.
\end{eqnarray}
which are nonnegative when applied to nonnegative functions.

The bilinear operators $Q^{\pm}$ are bounded from
$L^1_{s+\gm}({\mathbb R}^N) \times L^1_{s+\gm}({{\mathbb R}^{N}})$ to
$L^1_{s}({\mathbb R}^N)$ for $s\ge 0$, where $L^1_s({\mathbb R}^N)$ is
a subspace of $L^1_0({\bRN}):= L^1({\bRN})$ defined by
\begin{equation}\label{(1-Lspace)}
f\in L^1_s({\mathbb R}^N) \quad \Longleftrightarrow \quad \|f\|_{L^1_s}:=
\int_{{\mathbb R}^N}\la v\ra^{s} |f(v)|\, {\rm d}v<\infty.
\end{equation}
where we have used the standard notation
$$
\forall \, v \in \mathbb R, \quad \la v\ra:=\sqrt{1+|v|^2}.
$$
Since in the equation~\eqref{(B)}, $f=g=f_t$, by replacing
\[
B(v-v_*,\sg) \quad \mbox{ with } \quad
\frac{1}{2}\big[B(v-v_*,\sg)+B(v-v_*,-\sg)\big]
\]
one can assume without loss of generality that the function $b$ is
even: $b(-t)=b(t)$ for all $t\in [-1,1]$. This in turn implies that
the polar form of $Q^+$ satisfies
\begin{equation}\label{(1-Qcommut)}
Q^{+}(f,g) \equiv Q^{+}(g,f).
\end{equation}

\subsection{The definition of the solutions}
\label{sec:notion-solutions}

The equation~\eqref{(B)} is usually solved as an integral equation as
follows. Given any $0\le f_0\in L^1_2({\mathbb R}^N)$, we say that a
nonnegative Lebesgue measurable function $(v,t)\mapsto f_t(v)$ on
$[0,\infty)\times{\mathbb R}^N$ is a mild solution to
\eqref{(B)} if for every $t\ge 0$, $v\mapsto f_t(v)$ belongs to
$L^1_2({\mathbb R}^N)$, $\sup_{t\ge 0}\|f_t\|_{L^1_2}<\infty$, and
there is a Lebesgue null set $Z_0$ (which is independent of $t$) such
that
\begin{equation}
\begin{cases}\displaystyle
  \forall\,t\in[0,\infty),\quad \forall\, v\in {\mathbb
    R}^N\setminus Z_0, \quad
  \int_{0}^{t}Q^{\pm}(f_{\tau},f_{\tau})(v)\, {\rm
    d}\tau<\infty,\vspace{0.2cm} \\ \displaystyle
   \forall\,t\in[0,\infty), \quad \forall\,v\in
  {\mathbb R}^N\setminus Z_0, \quad
  f_t(v)=f_0(v)+\int_{0}^{t}Q(f_{\tau},f_{\tau})(v)\, {\rm
    d}\tau.
\end{cases}
\end{equation}

The bilinear operators $(f,g) \mapsto Q^{\pm}(f,g)$ can now be
extended to measures. For every $s\ge 0$, let ${\mathcal B}_s({\mathbb
  R}^N)$ with the norm $\|\cdot\|_s$ be the Banach space of real Borel measures on ${\mathbb R}^N$ defined by
\begin{equation}\label{(1-norm)}
F\in {\mathcal B}_s({\mathbb
  R}^N) \quad \Longleftrightarrow \quad \|F\|_s:=\int_{{\mathbb R}^N }\la v\ra^s \, {\rm d}|F|(v)<\infty,
\end{equation}
where the positive Borel measure
$|F|$ is the total variation of $F$. This norm  $\|\cdot\|_s$ can also
be defined by duality:
\begin{equation} \label{(dual)}
\|F\|_s=\sup_{\vp\in C_c({\bRN}),\,\|\vp\|_{L^{\infty}}\le 1}
\left|\int_{{\bRN}}\vp(v)\la v\ra^s\, {\rm d}F(v)\right|.
\end{equation}
The latter form is convenient when dealing with the difference of two
positive measures.  The norms $\|\cdot\|_s$ and $\|\cdot\|_{L^1_s}$
are related by
\begin{equation}\label{(norm)}
\|F\|_s=\|f\|_{L^1_s}\quad
  \mbox{if}\quad  {\rm d}F(v)=f(v)\, {\rm d}v.
\end{equation}

For any $F,G\in {\mathcal B}_{s+\gm}({\mathbb R}^N)$ ($s\ge 0$), we define the Borel measures
$Q^{\pm}(F,G)$ and
$$Q(F,G)=Q^{+}(F,G)-Q^{-}(F,G)$$ through Riesz's representation
theorem by
\begin{equation}\label{(1-Q+meas)}
  \int_{{\mathbb R}^N}\psi(v)\, {\rm d}Q^{+}(F,G)(v)=
  \intt_{{\mathbb R}^N\times{\mathbb R}^N }L_B[\psi](v,v_*)
  \, {\rm d}F(v)\, {\rm d}G(v_*),
\end{equation}
\begin{equation}\label{(1-Q-meas)}
  \int_{{\mathbb R}^N }\psi(v)\, {\rm d}Q^{-}(F,G)(v)=
  A_0\intt_{{\mathbb R}^N\times{\mathbb R}^N}
  |v-v_*|^{\gm}\psi(v)d F(v)\, {\rm d}G(v_*)
\end{equation}
for all bounded Borel functions $\psi$, where
\begin{equation}\label{(1-LB)}
L_B[\psi](v,v_*)=|v-v_*|^{\gm}\int_{{\mathbb S}^{N-1} }b({\bf n}\cdot \sg)
\psi(v')\,\, {\rm d}\sg,\quad {\bf n}=\frac{v-v_*}{|v-v_*|}
\end{equation}
and in case $v=v_*$ we define ${\bf n}$ to be a fixed unit vector
${\bf e}_1$. It is easily shown (see Proposition 2.3 of \cite{partI})
that the extended bilinear operators $Q^{\pm}$ are also bounded from
${\mathcal B}_{s+\gm}({\mathbb R}^N)\times {\mathcal
  B}_{s+\gm}({\mathbb R}^N)$ to ${\mathcal B}_{s}({\mathbb R}^N)$ for
$s\ge 0$: if $F,G\in{\mathcal B}_{s+\gm}({\mathbb R}^N)$ then
$Q^{\pm}(F,G)\in{\mathcal B}_{s}({\mathbb R}^N)$ and
\begin{equation}\label{(1-Qbound)}
\left\|Q^{\pm}(F,G)\right\|_{s}\le 2^{(s+\gm)/2}A_0\left(\|F\|_{s+\gm}\|G\|_0+\|F\|_{0}
\|G\|_{s+\gm}\right),
\end{equation}
\begin{equation}\label{(1-Q-differ-bound)}
\left\|Q^{\pm}(F,F)-Q^{\pm}(G,G)\right\|_{s}\le
2^{(s+\gm)/2}A_0\left(\|F+G\|_{s+\gm}\|F-G\|_{0}+\|F+
G\|_{0}\|F-G\|_{s+\gm}\right).
\end{equation}

Let us finally define the cone of positive distributions with $s$
moments bounded:
$$
{\mathcal B}_s^+({\mathbb R}^N) :=\left\{ F\in {\mathcal B}_s({\mathbb
    R}^N)\,|\, F\ge 0\right\}.
$$
We can now define the notion of solutions that we shall use in this
paper. We note that the condition $\gamma\in(0,2]$ as assumed in the
following definition is mainly used for ensuring the existence of
solutions.

\begin{definition}[Measure strong solutions]
  Let $B(z,\sg)$ be given by \eqref{(1.4)} with $\gm\in (0,2]$ and
  with $b$ satisfysing the condition \eqref{(Grad)}. Let
  $\{F_t\}_{t\ge 0}\subset {\mathcal B}_2^{+}(\mathbb{R}^N)$. We say
  that $\{F_t\}_{t\ge 0}$, or simply $F_t$, is a {\em measure strong
    solution} of equation~\eqref{(B)} if it satisfies the following:
\begin{itemize}
\item[(i)] $\sup\limits_{t\ge 0}\|F_t\|_2<\infty,$

\item[(ii)] $t\mapsto F_t\in C([0,\infty);{\mathcal
    B}_2(\mathbb{R}^N))\cap C^1([0,\infty);{\mathcal
    B}_0(\mathbb{R}^N))$ and
\begin{equation}\label{def-strong}
\forall \, t\in[0,\infty), \quad \frac{\, {\rm d}}{\, {\rm d}t}F_t=Q(F_t,F_t).
\end{equation}
\end{itemize}

Furthermore $F_t$ is called a {\em conservative solution} if $F_t$
conserves the mass, momentum and energy, i.e.
$$
\forall\, t\ge 0, \quad \int_{\mathbb{R}^N}\left( \begin{array}{c} 1 \\ v \\
      |v|^2 \end{array} \right)\, {\rm d}F_t(v)= \int_{\mathbb{R}^N}\left( \begin{array}{c} 1 \\ v \\
      |v|^2 \end{array} \right)\, {\rm d}F_0(v).
$$
\end{definition}

Observe that \eqref{(1-Qbound)} and \eqref{(1-Q-differ-bound)} imply
the strong continuity of $t\mapsto F_t\in C([0,\infty);{\mathcal
  B}_2(\mathbb{R}^N))$ and therefore the strong continuity of
$t\mapsto Q(F_t,F_t) \in C([0,\infty);{\mathcal
  B}_0(\mathbb{R}^N))$. Hence the differential equation
\eqref{def-strong} is equivalent to the integral equation
\begin{equation}\label{def-strong-bis}
\forall \, t\ge 0, \quad F_t=F_0+\int_{0}^{t}Q(F_\tau, F_\tau)\, {\rm d}\tau,
\end{equation}
where the integral is taken in the sense of the Riemann integration or
more generally in the sense of the Bochner integration.  Recall also
that here the derivative ${\rm d} \mu_t/{\rm d}t$ and integral
$\int_{a}^{b}\nu_t\, {\rm d}t$ are defined by
$$
\left(\frac{\, {\rm d} }{\, {\rm d}t}\mu_t\right)(E)=\frac{\, {\rm d}}{\, {\rm d}t} \mu_t(E),\quad
\left(\int_{a}^{b}\nu_t \, {\rm d}t\right)(E)= \int_{a}^{b}\nu_t (E)
\, {\rm d}t
$$
for all Borel sets $E\subset \mathbb{R}^N$.

\subsection{Recall of the main results of the first part}
\label{sec:recall-main-results}

The following results concerning moment production and uniqueness of
conservative solutions which will be used in the present paper are
extracted from our previous paper \cite{partI}. The following
properties (a) and (b) are a kind of ``gain of decay'' property of the
flow stating and quantifying how moments of the solutions become
bounded for any positive time even they are not bounded at initial
time; the following properties (c)-(d)-(e) concern the stability
of the flow.

\begin{theorem}[\cite{partI}]\label{(theo1.0)}
  Let $B(z,\sg)$ be defined in \eqref{(1.4)} with $\gm\in (0,2]$ and with the condition \eqref{(Grad)}. Then for any $F_0\in {\mathcal
    B}^{+}_2(\mathbb{R}^N)$ with $\|F_0\|_0> 0$, there exists a unique
  conservative measure strong solution $F_t$ of equation~\eqref{(B)}
  satisfying $F_t|_{t=0}=F_0$.  Moreover this solution satisfies:
\smallskip

\begin{itemize}
\item[(a)] $F_t$ satisfies the moment production estimate:
 \begin{eqnarray}\label{(1.12)}&&
\forall\,t>0,\quad \forall\,
s\ge 0, \quad \|F_t\|_{s}\le {\mathcal K}_s \left(1+\frac{1}{t}
\right)^{\frac{(s-2)^{+}}{\gm}}
\end{eqnarray}
where $(x-y)^+=\max\{x-y,0\}$,
\begin{eqnarray}\label{(1.14)}&&{\mathcal K}_s:=
{\mathcal K}_s(\|F_0\|_0,\|F_0\|_2)=\|F_0\|_2\left[2^{s+7}\frac{\|F_0\|_2}{\|F_0\|_0}
\left(1+\frac{1}{16\|F_0\|_2A_2\gm}\right)\right]^{\frac{(s-2)^{+}}{\gm}}
\\
&&
\label{(1-A2)}
A_2:=\left|\mathbb{S}^{N-2}\right|\int_{0}^{\pi}
b(\cos\theta)\sin^N\theta\,\, {\rm d}\theta.
\end{eqnarray}
\smallskip

\item[(b)] If $\gm\in (0,2)$ or if
\begin{equation} \label{(1-gamma 2)}
  \gm=2\quad {\rm and}\quad \exists\, 1<p<\infty\quad {\rm s.t.}\quad
  \int_{0}^{\pi}[b(\cos\theta)]^{p}\sin^{N-2}\theta\,\, {\rm
    d}\theta<\infty
\end{equation}
then $F_t$ satisfies the exponential moment production estimate:
\begin{eqnarray}
\label{(1.13)}
 \forall\,t>0, \quad \int_{\mathbb{R}^N}e^{\alpha(t)\langle v\rangle ^{\gm}} \, {\rm d}F_t(v)\le
2\|F_0\|_0
\end{eqnarray}
where
\begin{eqnarray}
  &&\label{(1-alpha)}
  \alpha(t)=2^{-s_0} \frac{\|F_0\|_0}{\|F_0\|_2}\left(1-e^{-\beta
      t}\right),  \quad \beta =16\|F_0\|_2A_2\gm>0,
\end{eqnarray}
and $1<s_0<\infty$ depends only on the function $b$ and $\gm$.
\smallskip

\item[(c)] Let $G_t$ be a conservative measure strong solutions of
  equation~\eqref{(B)} on the time-interval $[\tau, \infty)$ with an
  initial datum $G_t|_{t=\tau}=G_\tau\in{\mathcal
    B}^{+}_2(\mathbb{R}^N)$ for some $\tau\ge 0$. Then:
\begin{itemize}
\item If $\tau=0$, then
\begin{equation}\label{(1.21)}
\forall \, t\ge 0, \quad \left\|F_t-G_t\right\|_{2}\le  \Psi_{F_0}\left(\left\|F_0-G_0\right\|_2\right) e^{C(1+t)},\quad
\end{equation}
where
\begin{equation}\label{(1.20)}
  \Psi_{F_0}(r)=r+r^{1/3}+ \int_{|v|>r^{-1/3}}|v|^2\, {\rm d}F_0(v),\quad
  r>0, \qquad
  \Psi_{F_0}(0)=0,
\end{equation}
and $C={\mathcal R}(\gm,A_0,A_{2},\|F_0\|_0, \|F_0\|_2)$ is an
explicit positive continuous function on $({\mathbb R}_+ ^*)^5$.

\item If $\tau>0$, then
\begin{equation}\label{(1.21*)}
\forall \, t\in [\tau,
\infty), \quad \left\|F_t-G_t\right\|_{2} \le \left\|F_\tau-G_\tau\right\|_2 e^{C_{\tau}(t-\tau)},
\end{equation}
where
\[
C_{\tau}:= 4\left({\mathcal
    K}_{2+\gm}+\|F_0\|_2\right)\left(1+\frac{1}{\tau}\right),
\]
and ${\mathcal K}_{2+\gm}$ is defined by (\ref{(1.14)}) with
$s=2+\gm$.
\end{itemize}
\smallskip

\item[(d)] If $F_0$ is absolutely continuous with respect to the
  Lebesgue measure, i.e.
\[
{\rm d}F_0(v)=f_0(v)\, {\rm d}v \quad \mbox{ with } \quad 0\le
  f_0\in L^1_2(\mathbb{R}^N),
\]
then $F_t$ is also absolutely continuous with respect to the Lebesgue
measure: ${\rm d}F_t(v)=f_t(v)\, {\rm d}v$ for all $t\ge 0$, and
$f_t$ is the unique conservative mild solution of equation~\eqref{(B)}
with the initial datum $f_0$.  \smallskip

\item[(e)] If $F_0$ is not a single Dirac distribution, then there is
  a sequence $f_{k,t}$, $k \ge 1$, of conservative mild solutions of
  equation~\eqref{(B)} with initial data $0\le f_{k,0}\in
  L^1_2(\mathbb{R}^N)$ satisfying
  \begin{equation}\label{(1.22)}
  \int_{\mathbb{R}^N} \left( \begin{array}{c} 1 \\ v \\
      |v|^2 \end{array} \right) f_{k,0}(v) \, {\rm d}v=\int_{\mathbb{R}^N}
  \left( \begin{array}{c} 1 \\ v \\
      |v|^2 \end{array} \right)  {\rm
    d}F_0(v),\quad  k=1,2,\dots
\end{equation}
  such that
\begin{equation}\label{(1.23)}
  \forall\,\vp\in C_b(\mathbb{R}^N),\quad
  \forall\,t\ge 0, \quad \lim_{k\to\infty}\int_{\mathbb{R}^N}\vp(v) f_{k,t}(v)\, {\rm d}v
  =\int_{\mathbb{R}^N}\vp(v)\, {\rm d}F_t(v).
\end{equation}
Besides, the initial data can be chosen of the form
$f_{k,0}=I_{n_k}[F_0],\, k=1,2,3,\dots$ where
$\{I_{n_k}[F_0]\}_{k=1}^{\infty}$ is a subsequence of the Mehler
transforms $\{I_n[F_0]\}_{n=1}^{\infty}$ of $F_0$.
\end{itemize}
\end{theorem}

\begin{remarks}\label{remark1.1}
  \begin{enumerate} \item In the physical case, $N=3$ and $0<\gm \le
    1$, the moment estimates \eqref{(1.12)} and \eqref{(1.13)} also
    hold for conservative weak measure solutions of
    equation~\eqref{(B)} without angular cutoff (see~\cite{partI}).

  \item The \emph{Mehler transform}
    \[
    I_n[F](v):=e^{Nn}\int_{{\bRN}}M_{1,0,T}\left(e^n\left(v-u-\sqrt{1-e^{-2n}}\,(v_*-u)\right)\right)
    \, {\rm d}F(v_*) \in L^1_2({\mathbb R}^N)
    \]
    of a measure $F\in {\mathcal B}_2^{+}({\bRN})$ (which is not a
    single Dirac distribution) will be studied in Section 4 (after
    introducing other notations) where we shall show that $I_n[F]$ has
    a further convenient property:
\[
\lim_{n \to \infty} \|I_n[F]-M\|_2 = \|F-M\|_2
\]
and thus it is a useful tool in order to reduce the study of
properties of measure solutions to that of $L^1$ solutions. Here $M$
is the Maxwellian (equilibrium) having the same mass, momentum, and
energy as $F$, see \eqref{eq:max}-\eqref{eq:max-meas} below.
  \end{enumerate}
\end{remarks}

\subsection{Normalization}

In most of the estimates in this paper, we shall try as much as
possible to make explicit the dependence on the basic constants in the
assumptions. But first let us study the reduction that can be obtained
by scaling arguments.

Under the assumption \eqref{(Grad)}, it is easily seen that $F_t$ is a
measure solution of equation~\eqref{(B)} with the angular function $b$
if and only if $t\mapsto F_{A_0^{-1}t}$ is a measure solution of
equation~\eqref{(B)} with the scaled angular function
$A_0^{-1}b$. Therefore without loss of generality we can assume the
normalization
\begin{equation}\label{(Grad-1)}
  A_0=\left|{\mathbb S}^{N-2}\right|
  \int_{0}^{\pi}b(\cos\theta)\sin^{N-2}\theta\,{\rm d}\theta=1.
\end{equation}

Next given any $\rho>0$, $u\in{\bRN}$ and $T>0$, we define the bounded
positive linear operator ${\mathcal N}_{\rho,u,T}$ on ${\mathcal
  B}_2({\bRN})$ as follows: for any $F\in {\mathcal B}_2({\bRN})$,
there is a unique ${\mathcal N}_{\rho,u,T}(F)\in{\mathcal
  B}_2({\bRN})$ such that (thanks to Riesz representation theorem),
\begin{multline*}
\forall \, \psi \mbox{ Borel function s.t. } \sup_{v\in{\bRN}}|\psi(v)|\la v\ra^{-2}
<+ \infty, \\ \int_{{\bRN}}\psi(v)\, {\rm d}{\mathcal N}_{\rho,u,T}(F)(v)=
\frac{1}{\rho}\int_{{\bRN}}\psi\left(\frac{v-u}{\sqrt{T}}\right)\, {\rm d}F(v).\qquad \qquad \qquad
\end{multline*}
We call ${\mathcal N}_{\rho,u,T}$ the \emph{normalization operator}
associated with $\rho,u,T$.  The inverse ${\mathcal
  N}_{\rho,u,T}^{-1}$ of ${\mathcal N}_{\rho,u,T}$ is given by
${\mathcal N}_{\rho,u,T}^{-1}={\mathcal N}_{1/\rho,-u/\sqrt{T},1/T}$,
i.e.
$$\int_{{\bRN}}\psi(v)\, {\rm d}{\mathcal N}_{\rho,u,T}^{-1}(F)(v)=
\rho\int_{{\bRN}}\psi\left(\sqrt{T}\,v+ u\right)\, {\rm d}F(v).
$$
It is easily seen that for every $F\in {\mathcal B}_2({\bRN})$
\begin{eqnarray}
&& \label{(1.34)}\left\|{\mathcal N}_{\rho,u,T}(F)\right\|_{0}=\frac{1}{\rho}\|F\|_0,
\\ &&
\label{(1.35)} \left\|{\mathcal N}_{\rho,u,T}(F)\right\|_{2}\le C_{\rho,|u|,T}\|F\|_2,\\
&&\label{(1.36)}
\left\|{\mathcal N}_{\rho,u,T}^{-1}(F)\right\|_{2}\le C_{1/\rho, |u|/\sqrt{T}, 1/T}\|F\|_2
\end{eqnarray}
where
\begin{eqnarray}&&\label{(1.35c)}
C_{\rho,|u|,T}=
\frac{1}{\rho}\max\left\{ 1+\frac{|u|^2+|u|}{T}\, ; \ \frac{1+|u|}{T}\right\},\\
&&  \label{(1.36c)}
C_{1/\rho,|u|/\sqrt{T},1/T}=
\rho\max\left\{ 1+|u|^2+\sqrt{T}|u|\, ; \ T+\sqrt{T}|u|\right\}.
\end{eqnarray}

We then introduce the subclass ${\mathcal B}_{\rho,u,T}^{+}({\bRN})$
of ${\mathcal B}_2^{+}({\bRN})$ by
\begin{equation} \label{(1.33)}
\begin{cases}\displaystyle
 F\in{\mathcal
  B}_{\rho,u,T}^{+}({\bRN})\Longleftrightarrow
F\in {\mathcal B}_2^{+}({\bRN}) \quad{\rm and}\vspace{0.2cm} \\
\displaystyle
 \int_{{\bRN}}\, {\rm d}F(v)=\rho,\quad
{\color {red}{\fr{1}{\rho}}}\int_{{\bRN}}v\, {\rm d}F(v)=u,\quad
\frac{1}{N\rho}\int_{{\bRN}}|v-u|^2\, {\rm d}F(v) =T.
\end{cases}
\end{equation}
In other words,  $F\in{\mathcal B}_{\rho,u,T}^{+}({\bRN})$ means that
$F$ has the mass $\rho$,  mean-velocity $u$, and the kinetic temperature
$T$. It is obvious that $F_t$ conserves mass, momentum, and energy
is equivalent to that $F_t$ conserves mass, mean-velocity , and kinetic temperature.

When restricting ${\mathcal N}_{\rho,u,T}$ on ${\mathcal
  B}_{\rho,u,T}^{+}({\bRN})$, it is easily seen that
$$
{\mathcal N}_{\rho,u,T}: {\mathcal B}_{\rho,u,T}^{+}({\bRN}) \to
{\mathcal B}_{1,0,1}^{+}({\bRN}),\quad {\mathcal N}^{-1}_{\rho,u,T}:
{\mathcal B}_{1,0,1}^{+}({\bRN})\to {\mathcal
  B}_{\rho,u,T}^{+}({\bRN}).
$$

Similarly we define $L^{1}_{\rho,u,T}({\bRN})$ by
\begin{equation} \label{(1-general)}
f\in L^{1}_{\rho,u,T}({\bRN})\Longleftrightarrow
\begin{cases}\displaystyle
  0\le f\in L^1_2({\bRN}), \quad \int_{{\bRN}}f(v)\, {\rm d}v=\rho,
  \vspace{0.2cm} \\ \displaystyle {\fr{1}{\rho}}\int_{{\bRN}}v f(v)\, {\rm
    d}v=u,\quad \frac{1}{N\rho}\int_{{\bRN}}|v-u|^2 f(v)\, {\rm d}v =T.
\end{cases}
\end{equation}
In this case, the normalization operator ${\mathcal N}= {\mathcal
  N}_{\rho,u,T}: L^{1}_{\rho,u,T}({\bRN})\to L^{1}_{1,0,1}({\bRN})$
is written directly as
\begin{equation}\label{(1-norm-f)}
{\mathcal N}(f)(v)=\frac{T^{N/2}}{\rho}f\left(\sqrt{T}\, v+u\right).
\end{equation}

Recall that the Maxwellian $M\in L^{1}_{\rho,u,T}({\bRN})$ is given by
\begin{equation}\label{eq:max}
M(v) :=\frac{\rho}{(2\pi T)^{N/2}}\exp\left(-\frac{|v-u|^2}{2T}\right).
\end{equation}
For notational convenience we shall do not distinguish between a Maxwellian
distribution $M\in {\mathcal B}_{\rho,u,T}^+({\bRN})$ and its density function
$M\in L^{1}_{\rho,u,T}({\bRN})$: we write without risk of
confusion that
\begin{equation}\label{eq:max-meas}
{\rm d}M(v)=M(v)\, {\rm d}v.
\end{equation}

Due to the homogeneity of $z\mapsto B(z,\sg)=|z|^{\gm}b(\frac{z}{|z|}\cdot\sg)$, we have
\[
L_B\left[\psi\left(\frac{\cdot-u}{\sqrt{T}}\right)\right]
(v,v_*)=T^{\gm/2}L_B[\psi]\left(\frac{v-u}{\sqrt{T}},
  \frac{v_*-u}{\sqrt{T}}\right)
\]
and then by Fubini theorem we get (denoting simply ${\mathcal
  N}={\mathcal N}_{\rho,u,T}$ when no ambiguity is possible)
\[
\forall\, F\in{\mathcal B}_2^+({\bRN}), \quad {\mathcal
  N}\left(Q^{\pm}(F,F)\right)=\rho T^{\gm/2} Q^{\pm}\left({\mathcal N}(F),
{\mathcal N}(F)\right).
\]
Since ${\mathcal N}$ is linear and bounded, this implies that if $F_t$
is a measure strong solution of equation~\eqref{(B)} and $c=\rho T^{\gm/2}$,
then
\[
\frac{\, {\rm d}}{\, {\rm d}t}{\mathcal N}(F_{t/c})=
Q\left({\mathcal N}(F_{t/c}), {\mathcal N}(F_{t/c})\right).
\]
This together with \eqref{(1.34)}-\eqref{(1.36)} leads to the
following statement:

\begin{proposition}[Normalization]\label{(prop1.1)}
 Let $B(z,\sg)$ be defined by \eqref{(1.4)} with $\gm \in (0,2]$ and
  with the condition \eqref{(Grad-1)}.  Let $F_0\in{\mathcal
    B}_{\rho,u,T}^{+}({\bRN})$ with $\rho>0$, $u\in{\bRN}$ and $T>0$,
  and let $F_t$ be the unique conservative measure strong solution of
  equation~\eqref{(B)} with the initial datum $F_0$. Let $M\in
  {\mathcal B}_{\rho,u,T}^{+}({\bRN})$ be the Maxwellian defined by
  \eqref{eq:max}, let ${\mathcal N}:={\mathcal N}_{\rho,u,T}$ be the
  normalization operator, and let $c=\rho T^{\gm/2}$. Then:
\begin{itemize}
\item[(I)] The normalization $t\mapsto {\mathcal N}(F_{t/c})$ is the unique
  conservative measure strong solution of equation~\eqref{(B)} with the
  initial datum ${\mathcal N}(F_0)\in{\mathcal
    B}_{1,0,1}^{+}({\bRN})$.
\smallskip

\item[(II)]  For all $t\ge 0$
\begin{equation*}
\begin{cases}\displaystyle
  \left\|F_{t}-M\right\|_0=\rho\|{\mathcal N}(F_{t})-{\mathcal
    N}(M)\|_0 , \vspace{0.2cm} \\ \displaystyle
  \left\|F_{t}-M\right\|_2
  \le C_{1/\rho,|u|/\sqrt{T},1/T}\|{\mathcal N}(F_{t})-{\mathcal
    N}(M)\|_2, \vspace{0.2cm}\\ \displaystyle
  \left\|{\mathcal N}(F_{t})-{\mathcal N}(M)\right\|_2\le C_{\rho,|u|,T}
  \left\|F_{t}-M\right\|_2
\end{cases}
\end{equation*}
where $C_{\rho,|u|,T}$ and $C_{1/\rho,|u|/\sqrt{T},1/T}$ are given in
\eqref{(1.35c)}-\eqref{(1.36c)}.
\end{itemize}
\end{proposition}

\subsection{Linearized collision operator and spectral gap}

For any nonnegative Borel function $W$ on ${\bRN}$ we define the
weighted Lebesgue space $L^p({\bRN}, W)$ with $1\le p<\infty $ by
$$
f\in L^p({\bRN},W) \quad \Longleftrightarrow \quad \|f\|_{L^p(W)}
:=\left(\int_{{\bRN}}|f(v)|^p W(v)\, {\rm
    d}v\right)^{1/p}<\infty.
$$

Let $B(z,\sg)$ as defined in \eqref{(1.4)} with $\gm \in (0,2]$ and with
$b$ satisfying \eqref{(Grad-1)}. Let $M$
be the Maxwellian with mass $\rho>0$, mean velocity $u$ and temperature
$T>0$ defined in \eqref{eq:max}, and let
\[
L_M: L^2\left({\bRN}, M^{-1}\right)\to L^2\left({\bRN}, M^{-1}\right)
\]
be the linearized collision operator associated with $B(z,\sg)$ and
$M(v)$, i.e.
\begin{equation}\label{(1-linear operator)}
 L_M(h)(v)=\intt_{{\bRSN}}B(v-v_*,\sg)M(v_*)
\left(h'+h_*'-h-h_*\right)\, {\rm d}\sg \, {\rm d}v_*.
\end{equation}
It is well-known that the spectrum $\Sg(L_M)$ of $L_M$ is contained in
$(-\infty, 0]$ and has a positive spectral gap
$S_{b,\gm}(\rho,\mu,T) >0$, i.e.
\[
S_{b,\gm}(\rho,u,T) :=\inf\left\{ \ld>0\,\,|\,\, -\ld\in
  \Sg(L_M)\right\}>0.
\]

Moreover by simple calculations, one has the following scaling property on this
spectral gap
\[
S_{b,\gm}(\rho,u,T)= \rho T^{\gm/2} S_{b,\gm}(1,0,1).
\]

In the spatially homogeneous case, the study of the linearized
collision operator goes back to Hilbert \cite{MR1511713,MR0056184} who
computed the collisional invariant, the linearized operator and its
kernel in the hard spheres case, and showed the boundedness and
``complete continuity'' of its non-local
part. Carleman~\cite{Carleman} then proved the existence of a spectral
gap by using Weyl's theorem and the compactness of the non-local part
proved by Hilbert. Grad~\cite{Grad1,Grad2} then extended these results
to the case of hard potentials with cutoff.  All these results are
based on non-constructive arguments. The first constructive estimates
in the hard spheres case were obtained only recently in
\cite{Baranger-Mouhot} (see also \cite{MR2254617} for more general
interactions, and \cite{MR2301289} for a review). Let us also mention the works
\cite{WCUh:LBE:70,Bobylev1975,Boby:maxw:88} for the different setting
of \emph{Maxwell molecules} where the eigenbasis and eigenvalues can
be explicitly computed by Fourier transform methods. Although these
techniques do not apply here, the explicit formula computed are an
important source of inspiration for dealing with more general physical
models.

\subsection{Main results}

In order to use the results obtained in \cite{Mcmp} (see also
\cite{Vi03a,MR2081030}) for $L^1$ solutions, we shall need the
following additional assumptions for some of our main results:
\begin{eqnarray}
  \label{(b-upperbd)}&&\|b\|_{L^{\infty}}:=\sup_{t\in[-1,1]}b(t)<\infty,\\
  &&\label{(b-lowerbd)} \inf_{t\in[-1,1]}b(t)>0.
\end{eqnarray}
Recall that for the hard sphere model, i.e. $N=3, \gm=1,$ and
 $b \equiv {\rm const}.>0$, the conditions
\eqref{(b-upperbd)}-\eqref{(b-lowerbd)} are satisfied.

The first main result of this paper is concerned with the upper bound
of the rate of convergence to equilibrium when the dimension $N$ is
greater or equal to $3$.

\begin{theorem}[Sharp exponential relaxation rate]\label{(theo1.1)}
Suppose $N\ge 3$ and let $B(z,\sg)$ be given by \eqref{(1.4)}
with  $\gm\in(0,  \min\{2, N-2\}]$ and with $b$ satisfying
\eqref{(Grad-1)}, \eqref{(b-upperbd)}, and
\eqref{(b-lowerbd)}. Let $\rho>0$, $u\in{\bRN}$ and $T>0$, and let
\[
\ld=S_{b,\gm}(\rho,\mu,T)=S_{b,\gm}(1, 0, 1)\, \rho \, T^{\gm/2}>0
\]
be the spectral gap for the linearized collision operator
\eqref{(1-linear operator)} associated with $B(z,\sg)$ and the
Maxwellian $M\in {\mathcal B}^{+}_{\rho,u,T}({\bRN})$. Then for any
conservative measure strong solution $F_t$ of the equation~\eqref{(B)}
with $F_0\in {\mathcal B}^{+}_{\rho,u,T}({\bRN})$ we have:
\[
 \forall\, t\ge 0, \quad \|F_t-M\|_2\le C\|F_0-M\|_2^{1/2} e^{-\ld t}
\]
where
\[
C:=C_0 \, C_{1/\rho,|u|/\sqrt{T},1/T}
\, \left(C_{\rho,|u|,T} \right)^{1/2}
\]
with $C_{\rho,|u|,T}$ and $C_{1/\rho,|u|/\sqrt{T},1/T}$ given in
\eqref{(1.35)} and \eqref{(1.36)}, and with some constant $C_0<\infty$
which depends only on $N$, $\gm$, and the function $b$ (through the
bounds \eqref{(b-upperbd)}, \eqref{(b-lowerbd)}).
\end{theorem}

\begin{remark}\label{remark1.UpBd}
\begin{enumerate}

\item It should be noted that, in addition to the exponential rate,
  Theorem \ref{(theo1.1)} also shows that for the hard potentials
  considered here, the convergence to equilibrium is \emph{grossly
    determined}, i.e. the speed of the convergence only depends only
  on the collision kernel and the conserved macroscopic quantities
  (mass, momentum, energy). This is essentially different from those
  for non-hard potentials (i.e. $\gm\le 0$), see for
  instance~\cite{MR2546739}.

\item Applying Theorem \ref{(theo1.1)} to the normal initial data and
the Maxwellian $F_0, M\in{\mathcal B}_{1,0,1}^{+}({\bRN})$, and using $\left\|F_0-M\right\|_2^{1/2}\le \left(\left\|F_0\right\|_2+\left\|M\right\|_2\right)^{1/2}
  =(2(1+N))^{1/2}$
  we have
$$
\forall\, t\ge 0, \quad \left\|F_t-M\right\|_2\le C_0e^{-\ld t},\quad
\ld=S_{b,\gm}(1,0,1)
$$
where $C_0<\infty$ depends  only on $N$, $\gm$, and the function $b$.  Then by
normalization (using Proposition \ref{(prop1.1)}) and the relation
$S_{b,\gm}(\rho,\mu,T)=S_{b,\gm}(1, 0, 1)\rho T^{\gm/2}$, we conclude
that if $F_0, M\in{\mathcal B}_{\rho,u,T}^{+}({\bRN})$, then for the same constant $C_0$ we have
\begin{equation}\label{(4.36)}
  \left\|F_t-M\right\|_2\le
  C_0 \, C_{1/\rho,|u|/\sqrt{T},1/T} \, e^{-\ld t},\quad t\ge 0;\quad
  \ld=S_{b,\gm}(\rho,u,T).
\end{equation}
This estimate will be used in proving
our next results Corollary~\ref{cor:11+12} and Theorem
\ref{(theo1.3)}.

\item In general, in this paper we say that a constant $C$ depends
  only on some parameters $x_1, x_2, \dots, x_m$, if
  $C=C(x_1, x_2, \dots, x_m)$ is an explicit continuous function of
  $(x_1,x_2,\dots,x_m)\in I$ where $I\subset {\mathbb R}^m$ is a
  possible value range of the parameters $(x_1, x_2, \dots, x_m)$. In
  particular this implies that if $K$ is a compact subset of $I$, then
  $C$ is bounded on $K$. 
\end{enumerate}

  \end{remark}

  The second main result is concerned with the lower bound of the rate
  of convergence to equilibrium.
\begin{theorem}[Lower bound on the relaxation rate]\label{(theo1.2)}
Let $B(z,\sg)$ be given by \eqref{(1.4)} with
  $\gm \in (0,2]$ and with the condition \eqref{(Grad-1)}. Let
  $\rho>0$, $u\in{\bRN}$, $T>0$ and let $M\in{\mathcal
    B}^{+}_{\rho,u,T}({\bRN})$ be the Maxwellian.
 Then for any conservative measure strong solution $F_t$ of equation
  \eqref{(B)} with initial data $F_0\in{\mathcal B}^{+}_{\rho,u,T}({\bRN})$ we
  have:
\begin{itemize}
\item[(i)] If $0<\gm<2$, then
\begin{equation*}
\forall\, t\ge 0, \quad \|F_t-M\|_0\ge
  (4\rho)^{1-\alpha}\|F_0-M\|_0 ^{\alpha}
  \exp\left(-\beta\,t^{\frac{2}{2-\gm}}\right)
\end{equation*}
where
 \[
\alpha=\left(\frac{2}{\gm}\right)^{\frac{\gm}{2-\gm}} \quad \mbox{
  and } \quad  \beta= \left(1-\frac{\gm}{2}\right) \left(2^{6}(N+1)^2\rho T^{\gm/2}\right)^{\frac{2}{2-\gm}}.
\]

\item[(ii)] If $\gm=2$, then
$$
\forall\, t\ge 0, \quad \left\|F_t-M\right\|_0\ge
4\rho\left(\frac{\left\|F_{0}-M\right\|_0}{4\rho}\right)^{e^{\kappa\,t}}
$$
 with $\kappa=2^{6}(N+1)^2\rho T$.
\end{itemize}
\end{theorem}

\begin{remarks}\label{remark1.LowBd}
\begin{enumerate}

\item The lower bounds established with the norm $\|\cdot\|_0$ imply
  certain lower bounds in terms of the norm $\|\cdot\|_2$. In fact, on
  one hand, it is obvious that $\|F_t-M\|_2\ge \|F_t-M\|_0$. On the
  other hand, for the standard case
  $F_0, M\in {\mathcal B}^{+}_{1,0,1}({\bRN})$, applying the
  inequalities \eqref{(5.8)} and $\log y\le \sqrt{y}$\,\,($y\ge 1$) we
  have
$$\|F_0-M\|_0\ge \Big(\fr{1}{4(N+1)}\|F_0-M\|_2\Big)^2.$$
Then, for the general case $F_0,M\in{\mathcal
    B}^{+}_{\rho,u,T}({\bRN})$, we use part (II) of Proposition
\ref{(prop1.1)} (normalization) to deduce
$$\left\|F_0-M\right\|_0\ge \rho\Big(\fr{1}{4(N+1)}\cdot\fr{1}{
C_{1/\rho,|u|/\sqrt{T},1/T}}\|F_0-M\|_2\Big)^2.$$

\item To our knowledge, Theorem \ref{(theo1.2)} is perhaps the first
  result concerning the lower bounds on the relaxation rate for the
  hard potentials.  Of course --and in spite of the fact that the
  assumptions of Theorem \ref{(theo1.2)} are weaker than those of
  Theorem \ref{(theo1.1)}--, these lower bounds are very rough as
  compared with the corresponding upper bounds in Theorem
  \ref{(theo1.1)}. The particular formula in these lower bounds come
  from limitations of the method we adopted. We conjecture that
  under the same assumptions on the initial data (i.e. assuming only
  that $F_0$ have finite mass, momentum and energy), the lower bounds
  have the same form ${\rm cst}.e^{-{\rm cst}. t}$ as the upper
  bounds. This may be investigated in the future.
\end{enumerate}

\end{remarks}

Now let us state an important corollary of Theorem \ref{(theo1.1)} and
Theorem \ref{(theo1.2)}, which gives a positive answer (see the part
(iii) below), for hard potentials, to the question of \emph{eternal
  solutions} raised in \cite[Chapter~1,
subsection~2.9]{Villani-handbook} (see also \cite{MR2946633}).

\begin{corollary}\label{cor:11+12}
Under the same assumptions on $N,\gm $ and $B(z,\sg)$ as in
  Theorem \ref{(theo1.1)}, let $F_0\in{\mathcal
    B}^{+}_{\rho,u,T}({\bRN})$ with $\rho>0$, $u\in{\bRN}$ and $T>0$,
  and let $M\in {\mathcal B}^{+}_{\rho,u,T}({\bRN})$ be the
  Maxwellian. Then we have:
\begin{itemize}
 \item[(i)] Let $F_t\in{\mathcal B}^{+}_{\rho,u,T}({\bRN})$ be the unique
  conservative measure solution of equation~\eqref{(B)} on
  $[0,\infty)$ with the initial datum $F_0$. If $F_0\neq M$, then
  $F_t\neq M$ for all $t\ge 0$. In other words, $F_t$ can not arrive
  at equilibrium state in finite time unless $F_0$ is an equilibrium.

\item[(ii)] Let $F_t\in{\mathcal B}^{+}_{\rho,u,T}({\bRN})$ be a
  conservative backward measure strong solution of
  equation~\eqref{(B)} on an interval $(-t_{\infty}, 0]$ for some
  $0<t_{\infty}\le \infty$, i.e.
$$\frac{\, {\rm d}}{\, {\rm d}t}F_t=Q(F_t,F_t),\quad t\in (-t_{\infty}, 0]. $$

Then if $F_0\neq M$, then $(-t_{\infty}, 0]$ must be bounded, and if
$F_0=M$, then $F_t\equiv M$ on $(-t_{\infty}, 0]$. \, In particular we
have

\item[(iii)][Eternal measure solutions are stationary.] If a
  conservative measure strong solution $F_t$ of equation \eqref{(B)}
  in ${\mathcal B}^{+}_{\rho,u,T}({\bRN})$ is \emph{eternal},
  i.e. defined for all $t \in \mathbb R$, then it has to be stationary and
  $F_t = M$ for all $t \in \mathbb R$.
\end{itemize}
\end{corollary}

The proof of this Corollary is easy and we would like to present it here.

\begin{proof}[Proof of Corollary~\ref{cor:11+12}]
  Part (i) is follows simply from the lower bound in Theorem
  \ref{(theo1.2)}.  Part (iii) follows from part (ii). In fact let
  $F_t$ be an eternal solution of equation \eqref{(B)} as defined in
  the part (iii) of the statement. Then $F_t$ is also a backward
  measure strong solution of equation~\eqref{(B)} on the unbounded
  time-interval $(-\infty, 0]$. By part (ii) we conclude that $F_0=M$ and
  thus $F_t\equiv M$ on $(-\infty,0]$. Then by the uniqueness of
  forward solutions we conclude that $F_t=M$ for all $t\in \mathbb R$.

  To prove part (ii), we use the existence and the uniqueness theorem
  of conservative measure strong solutions (see Theorem
  \ref{(theo1.0)}) to extend the backward solution $F_t$ to the whole
  interval $(-t_{\infty},\infty)$. Fix any $\tau\in (-t_{\infty},
  0)$.  Then $t\mapsto F_{\tau+t}$ is a conservative measure strong
  solution of equation~\eqref{(B)} on $[0,\infty)$ with the initial
  datum $F_{\tau}$.  By using the upper bound of the convergence rate
  in Theorem \ref{(theo1.1)} (see also \eqref{(4.36)}), together with
  the conservation of mass, momentum, and energy we have (with
  $\ld=S_{b,\gm}(\rho,u,T)$)
\begin{equation}\label{(5.11)}
  \forall \, t\ge 0, \quad \left\|F_{\tau+t}-M\right\|_0\le Ce^{-\ld t},
\end{equation}
where $C>0$ only depends on $N,\gm, b,\rho, u,T$.  Taking $t=-\tau$
gives
$$
\left\|F_0-M\right\|_0\le Ce^{\ld\tau}.
$$
Thus if $\left\|F_0-M\right\|_0>0$, then
$$
-\tau\le \frac{1}{\ld}
\log\left(\frac{C}{\left\|F_0-M\right\|_0}\right)<\infty.
$$
Letting $\tau\to -t_{\infty}$ leads to
$$
t_{\infty}\le \frac{1}{\ld}
\log\left(\frac{C}{\left\|F_0-M\right\|_0}\right)<\infty.
$$
Next, applying Theorem \ref{(theo1.2)}, we have for all $t\ge 0$
\begin{eqnarray}\label{(5.12)}
  \left\|F_{\tau+t}-M\right\|_0 \ge & (4\rho)^{1-\alpha}
\left\|F_{\tau}-M\right\|_0^{\alpha}e^{-\beta\,t^{\frac{2}{2-\gm}}}
&  ({\rm if} \,\,\, 0<\gm<2)\\ \label{(5.13)}
\left\|F_{\tau+t}-M\right\|_0 \ge & 4\rho\left (\frac{\left\|F_{\tau}-M\right\|_0}{4\rho}\right)^{e^{\kappa\,t}}
& ({\rm if} \,\,\, \gm=2).
\end{eqnarray}
Now suppose $\left\|F_0-M\right\|_0=0$. Then taking $t=-\tau$ so that
$\left\|F_{\tau+t}-M\right\|_0=0$ we obtain from \eqref{(5.12)},
\eqref{(5.13)} that $\left\|F_{\tau}-M\right\|_0=0$. Since
$\tau\in(-t_{\infty},0)$ is arbitrary, this shows that $F_{t}\equiv
M$ on $(-t_{\infty},0]$
 and concludes the proof.
\end{proof}

The third main result is concerned with the global-in-time stability
of measure strong solutions.

\begin{theorem}\label{(theo1.3)} Let $N,\gm $ and
  $B(z,\sg)$ satisfy the same assumptions in Theorem \ref{(theo1.1)}.
Let $\rho_0>0$, $u_0 \in{\bRN}$, $T_0>0$ and let $M\in{\mathcal
  B}^{+}_{\rho,u,T}({\bRN})$ be the Maxwellian.

  Then for any conservative measure strong solutions $F_t$, $G_t$ of
  equation \eqref{(B)} with $F_0\in{\mathcal
    B}^{+}_{\rho_0,u_0,T_0}({\bRN})$, there are explicitable constants
  $\eta\in (0,1)$, $C \in (0,\infty)$ only depending on $N$, $\gm$, $b$,
  $\rho_0$, $u_0$, $T_0$, such that
$$
\sup_{t\ge 0}\left\|F_t-G_t\right\|_2\le \wt{\Psi}_{F_0}\left(\left\|F_0-G_0\right\|_2\right)
$$
where
$$
\forall \,  r\ge 0, \quad \wt{\Psi}_{F_0}(r) := C\Big(r+\left[\Psi_{F_0}(r)\right]^{\eta}\Big)
$$
with $\Psi_{F_0}(r)$ defined in \eqref{(1.20)}.
\end{theorem}

\subsection{Previous results and references}

Apart from the paper \cite{Mcmp} already mentioned concerning the
sharp rate of relaxation for $L^1$ solutions in the case of hard
spheres or hard potentials with cutoff, let us mention the many
previous works that developed quantitative estimates on the rate of
convergence
\cite{CC92,MR1263387,CC94,MR1725612,MR1991033,MR2546739,MR2546739,MR2546739}. Let
us also mention the recent work \cite{GMM} obtaining sharp rates of
relaxation for $L^1_v L^\infty_x$ solutions in the spatially
inhomogeneous case in the torus.

\subsection{Strategy and plan of the paper}
\label{sec:strategy-plan-paper}

The rest of the paper is organized as follows: In Section 2 we give an
integral representation for the one-step iterated collision operator
$(f,g,h)\mapsto Q^{+}(f, Q^{+}(g,h))$ and prove an $L^p$ gain of
integrability for this operator. This is a generalization of
Abrahamsson's result \cite{MR1697495} which is concerned with $N=3$
and $\gm=1$.  In order to obtain the required regularities of such
iterated collision operators, the technical difficulty is to deal with
small values of $\gm$. In that case one needs multi-step iteration of
$Q^+$. In Section 3 we use iteratively the previous multi-step
estimates on $Q^+$ to give a series of positive decompositions
$f_t=f^n_t+h^{n}_t$ for $t\in [t_0,\infty)$ with $t_0>0$, for an $L^1$
mild solution $f_t$. In this decomposition the $f^n_t$ are bounded (in
$L^{\infty}({\bRN})$) and regular (they belong at least to
$H^1({\bRN})$ for instance) when $n$ is large enough; whereas
$h^{n}_t$ decays in $L^1$ norm exponentially fast as $t\to +\infty$.  By
approximation we then extend such positive decompositions to the
measure strong solutions $F_t$. In Section 4 we first use the results
of \cite{Mcmp} and those obtained in Section 3 to prove Theorem
\ref{(theo1.1)} for $L^1$ mild solutions, and then we use
approximation by $L^1$ mild solutions to complete the proof of Theorem
\ref{(theo1.1)} for measure solutions. The proof of Theorem
\ref{(theo1.2)} is given in Section 5. In Section 6 we prove Theorem
\ref{(theo1.3)} which is an application of Theorem \ref{(theo1.0)} and
Theorem \ref{(theo1.1)}.

Throughout this paper, unless otherwise stated, we always assume that
$N\ge 2$ as already indicated in equation \eqref{(B)}.

\section{$L^p$-estimates of the iterated gain term}
\label{sec2}

We introduce the weighted Lebesgue spaces $L^p_s({\bRN})$ for $1\le
p\le \infty, 0\le s<\infty$ as:
\begin{equation*}
  \begin{cases} \displaystyle
    f\in L^p_s({\bRN})
    \Longleftrightarrow \|f\|_{L^p_s}
    =\left(\int_{{\bRN}}\la v\ra^{ps}|f(v)|^p\, {\rm
        d}v\right)^{1/p}<\infty,\quad 1\le p<\infty \vspace{0.2cm} \\
    \displaystyle
    f\in L^{\infty}_s({\bRN}) \Longleftrightarrow \|f\|_{L^{\infty}_s}
    =\sup_{v\in{\bRN}}\la v\ra^{s}|f(v)|<\infty,\quad
    p=\infty.
\end{cases}
\end{equation*}
In the case $s=0$, we denote $L^p_0({\mathbb R}^N)=L^p({\mathbb R}^N)$
as usual.

We shall use the following formula of change of variables. For any
${\bf n}\in {\bSN}$ and $\psi$ nonnegative measurable on ${\bSN}$:
\begin{equation}\label{(2.0)}
\int_{{\bSN}}\psi(\sg)\, {\rm d}\sg=\int_{-1}^{1}(1-t^2)^{(N-3)/2}
\left( \int_{{\mathbb S}^{N-2}({\bf n})}\psi\left(t{\bf
    n}+\sqrt{1-t^2}\,\og\right)\, {\rm d}^{\bot}\og \right) \, {\rm d}t
\end{equation}
where ${\mathbb S}^{N-2}({\bf n})=\{\og\in{\bSN}\,|\,\, \og\,
\bot\,{\bf n}\}$ and ${\rm d}^{\bot}\og$ denotes the sphere measure
element of ${\mathbb S}^{N-2}({\bf n})$.

For convenience we rewrite \eqref{(2.0)} as follows:
\begin{equation}\label{(2.1)}\int_{{\bSN}}\psi(\sg)\, {\rm d}\sg=\int_{{\mathbb R}}\zeta(t)
\left( \int_{{\mathbb S}^{N-2}({\bf n})}\psi(\sg_{{\bf n}}(t,\og))\,
  {\rm d}^{\bot}\og \right) \, {\rm d}t
\end{equation}
where
\begin{equation}\label{(2.2)}
  \forall \, t\in \mathbb R, \quad \zeta(t):=(1-t^2)^{\frac{N-3}{2}}{\bf 1}_{(-1,1)}(t)
\end{equation}
\begin{equation}\label{(2.3)}
\sg_{{\bf n}}(t,\og):=
\left\{\begin{array}{ll}-{\bf n} \quad & {\rm if} \quad t\le
    -1\vspace{0.2cm} \\
\,\,t{\bf n}+\sqrt{1-t^2}\,\og\quad
& {\rm if} \quad t \in (-1,1) \vspace{0.2cm} \\
\,\,\,\, {\bf n} \quad & {\rm if} \quad t\ge  1.
\end{array}\right.
\end{equation}

\begin{lemma}\label{(lem2.1)}  Suppose $N\ge 3$ and let $B(z,\sg)$ be given by \eqref{(1.4)} with  $b$ satisfying \eqref{(b-upperbd)}. Let $ f\in
  L^1_{\gm}({\bRN})$ and $g, h \in L^1_{2\gm}({\bRN})$.
Then $Q^{+}(f, Q^{+}(g,h))\in L^1({\bRN})$ with the estimate
  \begin{equation}\label{(2.11)}
    \left\|Q^{+}\left(f, Q^{+}(g,h)\right)\right\|_{L^1}\le A_0^2\|f\|_{L^1_{\gm}}
    \|g\|_{L^1_{2\gm}}\|h\|_{L^1_{2\gm}}.
  \end{equation}

  Moreover we have the following representation: for almost every
  $v\in{\bRN}$
  \begin{equation}\label{(2.4)}
    Q^{+}\left(f, Q^{+}(g,h)\right)(v) =\inttt_{{\bRRRN}}K_{B}(v,v_*,w,w_*)f(v_*)g(w)h(w_*)
    \, {\rm d}v_*\, {\rm d}w \, {\rm d}w_*
  \end{equation}
  where $K_B: {\mathbb R}^{4N}\to [0,\infty)$ is defined by
  \begin{equation}\label{(2.5)}
    K_{B}(v,v_*,w,w_*) :=
    \begin{cases} \displaystyle
      \frac{2^N}{|v-v_*||w-w_*|}\zeta\left({\bf n}\cdot
        \frac{2v-(w+w_*)}{|w-w_*|} \right ) \times \vspace{0.2cm} \\ \displaystyle
      \qquad \qquad  \int_{{\mathbb
          S}^{N-2}({\bf n})} \frac{B(w-w_*,\sg)\,B(
        w'-v_*,\sg')}{|w'-v_*|^{N-2}} \, {\rm d}^{\bot}\og  \vspace{0.2cm}
      \\ \displaystyle \qquad \qquad \qquad \qquad
      \mbox{if} \quad |v-v_*||w-w_*|\neq 0,
      \vspace{0.2cm}   \\ \displaystyle
\,\, \, 0 \quad\qquad \qquad \qquad\mbox{if} \quad
|w-w_*||v-v_*|= 0,
\end{cases}
\end{equation}
where the function $\zeta$ is given by \eqref{(2.2)}, and
 \begin{equation}\label{(2.6)}
   {\bf n}:=\frac{v-v_*}{|v-v_*|},\quad
   w'=\frac{w+w_*}{2}+\frac{|w-w_*|}{2}\sg,\quad \sg'=
   \frac{2v-v_*-w'}{|2v-v_*-w'|}
 \end{equation}
 with
\begin{equation}\label{(2.7)}
  \sg = \sg(\og)=\sg_{{\bf
      n}}(t,\og)
  \quad \mbox{at} \quad t={\bf
    n}\cdot\left( \frac{2v-(w+w_*)}{|w-w_*|} \right).
\end{equation}
\end{lemma}

\begin{remark}
  Inserting the formula \eqref{(Grad)} of $B(z,\sg)$ into
  \eqref{(2.5)} gives the more detailed expression of $K_B$:
  \begin{multline}
    \label{(2.9)}
    K_{B}(v,v_*,w,w_*) = \\
    \frac{2^N}{|w-w_*|^{1-\gm}|v-v_*|}\zeta\left( {\bf
        n}\cdot\frac{2v-(w+w_*)}{|w-w_*|} \right) \int_{{\mathbb
        S}^{N-2}({\bf n})}
    \frac{b\left(\frac{w-w_*}{|w-w_*|}\cdot\sg\right)\,b\left(
        \frac{w'-v_*}{|w'-v_*|}\cdot\sg'\right)}{|w'-v_*|^{N-2-\gm}} \,
    {\rm d}^{\bot}\og
  \end{multline}
  for $|w-w_*||v-v_*|\neq 0$.  Also we note that
  \begin{equation}
    K_{B}(v,v_*,w,w_*)>0 \Longrightarrow
    |v-v_*||w-w_*|\neq 0 \quad \mbox{and} \quad
    \left|{\bf n}\cdot\frac{2v-(w+w_*)}{|w-w_*|}\right|<1
  \end{equation}
  which implies, by using the formula \eqref{(2.7)} for $\sigma_{{\bf n}}(t,\omega)$
  in this case and the value of $t$, that $(v-w')\cdot(v-v_*)=0$ and
  therefore by Pythagoras' theorem
  \begin{eqnarray} \displaystyle
    &&\label{(2.10)}
    |w'-v_*|=|w'-v + v -v_*|=\sqrt{|v-v_*|^2+|v-w'|^2}, \vspace{0.2cm} \\
    \displaystyle
    &&\label{(2.10*)} |w'-v_*|\ge |v-v_*|.
  \end{eqnarray}
\end{remark}

\vskip2mm

\begin{proof}[Proof of Lemma~\ref{(lem2.1)}] We shall use the
  following formula of change of variables (see \cite[Chapter~1,
  Sections~4.5-4.6]{Villani-handbook}): For every nonnegative measurable
  function $\psi$ on ${\mathbb R}^{4N}$, one has
  \begin{multline}\label{(2-change)}
    \inttt_{{\bRRSN}}B(v-v_*,\sg) \psi(v',v_*',
  v,v_*)
  \, {\rm d}\sg \, {\rm d}v_* \, {\rm d}v\\
  =\inttt_{{\bRRSN}}B(v-v_*,\sg) \psi(v,v_*,v',v_*') \, {\rm d}\sg
  \, {\rm d}v_* \, {\rm d}v.
\end{multline}

We can assume that $f,g,h$ are all nonnegative.  Applying
\eqref{(1-Qcommut)}, \eqref{(2-change)}, and recalling definition of
$L_B[\vp]$ (see \eqref{(1-LB)}) we have, for any nonnegative
measurable function $\vp$ on ${\bRN}$,
\begin{equation*}
\begin{cases}\displaystyle
  \int_{{\bRN}}Q^{+}(f,
Q^{+}(g,h))(v)\vp(v) \, {\rm d}v =\int_{{\bRN}} f(v_*) \left( \int_{{\bRN}}
Q^{+}(g,h)(w)L_B[\vp](w,v_*)\, {\rm d}w \right) \, {\rm
  d}v_*, \vspace{0.2cm} \\
\displaystyle
\int_{{\bRN}} Q^{+}(g,h)(w)L_B[\vp](w,v_*) \, {\rm
  d}w=\intt_{{\bRRN}}L_B\left[L_B[\vp](\cdot, v_*)\right](w,w_*)g(w)h(w_*) \,
{\rm d}w \, {\rm d}w_*,
\end{cases}
\end{equation*}
and so
\begin{multline} \label{(2.12)}
\int_{{\bRN}}Q^{+}(f, Q^{+}(g,h))(v)\vp(v) \, {\rm d}v \\
=\inttt_{{\bRRRN}}L_B\left[L_B[\vp](\cdot, v_*)\right](w,w_*)f(v_*)g(w)h(w_*) \, {\rm d}v_* \, {\rm d}w
\, {\rm d}w_*.
\end{multline}
Taking $\vp =1$ and using the inequalities
\begin{equation}
\label{(2.12WV)}
\begin{cases}
|w-w_*|\le \la w\ra\la
w_*\ra,\vspace{0.2cm}\\ \displaystyle
 |w'-v_*|\le \frac{|w+w_*|}{2}+\frac{|w-w_*|}{2}+|v_*|
\le \la w\ra\la w_*\ra\la v_*\ra,
\end{cases}
\end{equation}
we obtain
\[
L_B\left[L_B[1](\cdot, v_*)\right](w,w_*)
 \le A_0^2\la v_*\ra^{\gm}\la w\ra^{2\gm}\la w_*\ra^{2\gm}
\]
and it follows from \eqref{(2.12)} that
\[
0\le Q^{+}\left(f, Q^{+}(g,h)\right)\in L^1({\bRN})
\]
and so \eqref{(2.11)} holds true.

Comparing \eqref{(2.12)} with \eqref{(2.4)}, it appears that in order
to prove the integral representation \eqref{(2.4)} we only need to
prove that the following identity
\begin{equation}\label{(2.14)}
\forall\, 0\le
\vp\in C_c({\bRN}), \quad L_B\left[L_B[\vp](\cdot,
v_*)\right](w,w_*)=\int_{{\bRN}}K_B(v,v_*,w,w_*)\vp(v)\, {\rm d}v
\end{equation}
holds for all $w,w_*,v_*\in{\bRN}$ satisfying
\begin{equation}\label{(2.13)}
0\neq \left|\frac{w+w_*}{2}-v_*\right|\neq \frac{|w-w_*|}{2}\neq 0.
\end{equation}

Observe that
\begin{equation*}
  \begin{cases} \displaystyle
    L_B\left[L_B[\vp](\cdot,v_*)\right](w,w_*)=
    L_B\left[L_B[\vp(v_*+\cdot)](\cdot,0)\right](w-v_*,w_*-v_*),\vspace{0.2cm} \\
    \displaystyle
    K_B(v,v_*,w,w_*)=K_B(v-v_*,0,w-v_*,w_*-v_*).
  \end{cases}
\end{equation*}
By replacing respectively $\vp(v_*+\cdot)$, $w-v_*$ and $w_*-v_*$ with
$\vp(\cdot)$, $w$ and $w_*$, we can assume without loss of generality
that $v_*=0$. That is, in order to prove \eqref{(2.14)}, we only need to
prove
\begin{equation}\label{(2.15)}
  \forall\, 0\le \vp\in C_c({\bRN}), \quad
  L_B\left[L_B[\vp](\cdot,0)\right](w,w_*)=\int_{{\bRN}}K_B(v,0,w,w_*)\vp(v)\,
  {\rm d}v.
\end{equation}

To do this we first assume that $b \in C([-1,1])$ so that the use of
the Dirac distribution is fully justified.  We compute
\begin{equation}\label{(2.16)}
  L_B\left[L_B[\vp](\cdot, 0)\right](w,w_*) =\int_{{\bSN}}
  B(w-w_*,\sg)\int_{{\bSN}}B(w',\og)\vp\left(\frac{w'}{2}
    +\frac{|w'|}{2}\og\right)\, {\rm d}\og \, {\rm d}\sg,
\end{equation}
and, by using \eqref{(2.6)} and \eqref{(2.13)} with $v_*=0$, we have
$w'\neq 0$ for all $\sg\in{\bSN}$. Let $\dt = \dt(x)$ be the
one-dimensional Dirac distribution. Applying the integral
representation
\[
\forall \, \psi\in C((0,\infty)),\ \rho>0, \quad
\psi(\rho)=\frac{2}{\rho^{N-2}}\int_{0}^{\infty}
r^{N-1}\psi(r)\dt(\rho^2-r^2)\, {\rm d}r
\]
to the function
\[
\psi(\rho):= \vp\left(\frac{w'}{2} +\rho\og\right)
\]
and then taking $\rho=|w'|/2$ and changing variable $r\og=z$ we
have
\begin{multline*}
\int_{{\bSN}}B(w',\og)\vp\left(\frac{w'}{2} +\frac{|w'|}{2}\og\right) \, {\rm
  d}\og \\ =2\left|\frac{w'}{2}\right|^{-(N-2)} \int_{{\bRN}}
B\left(w',\frac{z}{|z|}\right)\vp\left(\frac{w'}{2}+ z\right) \dt\left(\frac{|w'|^2}{4}-|z|^2\right)
\, {\rm d}z.
\end{multline*}
We then use the change of variable $z=v-w'/2$ and Fubini  theorem:
\begin{multline*}\label{(2,17)}
L_B\left[L_B[\vp](\cdot, 0)\right](w,w_*) = \\
\int_{{\bRN}}\vp(v)\left(\int_{{\bSN}}
  2 \left|\frac{w'}{2}\right|^{-(N-2)} B(w-w_*,\sg)
  B\left(w',\frac{v-w'/2}{|v-w'/2|}\right)
  \dt\left(\left|\frac{w'}{2}\right|^2-\left|v-\frac{w'}{2}\right|^2\right)
  \, {\rm d}\sg\right) \, {\rm d}v.
\end{multline*}

We now assume that $v\neq 0$ and ${\bf n}=v/|v|$ satisfy
\[
\left|{\bf n}\cdot\frac{2v-(w+w_*)}{|w-w_*|}\right|\neq 1,
\quad \left|v-\frac{w+w_*}{4}\right|\neq \left|\frac{w-w_*}{4}\right|.
\]
We deduce that $|v-w'/2|>0$ for all $\sg\in{\bSN}$ and we
compute using \eqref{(2.1)}-\eqref{(2.2)}-\eqref{(2.3)} with ${\bf
  n}=v/|v|$ that
\begin{multline*}
\int_{{\bSN}}2 \left|\frac{w'}{2}\right|^{-(N-2)}
B(w-w_*,\sg) B\left(w',\frac{v-w'/2}{|v - w'/2|}\right)
\dt \left(\frac{|w'|^2}{4}-\left|v-\frac{w'}{2}\right|^2\right)\, {\rm d}\sg\\
=\int_{{\mathbb R}}\zeta(t)\,\left(\int_{{\mathbb S}^{N-2}({\bf
      n})} 2 \left|\frac{w'}{2}\right|^{-(N-2)} B(w-w_*,\sg)
  B\left(w',\frac{v-w'/2}{|v- w'/2|}\right)\Bigg|_{\sg=\sg_{{\bf n}}(t,\og)}\, {\rm d}^{\bot}\og\right)\\
\times \dt\left(|v|\frac{|w-w_*|}{2}t -v\cdot\left(v- \frac{w+w_*}{2}\right)\right)\, {\rm d} t\\
=\frac{2^N}{|v||w-w_*|}\zeta\left(\frac{v\cdot(
  2v-(w+w_*))}{|v||w-w_*|}\right)\int_{{\mathbb
    S}^{N-2}({\bf n})} |w'|^{-(N-2)} B(w-w_*,\sg)
B\left(w',\frac{2v-w'}{|2v-w'|}\right)\, {\rm d}^{\bot}\og
\end{multline*}
where $\sg$ in the last line is given by \eqref{(2.7)}. Thus we obtain
\begin{multline} \label{(2.18)} \qquad L_B[L_B[\vp](\cdot,
0)](w,w_*)
=\frac{2^N}{|w-w_*|}\int_{{\bRN}} \frac{\vp(v)}{|v|}\zeta\left({\bf
  n}\cdot \frac{2v-(w+w_*)}{|w-w_*|}\right) \\ \times\int_{{\mathbb
  S}^{N-2}({\bf n})}|w'|^{-(N-2)} B(w-w_*,\sg)
B\left(w',\frac{2v-w'}{|2v-w'|}\right)\, {\rm d}^{\bot}\og \, {\rm
  d}v.
\end{multline}
This proves \eqref{(2.15)}.

Finally, thanks to $N\ge 3$, we use standard approximation arguments
in order to prove that \eqref{(2.15)} still holds without the
continuity assumption on the function $b$. We skip these classical
calculations.
\end{proof}

\begin{lemma}\label{(lem2.2)}  Suppose $N\ge 3$ and let
  $B(z,\sg)$ be defined in \eqref{(1.4)} with  $b$
  satisfying \eqref{(b-upperbd)}. Let $1\le p,q\le \infty$ satisfy
  $1/p+1/q=1$.

Then, in the case where we have
\begin{equation} \label{(2.19)}
0<\gm < N-2,\quad \frac{N-1}{N-1-\gm}\le p< \frac{N}{N-1-\gm},
\end{equation}
the following estimate holds
\begin{equation} \label{(2.20)}
\left(\int_{{\bRN}}[K_{B}(v,v_*,w,w_*)]^p\, {\rm d}v\right)^{1/p} \le
C_{p}\|b\|_{L^{\infty}}^2|w-w_*|^{2\gm-N/q}.
\end{equation}

Second, in the case where we have
\begin{equation} \label{(2.21)}
\gm \ge N-2,\quad 1\le p<N,
\end{equation}
then the following estimate holds:
\begin{equation} \label{(2.22)}\left(\int_{{\bRN}}[K_B(v,v_*,w,w_*)]^{p} \, {\rm d}v\right)^{1/p}\le
C_{p}\|b\|_{L^{\infty}}^2\langle{v_*}\rangle^{2\gm-N/q} \langle{w}\rangle^{2\gm-N/q}\langle{w_*}\rangle^{2\gm
-N/q}. \end{equation}
The constants $C_p$ only depend on $N,\gm, p$.
\end{lemma}

\begin{proof} By replacing the function $b$ with
  $b/\|b\|_{L^{\infty}}$ we can assume for notation convenience that
  $\|b\|_{L^{\infty}}=1$. Fix $w,w_*, v_*\in{\bRN}$. To prove the
  lemma we may assume that $w\neq w_*$.  Recall that $N\ge 3$ implies
  $\zeta(t)\le {\bf 1}_{(-1,1)}(t)$. Then from
  \eqref{(2.9)}-\eqref{(2.10*)} we have
  \begin{multline} \label{(2.24)} K_B(v,v_*,w,w_*)\\
    \le \frac{2^N|{\mathbb S}^{N-2}|
    }{|w-w_*|^{1-\gm}}\cdot\frac{1}{|v-v_*|^{N-1-\gm}}{\bf
      1}_{(-1,1)}\left({\bf n}\cdot
      \frac{2v-(w+w_*)}{|w-w_*|}\right)\quad {\rm for}\quad
    0<\gm<N-2,
\end{multline}
whereas for $\gm\ge N-2$ we have
\begin{multline} \label{(2.25)}
K_B(v,v_*,w,w_*)\\
\le  2^{N}|{\mathbb S}^{N-2}|\la v_*\ra
^{\gm+2-N}\la w\ra^{2\gm+1-N} \la
w_*\ra^{2\gm+1-N}\frac{1}{|v-v_*|}{\bf 1}_{(-1,1)}\left({\bf n}\cdot
  \frac{2v-(w+w_*)}{|w-w_*|}\right)
\end{multline}
where we used $N-2\ge 1$ and the inequalities in \eqref{(2.12WV)}.

Let us define
$$
J_{\beta}(w,w_*)=\int_{{\bRN}}\frac{1}{|v-v_*|^{\beta}}{\bf
  1}_{(-1,1)}\left({\bf n}\cdot
  \frac{2v-(w+w_*)}{|w-w_*|}\right) \, {\rm d}v.
$$
We need to prove that
\begin{eqnarray}\label{(2.27)}&& J_{\beta}(w,w_*)\le
|{\bSN}|(\la v_*\ra\la w\ra\la w_*\ra)^{N-1-\beta}|w-w_*|\quad
\mbox{when} \quad 0<\beta<N-1,\\
&&
\label{(2.28)} J_{\beta}(w,w_*)
 \le \frac{|{\bSN}|}{(N-\beta)}
|w-w_*|^{N-\beta} \quad \mbox{when} \quad N-1\le \beta<N.
\end{eqnarray}
To do this we use the change of variable
\[
v=v_*+\frac{|w-w_*|}{2} r\sg, \quad
{\rm d}v=\left|\frac{w-w_*}{2}\right|^{N} r^{N-1}\, {\rm d}r\, {\rm d}\sg
\]
to compute
\begin{equation} \label{(2.29)}J_{\beta}(w,w_*)
=\left|\frac{w-w_*}{2}\right|^{N-\beta}\int_{{\bSN}}I(u\cdot\sg)
\, {\rm d}\sg \end{equation}
where
$$
I(u\cdot \sg)=\int_{0}^{\infty}r^{N-1-\beta}{\bf 1}_{\{|r+u\cdot \sg|<1\}}
\, {\rm d}r,\quad u=\frac{2v_*-(w+w_*)}{|w-w_*|}.
$$

Let us now estimate $I(u\cdot \sg)$ uniformly in $\sg$.  If $u\cdot
\sg\ge 1$, then $I(u\cdot \sg)=0$.  Suppose $u\cdot \sg< 1$. If
$0<\beta<N-1$, then
\begin{equation*}
I(u\cdot \sg) \le
  2(1+|u|)^{N-1-\beta}\le 2
  \left(\frac{|\frac{w-w_*}{2}|+|\frac{w+w_*}{2}|+|v_*|}{|\frac{w-w_*}{2}|}\right)^{N-1-\beta}
\end{equation*}
which together with \eqref{(2.29)} and the third inequality in
\eqref{(2.12WV)} gives \eqref{(2.27)}. If $N-1\le \beta<N$, then
$0<N-\beta\le 1$ so that
 \begin{equation*}
 I(u\cdot \sg)
\le\frac{2^{N-\beta}}{N-\beta}
\end{equation*}
which implies \eqref{(2.28)}.

Now suppose \eqref{(2.19)} is satisfied. Then using \eqref{(2.24)} and
\eqref{(2.27)} with
\[
N-1\le \beta=(N-1-\gm)p<N
\]
we obtain \eqref{(2.20)}:
\begin{multline*}
  \left(\int_{{\bRN}}[K_{B}(v,v_*,w,w_*)]^p\, {\rm d}v\right)^{1/p}
  \le C_p{|w-w_*|^{\gm-1}}\left(J_{(N-1-\gm)p}(w,w_*)
  \right)^{1/p}\\
\le C_{p}|w-w_*|^{\gm-1}|w-w_*|^{\frac{N-p(N-1-\gm)}{p}}
  =C_{p}|w-w_*|^{2\gm-N/q}.
\end{multline*}

Next suppose \eqref{(2.21)} is satisfied.  If $N-1\le p<N$, then by
\eqref{(2.25)}-\eqref{(2.28)} with $\beta=p$ and using $|w-w_*|\le
\la w\ra\la w_*\ra$ we have
\begin{multline*}\left(\int_{{\bRN}}[K_{B}(v,v_*,w,w_*)]^p\, {\rm
      d}v\right)^{1/p}
\le C_{p}\langle v_*\rangle ^{\gm+2-N}\langle w\rangle^{2\gm+1-N}
\langle w_*\rangle ^{2\gm+1-N}
\left(J_{p}(w,w_*)\right)^{1/p}\\
\le C_{p}\langle v_*\rangle ^{\gm+2-N}\langle w\rangle^{2\gm-N/q} \langle
w_*\rangle ^{2\gm-N/q}.
\end{multline*}

Similarly if $1\le p<N-1$, then using \eqref{(2.25)}-\eqref{(2.27)}
with $\beta=p$ we have
$$\left(\int_{{\bRN}}[K_{B}(v,v_*,w,w_*)]^p\, {\rm d}v\right)^{1/p}\\
\le C_{p}\langle v_*\rangle ^{\gm-(N-1)/q}\langle w\rangle^{2\gm-N/q} \langle
w_*\rangle ^{2\gm-N/q}. $$
Since $\gm\ge N-2\ge 1$ and $1\le p<N$ imply
\[
\max\left\{\gm +2-N, \gm-\frac{N-1}{q}\right\}\le 2\gm-N/q,
\]
it follows that
\[
\max\left\{\la v_*\ra^{\gm +2-N},\,\la v_*\ra^{\gm-\frac{N-1}{q}}\right\}\le \la
v_*\ra^{2\gm-N/q}.
\]
This concludes the proof of \eqref{(2.22)}.
\end{proof}

\begin{lemma}\label{(lem2.3)}
  Let $K(v,v_*)$ be a measurable function on ${\bRRN}$ and let
$$
\forall \, v\in {\bRN}, \quad T(f)(v):=\int_{{\bRN}}K(v,v_*)f(v_*)\,
{\rm d}v_*.
$$
Assume that $1\le r<\infty,\, 0\le s<\infty$ and that there is
$0<A<\infty$ such that
\begin{equation} \label{(2.30)} \left(\int_{{\bRN}}|K(v,v_*)|^r \,
    {\rm d} v\right)^{\frac{1}{r}}\le A\la v_*\ra^{s}\qquad {\rm
    a.e.}\quad v_*\in {\bRN}.
\end{equation}
Then
\begin{equation}\label{(2.31)}
\forall \, f\in L^1_s({\bRN}), \quad \|T(f)\|_{L^r}\le A\|f\|_{L^1_s}.
\end{equation}
Furthermore let $1\le p,q\le \infty$
satisfy
\[
\frac{1}{p}=\frac{1}{q}+\frac{1}{r}-1
\]
and assume that there is $0< B<\infty$ such that
\begin{equation}\label{(2.32)}
\left(\int_{{\bRN}}(|K(v,v_*)|\la v_*\ra^{-s})^{r} \, {\rm d}v_*\right)^{\frac{1}{r}}
\le B\qquad {\rm a.e.} \quad v\in {\bRN}.
\end{equation}
Then
\begin{equation}\label{(2.33)}
\forall\,f\in L^q_s({\bRN}), \quad \|T(f)\|_{L^p}\le A^{\frac{r}{p}}
B^{1-\frac{r}{p}}\|f\|_{L^q_s}.
\end{equation}
\end{lemma}

\begin{proof}  Let us define
$$
\forall \, v\in {\bRN}, \quad T_s(f)(v):=\int_{{\bRN}}K(v,v_*)\la
v_*\ra^{-s}f(v_*)\, {\rm d}v_*.
$$
By Minkowski inequality and \eqref{(2.30)} we have
\begin{equation}\label{(2.34)}
\forall\, f\in L^1({\bRN}), \quad \left\|T_s(f)\right\|_{L^r}\le A\|f\|_{L^1}.
\end{equation}
Also by \eqref{(2.32)} we have
\begin{equation} \label{(2.35)}
  \forall\, f\in L^{r'}({\bRN}), \quad \left\|T_s(f)\right\|_{L^{\infty}}\le  B\|f\|_{L^{r'}}
\end{equation}
where $1\le r'\le \infty$ satisfies $1/r+1/{r'}=1$.

Let $\theta=1-r/p$. By assumption on $p,q,r$ we have $0\le \theta\le
1$ and
$$
\frac{1}{p}=\frac{1-\theta}{r}+\frac{\theta}{\infty},\quad
\frac{1}{q}=\frac{1-\theta}{1}+\frac{\theta}{r'}.
$$
So by \eqref{(2.34)}, \eqref{(2.35)} and Riesz-Thorin interpolation
theorem (see e.g. \cite[Chapter~5]{MR0304972}) we have
\begin{equation}\label{(2.36)}
\forall\,f\in L^q({\bRN}), \quad \left\|T_s(f)\right\|_{L^p}\le A^{1-\theta}B^{\theta}\|f\|_{L^q}.
\end{equation}
Now if we set $(f)_s(v_*)=\la v_*\ra^sf(v_*)$, then
\[
T(f)=T_s((f)_s), \quad \|(f)_s\|_{L^1}=\|f\|_{L^1_s}, \quad
\|(f)_s\|_{L^q}=\|f\|_{L^q_s}
\]
and thus \eqref{(2.31)}-\eqref{(2.33)} follow from \eqref{(2.34)}-\eqref{(2.36)}.
\end{proof}

In order to highlight structures of inequalities, we adopt the
following notional convention: {\it Functions $f,g,h$
  appeared below are arbitrary members in the classes
  indicated. Whenever the notation (for instance) $\|f\|_{L^p_s}$
  appears, it always means that $f\in L^p_s({\mathbb R}^N)$ with the
  norm $\|f\|_{L^p_s}$; and if $\|f\|_{L^p_s}$, $\|f\|_{L^q_k}$ appear
  simultaneously, it means that $f\in L^p_s({\mathbb R}^N)\cap
  L^q_k({\mathbb R}^N)$.}

\begin{lemma}\label{(lem2.4)} Let
  $0<\alpha<N, 1\le \alpha q<N,1<p\le \infty$ with $1/p+1/q=1$. Then
\[
\sup_{v\in{\bRN}}\int_{{\bRN}}\frac{|f(v_*)|}{|v-v_*|^{\alpha}} \,
{\rm d}v_* \le
2\left(\frac{\left|{\bSN}\right|}{N-\alpha q}\right)^{\frac{\alpha }{N}} \|f\|_{L^1}^{1-\frac{\alpha q
}{N}} \|f\|_{L^p}^{\frac{\alpha q }{N}}.
\]
\end{lemma}

\begin{proof} This follows from H\"{o}lder inequality and a
  minimizing argument.
\end{proof}

\begin{lemma}\label{(lem2.5)}   Suppose $N\ge 3$ and let
  $B(z,\sg)$ be defined in \eqref{(1.4)} with the condition
  \eqref{(b-upperbd)}. For any $w,w_*\in{\bRN}$ with $w\neq w_*$, let
\begin{equation} \label{(2.37)}
\forall \, v\in{\bRN}, \quad T_{w,w_*}(f)(v):=\int_{{\bRN}}K_{B}(v,v_*,w,w_*)f(v_*)\, {\rm d}v_*
\end{equation}
for nonnegative measurable or certain integrable functions $f$
as indicated below.

\begin{itemize}
\item[(i)] Suppose $0<\gm< N-2$.
Let $p_1=(N-1)/(N-1-\gm)$. Then
\begin{equation} \label{(2.38)}\|T_{w,w_*}(f)\|_{L^{p_1}}\le C_{p_1}\|b\|_{L^{\infty}}^2
\|f\|_{L^1}\la w\ra^{\gm}\la w_*\ra^{\gm}.\end{equation}
Also if $1< p, q<\infty$ satisfy $$\frac{1}{p}=
\frac{1}{q}+\frac{1}{p_1}-1$$ then
\begin{equation} \label{(2.39)}
  \left\|T_{w,w_*}(f)\right\|_{L^{p}}\le
  C_{p}\|b\|_{L^{\infty}}^2 \|f\|_{L^q_{1}} \frac{\la
    w\ra^{1-\frac{p_1}{p}}\la w_*\ra^{1-\frac{p_1}{p}}}{
    |w-w_*|^{1-\gm-\frac{1}{p}}}.
\end{equation} And if $1< p\le
\infty, 1\le q< N/(N-1-\gm)$ satisfy $1/p+1/q=1$, then
\begin{equation} \label{(2.40)}
\left\|T_{w,w_*}(f)\right\|_{L^{\infty}}\le \frac{C_{p}\|b\|_{L^{\infty}}^2}{|w-w_*|^{1-\gm}}
\|f\|_{L^1}^{1-\frac{N-1-\gm}
{N}q} \|f\|_{L^{p}}^{\frac{N-1-\gm}{N}q}.
\end{equation}

\item[(ii)] Suppose $\gm\ge  N-2$.  Let $1< p<N$, $1/p+1/q=1$.  Then
\begin{equation} \label{(2.41)}
\left\|T_{w,w_*}(f)\right\|_{L^p}\le
C_{p}\|b\|_{L^{\infty}}^2\|f\|_{L^1_{2\gm-N/q}}\la w\ra^{2\gm-N/q}\la
w_*\ra^{2\gm-N/q}.
\end{equation}
Furthermore if $N/(N-1)< p<N$, then
\begin{equation} \label{(2.42)}
\left\|T_{w,w_*}(f)\right\|_{L^{\infty}}\le C_{p}\|b\|_{L^{\infty}}^2
\|f\|_{L^1_{\gm+2-N}}^{1-\frac{q}{N}} \|f\|_{L^p_{\gm+2-N}}^{\frac{q}{N}} \la
w\ra^{2\gm+1-N}\la w_*\ra^{2\gm+1-N}.
\end{equation}
\end{itemize}

The constants $C_p<\infty$ only depend on $N,\gm, p$.
\end{lemma}

\begin{proof}
As in the proof of Lemma \ref{(lem2.2)} we can assume that $\|b\|_{L^{\infty}}=1$.
\smallskip

\noindent {\it Case~{\rm (i)}.} Suppose $0<\gm<N-2$.  By Lemma \ref{(lem2.2)}
we have
\begin{equation} \label{(2.43)}
\left(\int_{{\bRN}}[K_B(v,v_*,w,w_*)]^{p_1}\, {\rm d}v\right)^{1/{p_1}}\le
C_{p_1}|w-w_*|^{2\gm-N/{q_1}}
\end{equation}
where $q_1=(p_1)/(p_1-1)=(N-1)/\gm$.  Since
\[
0< 2\gm-N/{q_1}< \gm \quad \mbox{ and } \quad |w-w_*|\le \la w\ra\la
w_*\ra,
\]
\eqref{(2.38)} follows from \eqref{(2.31)} with $r=p_1, s=0$. Next
recalling \eqref{(2.10*)} and the second inequality in
\eqref{(2.12WV)} we see that that
\[
K_B(v,v_*,w,w_*)> 0 \quad \Longrightarrow \quad \sqrt{1+|v-v_*|^2} \le
\sqrt{2}\la w\ra\la w_*\ra \la v_*\ra
\]
so that
\begin{equation} \label{(2.44)}K_B(v,v_*,w,w_*)\la v_*\ra^{-1}\le
  \sqrt{2}\la w\ra\la w_*\ra
  \frac{K_B(v,v_*,w,w_*)}{\sqrt{1+|v-v_*|^2}}.
\end{equation}
This together with \eqref{(2.24)} and $(N-1-\gm)p_1=N-1$ gives
\begin{multline}\label{(2.45)}
\left(\int_{{\bRN}}(K_B(v,v_*,w,w_*)\la
  v_*\ra^{-1})^{p_1} \, {\rm d}v_*\right)^{1/{p_1}}\\
 \le C_{p_1}\frac{\la w\ra\la w_*\ra}{
  |w-w_*|^{1-\gm}}\left(\int_{{\bRN}}\frac{{\rm
      d}v_*}{(1+|v-v_*|^2)^{p_1/2} |v-v_*|^{N-1}}\right)^{1/{p_1}}
=C_{p_1}\frac{\la w\ra\la w_*\ra}{
  |w-w_*|^{1-\gm}}.
\end{multline}
Note that the above integral is finite since $p_1>1$.  If we set
$$A_{w,w_*}:=C_{p_1}
|w-w_*|^{2\gm-N/{q_1}},\quad B_{w,w_*}:=C_{p_1}\frac{\la w\ra\la
  w_*\ra}{ |w-w_*|^{1-\gm}}$$ then we see from \eqref{(2.43)} and
\eqref{(2.45)} that Lemma \ref{(lem2.3)} can be used for $T_{w,w_*}(f)$
with $r=p_1$ and $s=1$, and thus for all $f\in L^q_1({\bRN})$
$$
\left\|T_{w,w_*}(f)\right\|_{L^p} \le
\left(A_{w,w_*}\right)^{\frac{p_1}{p}}\left(B_{w,w_*}\right)^{1-\frac{p_1}{p}}\|f\|_{L^q_1}
=C_p\frac{\la w\ra^{1-\frac{p_1}{p}}\la w_*\ra^{1-\frac{p_1}{p}}}
{|w-w_*|^{1-\gm-\frac{1}{p}}}\|f\|_{L^q_1}
$$
where we have computed (using the definitions of $p_1, q_1$)
$$
\frac{p_1}{p}\left(2\gm-\frac{N}{q_1}\right)-\left(1-\frac{p_1}{p}\right)(1-\gm)
=\frac{1}{p}+\gm-1.
$$
This proves \eqref{(2.39)}. To prove \eqref{(2.40)} we use \eqref{(2.24)}
to get
$$\left|T_{w,w_*}(f)(v)\right|\le \frac{2^N|{\mathbb S}^{N-2}|}{|w-w_*|^{1-\gm}}
\int_{{\bRN}}\frac{|f(v_*)|\, {\rm d}v_*}{|v-v_*|^{N-1-\gm}}. $$
Since $1\le (N-1-\gm)q<N$, it follows from Lemma \ref{(lem2.4)} that
$$
\left\|T_{w,w_*}(f)\right\|_{L^{\infty}}\le \frac{C_{p}}{|w-w_*|^{1-\gm}}
\|f\|_{L^1} ^{1-\frac{N-1-\gm}
{N}q} \|f\|_{L^p} ^{\frac{N-1-\gm}{N}q}.
$$
\smallskip

\noindent {\it Case~{\rm (ii)}.} Suppose $\gm\ge N-2$. In this case we recall
the inequality \eqref{(2.22)}. Let $1< p<N$ and $1/p+1/q=1$. Then applying Lemma \ref{(lem2.3)} to
$T_{w,w_*}(f)$ with $r=p, s=2\gm-N/q$ gives \eqref{(2.41)}.  Finally
suppose $N/(N-1)< p<N$. Recalling \eqref{(2.37)} and using
\eqref{(2.25)} we have
\begin{equation*}
\left|T_{w,w_*}(f)(v)\right|\le C_{N,\gm}\la w\ra^{2\gm+1-N} \la w_*\ra^{2\gm+1-N}
\int_{{\bRN}} \frac{\la v_*\ra^{\gm+2-N}}{|v-v_*|}|f(v_*)|\, {\rm
  d}v_*.
\end{equation*}
Since $q=p/(p-1)<N$, it follows from Lemma \ref{(lem2.4)}
that \begin{equation*}
  \int_{{\bRN}} \frac{\la
    v_*\ra^{\gm+2-N}}{|v-v_*|}|f(v_*)|\, {\rm d}v_*\le C_{p}
  \|f\|_{L^1_{\gm+2-N}}^{1-\frac{q}{N}}
  \|f\|_{L^p_{\gm+2-N}}^{\frac{q}{N}}.
\end{equation*}
This proves~\eqref{(2.42)}.
\end{proof}

Let $f,g,h$ be nonnegative measurable functions on ${\bRN}$. Define for any
$s\ge 0$
$$
(f)_s(v):=\la v\ra ^sf(v).
$$
Then applying the inequality $\la v\ra\le \la v'\ra\la v_*'\ra$ we
have
\begin{equation*}
\begin{cases} \displaystyle
 \left(Q^{+}(f,g)\right)_s\le Q^{+}((f)_s, (g)_s),\vspace{0.2cm} \\ \displaystyle
\left(Q^{+}(f, Q^{+}(g, h))\right)_s \le Q^{+}\left((f)_s,
  Q^{+}((g)_s, (h)_s)\right),
\end{cases}
\end{equation*}
and so on and so forth. Consequently we have for all $s\ge 0$ and $1\le p\le
\infty$:
\begin{equation} \label{(2.46)}
\left\|Q^{+}(f, Q^{+}(g, h))\right\|_{L^p_s}\le
\left\|Q^{+}\left((f)_s,Q^{+}\left((g)_s, (h)_s\right)\right)\right\|_{L^p}
\end{equation}
 provided that the right hand side makes sense.

 Now we are going to prove the $L^{p}$ and $L^{\infty}$ boundedness of
 the iterated operator $Q^{+}(f, Q^{+}(g,h))$.  Let
\begin{equation} \label{(2.47)}
N_{\gm}=\left\{\begin{array}{ll} \left[\frac{N-1}{\gm}\right] & {\rm
if}\,\,\,0<\gm<N-2
\\ \\
\,\,\,\,\,\,1 & {\rm if}\,\,\,\gm\ge N-2
\end{array}\right.
\end{equation}
where $[x]$ denotes the largest integer not exceeding $x$.

\begin{theorem}\label{(theo2.1)} Suppose $N\ge 3$ and let
  $B(z,\sg)$ be defined by \eqref{(1.4)} with the condition
  \eqref{(b-upperbd)}. Given any $s\ge 0$ we have:
\begin{itemize}
\item[(i)]  Suppose $0<\gm<N-2$. Let
\,$N_{\gm}$ be defined in \eqref{(2.47)} and let
$$
p_n=\frac{N-1}{N-1-n\gm} \in (1,\infty],\quad n=1,2,\dots,N_{\gm}.
$$
Then
\begin{itemize}
\item[(a)] For all $n=1,2,\dots,N_{\gm}$, we have
  \begin{equation}\label{(2.48)}
\left\|Q^{+}(f, Q^{+}(g,h))\right\|_{L^{p_{1}}_s}\le
  C_{p_1}\|b\|_{L^{\infty}}^2\|f\|_{L^1_{s}}
\|g\|_{L^1_{s+\gm}} \|h\|_{L^1_{s+\gm}}.
\end{equation}
\item[(b)] If $1\le n\le N_{\gm}-1$, then
\begin{equation}\label{(2.49)}
  \left\|Q^{+}(f, Q^{+}(g,h))\right\|_{L^{p_{_{n+1}}}_s}\le
  C_{p_{_{n+1}}}\|b\|_{L^{\infty}}^2\|f\|_{L^{p_n}_{s+1}}
  \|g\|_{L^1_{s+\gm_1}} \|h\|_{L^1_{s+\gm_1}}^{1-\theta_n}
\|h\|_{L^{p_n}_{s+\gm_1}}^{\theta_n}
\end{equation}
where
\begin{equation}\label{(2.50)}
\gm_1=\max\{\gm \, ,\ 1\},\quad
\theta_n=\frac{1}{N}\left(1-\frac{N-2}{n}\right)^{+},\quad  n\ge 1.
\end{equation}

\item[(c)] Finally if $n=N_{\gm}$, then
\begin{equation}\label{(2.51)}
  \left\|Q^{+}(f,Q^{+}(g,h))\right\|_{L^{\infty}_s}\le
  C_{\infty}\|b\|_{L^{\infty}}^2 \|f\|_{L^1_s} ^{1-\alpha_1} \|f\|_{L^{p_{_{N_{\gm}}}}_s}^{\alpha_1}
\|g\|_{L^1_{s+\gm_*}} \|h\|_{L^1_{s+\gm_*}}^{1-\alpha_2}
\|h\|_{L^{p_{_{N_{\gm}}}}_s}^{\alpha_2} \end{equation}
where $\gm_*=(\gm-1)^{+}$ and
\begin{equation}\label{(2.52)}
0<\alpha_1 := \left( \frac{N-1}{\gm N_{\gm}}\right) \left(
  \frac{N-1-\gm}{N} \right) <1,\quad 0\le \alpha_2 :=\left(
  \frac{N-1}{\gm N_{\gm}}\right) \left( \frac{(1-\gm)^{+}}{N} \right)<1.
\end{equation}
\end{itemize}

\item[(ii)] Suppose $\gm\ge N-2$.  Let $1< p<N$, $1/p+1/q=1$. Then
  \begin{equation}\label{(2.53)}
    \left\|Q^{+}(f, Q^{+}(g,h))\right\|_{L^p_s}\le
    C_{p}\|b\|_{L^{\infty}}^2\|f\|_{L^1_{s+2\gm-N/q}}
\|g\|_{L^1_{s+2\gm-N/q}}\|h\|_{L^1_{s+2\gm
-N/q}}.
\end{equation}
Furthermore if $N/(N-1)<p<N$, then
\begin{equation}\label{(2.54)}
  \left\|Q^{+}\left(f, Q^{+}(g,h)\right)\right\|_{L^{\infty}_s}
  \le C_{p}\|b\|_{L^{\infty}}^2 \|f\|_{L^1_{s+\gm+2-N}}^{1-\frac{q}{N}}
  \|f\|_{L^p_{s+\gm+2-N}}^{\frac{q}{N}}
  \|g\|_{L^1_{s+2\gm+1-N}}\|h\|_{L^1_{s+2\gm+1-N}}.
\end{equation}
\end{itemize}

All the constants $C_p<\infty$ only depend on $N,\gm,p$.
\end{theorem}

\begin{remark} Observe that in the case $\gamma \ge N-2$, the iterated
  operator maps (forgetting about the weights)
  $L^1 \times L^1 \times L^1$ to $L^p$ for $1<p<N$. This can be
  recovered heuristically from Lions' theorem in~\cite{MR1284432}
  and~\cite{MR2081030} that states that $Q^+$ maps $L^1 \times H^s$ to
  $H^{s+(N-1)/2}$ for $s \in \mathbb R$. Then $L^1$ is contained in
  $H^{-N/2-0}$ and applying twice Lions' theorem one gets that the
  iterated operator maps $L^1\times L^1 \times L^1$ to
  $H^{2(N-1)/2-N/2-0} = H^{N/2-1-0}$. And the Sobolev embedding for
  the space $H^{N/2-1}$ is precisely $L^N$.
\end{remark}

\begin{proof}  The proof is a direct application of the inequalities
obtained in Lemma \ref{(lem2.4)} and Lemma \ref{(lem2.5)}. We can assume
that $f,g,h $ are all nonnegative.  And because of \eqref{(2.46)}, we need
only to prove the theorem for $s=0$.  By the integral representation of
$Q^{+}(f, Q^{+}(g,h))$ and the definition of $T_{w,w_*}$ we have
$$
Q^{+}(f,Q^{+}(g,h))(v)=\intt_{{\bRRN}}T_{w,w_*}(f)(v)g(w) h(w_*)\,
{\rm d}w \, {\rm d}w_*,
$$
and therefore
\begin{equation}\label{(2.55)}
\left\|Q^{+}\left(f, Q^{+}(g,h)\right)\right\|_{L^p}\le
\intt_{{\bRRN}}\left\|T_{w,w_*}(f)\right\|_{L^p}g(w) h(w_*) \, {\rm
  d}w \, {\rm d}w_*
\end{equation}
for all $1\le p\le\infty$.
\smallskip

\noindent
{\it Case~{\rm (i)}.} Suppose $0<\gm<N-2$. By \eqref{(2.38)} we have
\[
\left\|T_{w,w_*}(f)\right\|_{L^{p_1}}\le
C_{p_1}\|b\|_{L^{\infty}}^2\|f\|_{L^1}\la w\ra^{\gm}\la w_*\ra^{\gm}
\]
and so by \eqref{(2.55)}
$$
\left\| Q^{+}(f, Q^{+}(g,h)) \right\|_{L^{p_1}}\le
C_{p_1}\|b\|_{L^{\infty}}^2\|f\|_{L^1}\|g\|_{L^1_{\gm}}\|h\|_{L^1_{\gm}}.
$$
This proves \eqref{(2.48)} (with $s=0$).

- Suppose $1\le n\le N_{\gm}-1$. By definition of $p_n$ we have
$$
\frac{1}{p_{n+1}}=\frac{1}{p_n}+\frac{1}{p_1}-1,\quad
n=1,2,\dots,N_{\gm}-1.
$$
By \eqref{(2.39)} and \eqref{(2.55)} we have
\begin{equation} \label{(2.56)}
\begin{cases}\displaystyle
\left\|T_{w,w_*}(f)\right\|_{L^{p_{n+1}}}\le C_{p_{n+1}}\|b\|_{L^{\infty}}^2
\frac{\la w\ra^{1-\frac{p_1}{p_{n+1}}}\la
w_*\ra^{1-\frac{p_1}{p_{n+1}}}}{|w-w_*|^{1-\gm-\frac{1}{p_{n+1}}}}\|f\|_{L^{p_n}_1},
\vspace{0.2cm}  \\ \displaystyle
\|Q^{+}(f,Q^{+}(g,h))\|_{L^{p_{_{n+1}}}} \vspace{0.2cm} \\
\displaystyle \mbox{ } \qquad
\le C_{p_{_{n+1}}}\|b\|_{L^{\infty}}^2\|f\|_{L^{p_n}_1}
\intt_{{\bRRN}}\frac{\la w\ra^{1-\frac{p_1}{p_{n+1}}}\la w_*\ra^{1-\frac{p_1}{p_{n+1}}}}{|w-w_*|^{1-\gm-\frac{1}{p_{n+1}}}}g(w)h(w_*) {\rm
d}w \, {\rm d}w_*.
\end{cases}
\end{equation}
 Let $q_n\ge 1$ be defined by
\[
\frac{1}{q_n}+\frac{1}{p_n}=1, \quad \mbox{i.e.} \quad
q_n=\frac{N-1}{n\gm}.
\]
We have
$$
1-\gm-\frac{1}{p_{n+1}}=\frac{n+2-N}{N-1}\gm, \quad
q_n\left(1-\gm-\frac{1}{p_{n+1}}\right)=1-\frac{N-2}{n}.
$$
If $n\le N-2$, then $1-\gm-1/p_{n+1}\le 0$ and using
$|w-w_*|\le \la w\ra\la w_*\ra$ we have $$\frac{\la
  w\ra^{1-\frac{p_1}{p_{n+1}}}\la
  w_*\ra^{1-\frac{p_1}{p_{n+1}}}}{|w-w_*|^{1-\gm-\frac{1}{p_{n+1}}}}
\le \la w\ra^{\gm}\la w_*\ra^{\gm}$$
and so by \eqref{(2.56)} we have
\[
\left\|Q^{+}\left(f, Q^{+}(g,h)\right)\right\|_{L^{p_{_{n+1}}}} \le C_{p_{_{n+1}}}
\|b\|_{L^{\infty}}^2\|f\|_{L^{p_n}_1}
\|g\|_{L^1_{\gm}}\|h\|_{L^1_{\gm}}.
\]

\noindent If $n>N-2$, then $0<q_n(1-\gm- 1/p_{n+1})<1$ so that
applying Lemma \ref{(lem2.4)} (with $q=q_{n},\,
\alpha=1-\gm-1/p_{n+1}$) and recalling the definition of
$\theta_n$ we have
\begin{equation}\label{(2.57)}
  \intt_{{\bRRN}}\frac{\la w\ra^{1-\frac{p_1}{p_{n+1}}}\la w_*\ra^{1-\frac{p_1}{p_{n+1}}}}
  {|w-w_*|^{1-\gm-\frac{1}{p_{n+1}}}}g(w)h(w_*) \, {\rm d}w \, {\rm d}w_*  \le
  C_{p_{_{n+1}}}\|g\|_{L^1_1}(\|h\|_{L^1_1})^{1-\theta_n}(\|h\|_{L^{p_n}_1})^{\theta_n}
\end{equation}
and thus \eqref{(2.49)} (with $s=0$) follows from \eqref{(2.56)} and
\eqref{(2.57)}.

- Now let $n=N_{\gm}$. Let us recall that
\[
q_{N_\gm}=\frac{N-1}{N_{\gm}\gm}, \quad
N_{\gm}> \frac{N-1}{\gm}-1>0,
\]
hence
\begin{equation} \label{(2.58)}
q_{N_{\gm}}(N-1-\gm)< N-1.
\end{equation}
Using \eqref{(2.40)} and \eqref{(2.55)} (for the $L^{\infty}$ norm) we
have
\begin{equation} \label{(2.59)}
\begin{cases} \displaystyle
\left\|T_{w,w_*}(f)\right\|_{L^{\infty}}\le \frac{C_{\infty}\|b\|_{L^{\infty}}^2}{|w-w_*|^{1-\gm}}
\|f\|_{L^1}^{1-\frac{N-1-\gm}
{N}q_{N_{\gm}}} \|f\|_{L^{p_{N_{\gm}}}}^{\frac{N-1-\gm}{N}
q_{N_{\gm}} }, \vspace{0.2cm} \\ \displaystyle
\left\|Q^{+}\left(f,Q^{+}(g,h)\right)\right\|_{L^{\infty}} \vspace{0.2cm} \\ \displaystyle
\mbox{ } \qquad \le C_{\infty}\|b\|_{L^{\infty}}^2 \|f\|_{L^1} ^{1-\frac{N-1-\gm}
{N}q_{N_{\gm}}} \left( \|f\|_{L^{p_{N_{\gm}}}}\right)^{\frac{N-1-\gm}{N} q_{N_{\gm}}}
\intt_{{\bRRN}}\frac{g(w)h(w_*)}{|w-w_*|^{1-\gm}} \, {\rm d}w \, {\rm
    d}w_*.
\end{cases}
\end{equation}
If $0<\gm<1$, then from \eqref{(2.58)} and $N\ge 3$ we have
$0<q_{N_{\gm}}(1-\gm)<N-1$ so that using Lemma \ref{(lem2.3)} gives
\begin{equation} \label{(2.60)}
  \intt_{{\bRRN}}\frac{g(w)h(w_*)}{|w-w_*|^{1-\gm}} \, {\rm d}w \,
  {\rm d}w_*
 \le C_{\gm}\|g\|_{L^1}
\|h\|_{L^1}^{1-\frac{1-\gm}{N}q_{N_{\gm}}}
\|h\|_{L^{p_{N_{\gm}}}}^{\frac{1-\gm}{N}q_{N_{\gm}}}.
\end{equation}
If $1\le \gm<N-2$, then $|w-w_*|^{\gm-1}\le \la
w\ra^{\gm-1} \la w_*\ra^{\gm-1} $ and so
\begin{equation}\label{(2.61)}
  \intt_{{\bRRN}} \frac{g(w)h(w_*)}{|w-w_*|^{1-\gm}} \, {\rm
    d}w \, {\rm d}w_* \le\|g\|_{L^1_{\gm-1}}\|h\|_{L^1_{\gm-1}}.
\end{equation}
Thus \eqref{(2.51)} (with $s=0$) follows from \eqref{(2.59)},
\eqref{(2.60)} and \eqref{(2.61)}.
\smallskip

\noindent {\it Case~{\rm(ii)}.}  Suppose $\gm\ge N-2$ and let $1<p<N$, $1/p+1/q=1$. Then
\eqref{(2.53)} (with $s=0$) follows from \eqref{(2.55)} and
\eqref{(2.41)}. Furthermore if $N/(N-1)<p<N$, then \eqref{(2.54)} (with $s=0$)
follows from \eqref{(2.55)} (for the $L^{\infty}$ norm) and
\eqref{(2.42)}.
\end{proof}

\section {Iteration and Decomposition of Solutions}
\label{sec3}

We begin by the study of the process of iteration of the collision
operator and decomposition of solutions though the following
lemma. Roughly speaking the strategy of the decomposition is the
following. We use the Duhamel representation formula to decompose the
flow associated with the equation into two parts, one of which is more
regular than the initial datum, while the amplitude of the other
decreases exponentially fast with time, and we repeat this process in
order to increase the smoothness, starting each time a new flow having
the smoother part of the previous solution as initial datum. Each time
we start a new flow, we depart from the true solution, since the
initial datum is not the real solution, and we keep track of this
error through a Lipschitz stability estimate. Finally the times of the
decomposition are chosen in such a way that the time-decay of the
non-smooth parts dominates the time-growth in these Lipschitz
stability errors.

\begin{lemma}\label{(lem3.1} Let
  $B(z,\sg)$ be defined in \eqref{(1.4)} with $\gm\in(0,2]$ and with the condition
  \eqref{(Grad-1)}.  Let $f_t\in
  L^1_2({\bRN})$ be a mild solution of equation~\eqref{(B)}. Let us
  define
  \begin{equation} \label{(3.1)} \forall \, s, t\ge 0, \quad
    E_{s}^{t}(v):=\exp\left(-\int_{s}^{t}
      \int_{{\bRN}}|v-v_*|^{\gm}f_{\tau}(v_*) \, {\rm d}v_* \, {\rm
        d}\tau \right).
\end{equation}

Given  any $t_0\ge 0$ we also define for all $t\ge t_0$
\begin{eqnarray}\label{(3.2)}&&  f^{0}_t(v)=f_t(v),\quad  h^0_t(v) =0,\\
&&\label{(3.3)}f_{t}^{n}(v)
=\int_{t_0}^{t}E_{t_1}^{t}(v)\int_{t_0}^{t_1}
Q^{+}\left( f^{n-1}_{t_1},\,Q^{+}(f^{n-1}_{t_2}, f^{n-1}_{t_2})E_{t_2}^{t_1}
\right)(v) \, {\rm d}t_2 \, {\rm d}t_1, \\
&&\label{(3.4)}
 h_t^{1}(v)=f_{t_0}(v)E_{t_0}^{t}(v)+\int_{t_0}^{t}
Q^{+}\left(f_{t_1},\,f_{t_0}E_{t_0}^{t_1}\right)(v)E_{t_1}^{t}(v)\, {\rm d}t_1,\\
&&\label{(3.5)}h_t^{n}(v)=h_t^{1}(v) +\int_{t_0}^{t}E_{t_1}^{t}(v)
\int_{t_0}^{t_1}Q^{+}\left(f_{t_1},\,Q^{+}(f_{t_2},
\,h^{n-1}_{t_2})E_{t_2}^{t_1}\right)(v) \, {\rm d}t_2 \, {\rm d}t_1 \\
\nonumber &&
\qquad\,\,\, +\int_{t_0}^{t}E_{t_1}^{t}(v) \int_{t_0}^{t_1}
Q^{+}\left(f_{t_1},\,Q^{+}(f^{n-1}_{t_2},
h^{n-1}_{t_2})E_{t_2}^{t_1} \right)(v)\, {\rm d}t_2 \, {\rm d}t_1
\nonumber\\
&& \qquad\,\,\, +\int_{t_0}^{t}E_{t_1}^{t}(v)
\int_{t_0}^{t_1}Q^{+}\left(h^{n-1}_{t_1},\,
Q^{+}\left(f^{n-1}_{t_2}, f^{n-1}_{t_2}\right)E_{t_2}^{t_1}\right)(v)
\, {\rm d}t_2 \, {\rm d}t_1
\nonumber\end{eqnarray}
for $n=1,2,3, \dots$

Then $f_t^n\ge 0$, $h_t^n\ge 0$ and there is a null set $Z\subset
{\bRN}$ which is independent of $t$ and $n$ such that for all $v\in
{\bRN}\setminus Z$
\begin{equation} \label{(3.6)}
 \forall\, t\in[t_0,\infty),\quad n=1,2,3,\dots, \quad
 f_t(v)=f_{t}^{n}(v)+h_t^{n}(v).
\end{equation}
\end{lemma}

\begin{proof}  In the following we denote by $Z_0, Z_1, Z_2, \dots \subset
  {\bRN}$ some null sets (i.e. ${\rm meas}(Z_n)=0$) which are
  independent of the time variable $t$. The decomposition
  \eqref{(3.6)} is based on the Duhamel representation formula for the
  solution $f_t$: for all $v\in {\bRN}\setminus Z_0$
  \begin{equation} \label{(3.7)}
    \forall \, t\ge t_0, \quad f_t(v)=
    f_{t_0}(v)E_{t_0}^t(v)+\int_{t_0}^{t} Q^{+}(f_{t_1},
    f_{t_1})(v)E_{t_1}^t(v)\, {\rm d}t_1.
\end{equation}
 Here we note that in the definition of $E_{s}^{t}(v)$ we have used the assumption  \eqref{(Grad-1)}, i.e.,
$A_0=1$.  Applying \eqref{(3.7)} to $f_{t}$ at time $t=t_1$ and inserting it
into the second argument of $Q^{+}(f_{t_1}, f_{t_1})$ we obtain for
all $t\ge t_0$ and all $v\in {\bRN}\setminus Z_1$
\begin{equation} \label{(3.8)}
f_t(v)= h_t^1(v) +\int_{t_0}^{t}E_{t_1}^{t}(v)\int_{t_0}^{t_1}
Q^{+}\left( f_{t_1},\,Q^{+}(
f_{t_2}, f_{t_2})E_{t_2}^{t_1}\right)(v)
\, {\rm d}t_2 \, {\rm d}t_1.
\end{equation}
That is, we have the decomposition
$$
\forall\, t\ge t_0,\quad \forall\, v\in {\bRN}\setminus Z_1, \quad
f_t(v)=h_t^{1}(v)+f_{t}^{1}(v).
$$

Suppose for some $n\ge 1$, the decomposition $f_t(v)=
h^{n}_t(v)+f_t^{n}(v)$ holds for all $t\ge t_0$ and all $v\in
{\bRN}\setminus Z_n$.  Let us insert $f_{t_2}=h_{t_2}^n+f_{t_2}^n$ and
$f_{t_1}=h_{t_1}^n+f_{t_1}^n$ into $ Q^{+}\left( f_{t_1},\,Q^{+}(
  f_{t_2}, f_{t_2})E_{t_2}^{t_1}\right)$ in the following way:
\begin{equation*}
\begin{cases} \displaystyle
  Q^{+}\left( f_{t_1},\,Q^{+}( f_{t_2}, f_{t_2})E_{t_2}^{t_1} \right)=
  Q^{+}\left(f_{t_1},\,Q^{+}(f_{t_2}, h_{t_2}^n)E_{t_2}^{t_1}\right)
  +Q^{+}\left(f_{t_1},\,Q^{+}(f_{t_2}, f_{t_2}^n)E_{t_2}^{t_1}\right),
  \vspace{0.2cm} \\ \displaystyle Q^{+}\left(f_{t_1},\,Q^{+}(f_{t_2},
    f_{t_2}^n)E_{t_2}^{t_1}\right)
  =Q^{+}\left(f_{t_1},\,Q^{+}(f_{t_2}^n,
    h_{t_2}^n)E_{t_2}^{t_1}\right)+
  Q^{+}\left(f_{t_1},\,Q^{+}(f_{t_2}^n,
    f_{t_2}^n)E_{t_2}^{t_1}\right), \vspace{0.2cm} \\ \displaystyle
  Q^{+}\left(f_{t_1},\,Q^{+}(f_{t_2}^n, f_{t_2}^n)E_{t_2}^{t_1}\right)
  = Q^{+}\left(h_{t_1}^n,\,Q^{+}(f_{t_2}^n, f_{t_2}^n)E_{t_2}^{t_1}
  \right)+Q^{+}\left(f_{t_1}^n,\,Q^{+}(f_{t_2}^n,
    f_{t_2}^n)E_{t_2}^{t_1}\right).
\end{cases}
\end{equation*}
Then
\begin{multline*}
 Q^{+}\left( f_{t_1},\,Q^{+}(
f_{t_2}, f_{t_2})E_{t_2}^{t_1}\right)=
Q^{+}\left(f_{t_1},\,Q^{+}(f_{t_2}, h_{t_2}^n)E_{t_2}^{t_1}\right)
+Q^{+}\left(f_{t_1},\,Q^{+}(f_{t_2}^n, h_{t_2}^n)E_{t_2}^{t_1}\right) \\
+Q^{+}\left(h_{t_1}^n,\,Q^{+}(f_{t_2}^n, f_{t_2}^n)E_{t_2}^{t_1}
\right)+Q^{+}\left(f_{t_1}^n,\,Q^{+}(f_{t_2}^n,
  f_{t_2}^n)E_{t_2}^{t_1}\right).
\end{multline*}
Inserting this into \eqref{(3.8)} yields
$$
\forall\, t\ge t_0,\quad\forall\,\in {\bRN}\setminus Z_{n+1}, \quad
f_t(v)= h^{n+1}_t(v)+f_t^{n+1}(v).
$$
This proves the lemma by induction, and the null set $Z$ can be chosen
$Z=\bigcup_{n=1}^{\infty}Z_n$. \end{proof}

Note that the above iterations make sense since $f^0_t
= f_t\ge 0$ and, by induction, all $f^n_t$ are nonnegative.  Note also that if
$t_0>0$, then, by moment production estimate (Theorem \ref{(theo1.0)}),
we have
\[
\forall \, t \ge t_0,\quad f_t\in \bigcap_{s\ge 0} L^1_s({\bRN}).
\]
This enables us to use moment estimates for $Q^{+}(f, Q^{+}(g,h))$:
\begin{equation} \label{(3.9)} \forall \, s\ge 0, \quad
  \left\|Q^{+}\left(f, Q^{+}(g,h)\right)\right\|_{L^1_s} \le
  \|f\|_{L^1_{s+\gm}}\|g\|_{L^1_{s+2\gm}}\|h\|_{L^1_{s+2\gm}}.
\end{equation}

Before we can show the regularity property of $f^n_t$ and the
exponential decay (in norm) of $h^{n}_t$ we need further preparation.

Recall that the Sobolev space $H^{s}({\bRN})$ $(s>0)$ is a subspace of $f\in
L^2({\bRN})$ defined by
$$f\in H^{s}({\bRN}) \Longleftrightarrow  \|f\|_{H^{s}}=\|\widehat{f}\|_{L^2_s}
=\left(\int_{{\bRN}}\la \xi\ra ^{2s}|\widehat{f}(\xi)|^2 \, {\rm d}\xi\right)^{1/2}
<\infty$$
where $\widehat{f}(\xi)$ is the Fourier transform of $f$:
$$
\widehat{f}(\xi)={\mathcal F}(f)(\xi)=\int_{{\bRN}}f(v)e^{-{\rm i}\xi\cdot v}
\, {\rm d}v.
$$
As usual we denote the \emph{homogeneous} seminorm as
$$\|f\|_{\dot{H}^{s}}=\|\widehat{f}\|_{\dot{L}^2_{s}}
=\left(\int_{{\bRN}}|\xi|^{2s}|\widehat{f}(\xi)|^2 \, {\rm
    d}\xi\right)^{1/2}.
$$
The norm and seminorm are related by
\begin{equation}\label{(3.13-1)}
\| f \|_{\dot{H^s}} \le \|f\|_{H^{s}}\le (2\pi)^{N/2}2^{s/2}\|f\|_{L^2}+
2^{s/2}\|f\|_{\dot{H}^{s}}.\end{equation}

It is easily proved (see \cite[pp. 416-417]{MR1663589}) that if the
angular function $b$ satisfies
\begin{equation}\label{(3.12)}
\|b\|_{L^2}^2:=\left|{\mathbb
    S}^{N-2}\right|\int_{0}^{\pi}b(\cos\theta)^2\sin^{N-2}\theta
\, {\rm d}\theta <\infty
\end{equation}
then
$Q^{+}: L^2_{N+\gm}({\bRN})\times L^2_{N+\gm}({\bRN})\to L^2({\bRN})$
is bounded with
\begin{equation}\label{(3.13-2)}\|Q^{+}(f,g)\|_{L^2}\le C\|b\|_{L^2}\|f\|_{L^2_{N+\gm}}\|g\|_{L^2_{N+\gm}}
\end{equation}
where $C<\infty$ only depends on $N$ and $\gm$.  This together with
\eqref{(3.13-1)}, \eqref{(3.13-2)} and the estimate of
$\|Q^{+}(f,g)\|_{\dot{H}^{s}}$ obtained in \cite{MR1639275,MR1663589}
for $s=(N-1)/2$ leads to the following lemma.

\begin{lemma}[\cite{MR1639275,MR1663589}]\label{(lemBD)}
 Let $B(z,\sg)$ be defined in
  \eqref{(1.4)} with the condition \eqref{(3.12)}. Then $Q^{+}:
  L^2_{N+\gm}({\bRN})\times L^2_{N+\gm}({\bRN})\to
  H^{\frac{N-1}{2}}({\bRN})$ is bounded with the estimate
 $$\|Q^{+}(f,g)\|_{H^{\frac{N-1}{2}}}\le
 C\|b\|_{L^2}\|f\|_{L^2_{N+\gm}} \|g\|_{L^2_{N+\gm}} $$
where $C<\infty$ only depends on $N,\gm$.
\end{lemma}

The following lemma will be useful to prove the $H^1$-regularity of
$f^n_t$ in the decomposition $f_t=f_t^n+h_t^n$.

\begin{lemma}\label{(lem3.2)}
Let $B(z,\sg)$ be defined in \eqref{(1.4)} with $\gamma\in(0,2]$ and with the condition \eqref{(Grad-1)}. Let
  $F_t\in {\mathcal B}_{1,0,1}^{+}({\bRN})$ be a conservative measure
  strong solution of equation \eqref{(B)}. Then for any $t_0>0$ we
  have
\begin{equation}\label{(3.10)}
\forall\,t\ge t_0, \ \forall\, v\in{\bRN},\quad \int_{{\bRN}}|v-v_*|^{\gm}\, {\rm d}F_t(v_*)
\ge a\la v\ra^{\gm}\ge a\qquad \end{equation}
where
\begin{equation} \label{(3.11)} a:=\left[{\mathcal
      K}_4\Big(1+\max\{1,\,1/t_0\}\Big)^{2/\gm}\right]^{-(2-\gm)/2}
\end{equation}
and ${\mathcal K}_4={\mathcal K}_4(1,1+N)\,(>1)$ is the constant in
\eqref{(1.14)}. In particular if $t_0\ge 1$, then $a$ is independent
of $t_0$.

Moreover for any $t_0\le t_1\le t<\infty$, let $E_{t_1}^{t}(v)$ be defined as in
\eqref{(3.1)} for the measure $F_\tau$, i.e.
$$
E_{t_1}^{t}(v):=\exp\left(-\int_{t_1}^{t}
  \int_{{\bRN}}|v-v_*|^{\gm}\, {\rm d}F_\tau(v_*) \, {\rm
  d}\tau \right).
$$
Then for any $f\in L^{\infty}(\bRN) \cap L^1_2(\bRN)\cap H^1({\bRN})$ we have
$fE_{t_1}^{t}\in H^1({\bRN})$ and
\begin{equation} \label{(3-H-regularity)}
\left\|f E_{t_1}^{t}\right\|_{H^1(\bRN)}\le  C\left[\|f\|_{L^{\infty}(\bRN)}+\|f\|_{L^1_2(\bRN)}
+\|f\|_{H^1(\bRN)} \right] e^{-a(t-t_1)}(1+t-t_1)\end{equation}
where $C$ only depends on $N,\gm$.
\end{lemma}

\begin{proof} Let
\[
L_s(F_t)(v):=\int_{{\bRN}}|v-v_*|^{s}\, {\rm d}F_t(v_*).
\]
By conservation of mass, momentum and energy, we have
$L_2(F_t)(v)=N+|v|^2>\la v\ra ^2$ and so
the inequality
\eqref{(3.10)} is obvious when $\gm=2$. Suppose $0<\gm<2$. In this
case we use the inequality $|v-v_*|\le \la v\ra\la v_*\ra$ and the moment production estimate \eqref{(1.12)} with $s=4$
to get
\[
L_4(F_t)(v)\le \la v\ra^4\left\|F_t\right\|_{4} \le \la v\ra^4{\mathcal
  K}_4\left(1+1/t_0\right)^{\frac{2}{\gm}}.
\]
Then from the decomposition
$2=\gm\cdot\fr{2}{4-\gm}+4\cdot\fr{2-\gm}{4-\gm}$
and using H\"{o}lder inequality
we have
\begin{eqnarray*}
\la v\ra^2&<& L_2(F_t)(v) \le
[L_{\gm}(F_t)(v)]^{\frac{2}{4-\gm}} [L_4(F_t)(v)]^{\frac{2-\gm}{4-\gm}}\\
&\le &
[L_{\gm}(F_t)(v)]^{\frac{2}{4-\gm}}\la v\ra^{\frac{4(2-\gm)}{4-\gm}}\left[{\mathcal
K}_4(1+1/t_0)^{\frac{2}{\gm}}\right]^{\frac{2-\gm}{4-\gm}}.
\end{eqnarray*}
This gives
\[
\forall\,t\ge t_0, \quad \la v\ra^{\gm}\le L_{\gm}(F_t)(v)\left[{\mathcal
  K}_4(1+1/t_0)^{\frac{2}{\gm}}\right]^{\frac{2-\gm}{2}}\le \frac{1}{a}
L_{\gm}(F_t)(v)
\]
and \eqref{(3.10)} follows.

The proof of \eqref{(3-H-regularity)} is based on the following a
priori estimates.  First of all we have
$$|\p_{ v_j}E_{t_1}^{t}(v)|^2\le \gm^2 e^{-2a(t-t_1)}(t-t_1)\int_{t_1}^{t}{\rm d}\tau\int_{{\bRN}}|v-v_*|^{2(\gm -1)}{\rm d}F_{\tau}(v_*)$$
where we used Cauchy-Schwartz inequality, $\|F_{\tau}\|_0=1$, and $E_{t_1}^{t}(v)\le e^{-a(t-t_1)}$.
\smallskip

\noindent {\it Case 1: $0<\gm<1$.} In this case we have
$-N<2(\gm-1)<0$ so that
$$
\int_{{\bRN}}|f(v)|^2|v-v_*|^{2(\gm-1)}\, {\rm d}v
\le C\|f\|_{L^{\infty}}^2
+\|f\|_{L^2}^2$$
hence
$$
\sum_{j=1}^N\|f \p_{v_j}E_{t_1}^{t}(v) \|_{L^2}^2\le C \left(\|f\|_{L^{\infty}}^2
+\|f\|_{L^2}^2 \right)e^{-2a(t-t_1)}(t-t_1)^2.$$
\smallskip

\noindent
{\it Case 2: $ \gm \ge 1$.} Since $\gm\le 2$, this implies
$|v-v_*|^{2(\gm-1)}\le \la v\ra^2 \la v_*\ra^2$.
Then recalling $\|F_{\tau}\|_2=1+N$ and $f\in L^{\infty}({\bRN})$ we have
$$
\int_{{\bRN}}\left(\int_{{\bRN}}|f(v)|^2
  |v-v_*|^{2(\gm-1)}\, {\rm d}v\right)\, {\rm d}F_{\tau}(v_*)
\le(1+N)\|f\|_{L^{\infty}} \|f\|_{L^1_2}
$$
which shows that
\[
\sum_{j=1}^N\left\|f \p_{v_j}E_{t_1}^{t}\right\|_{L^2}^2
\le C \|f\|_{L^{\infty}} \|f\|_{L^1_2} e^{-2a(t-t_1)}(t-t_1)^2.
\]

Combing the two cases and using
$\|f\|_{L^{\infty}}\|f\|_{L^1_2}\le
\frac{1}{2}\|f\|_{L^{\infty}}^2+\frac{1}{2}\|f\|_{L^1_2}^2$
we obtain
\begin{equation*}
\left\|f E_{t_1}^{t} \right\|_{H^1(\bRN)}^2
\le  C\left(\|f\|_{L^{\infty}(\bRN)}^2+\|f\|_{L^1_2(\bRN)}^2+
\|f\|_{H^1(\bRN)}^2
\right)e^{-2a(t-t_1)}(1+t-t_1)^2.
\end{equation*}

A full justification requires standard smooth approximation arguments,
for instance one may replace $f$ and $|v-v_*|^{\gm}$ with
$f*\chi_{\vep}$ and $(\vep^2+|v-v_*|^2)^{\gm/2}$ respectively, and
then let $\vep\to 0^{+}$, etc., where
$\chi_{\vep}(v)=\vep^{-N}\chi(\vep^{-1}v)$ and $\chi\ge 0$ is a smooth
mollifier. We omit the details here.
\end{proof}

\begin{theorem}\label{(theo3.1)}  Suppose $N\ge 3$ and let $B(z,\sg)$
  be defined by \eqref{(1.4)} with $\gamma\in(0, 2]$ and with the
  conditions \eqref{(Grad-1)}-\eqref{(b-upperbd)}.  Let
  $f_t\in L^1_{1,0,1}({\bRN})$ be a conservative mild solution of
  equation \eqref{(B)}.

  Then for any $t_0>0$, the positive decomposition $f_t=f^n_t+h^n_t$
  given in (3.1)-(3.5) on $[t_0,\infty)$ satisfies the following
  estimates for all $s\ge 0$:
\begin{eqnarray} \label{(3.14)}&&
\sup_{n\ge N_{\gm}+1,\, t\ge t_0} \left\|f_{t}^{n}\right\|_{L^{\infty}_s}\le
C_{t_0, s}\\
&&\label{(3.15)}\sup_{n\ge N_{\gm}+2,\, t\ge t_0} \left\|f_{t}^{n}\right\|_{H^1}\le
C_{t_0}\\
&&\label{(3.16)}
\forall\, t\ge t_0,\ \forall\,n\ge 1, \quad \|h^n_t\|_{L^1_s}\le C_{t_0,s,n}e^{-\frac{a}{2}(t-t_0)} \\
&&\label{(3.15*)}
\forall\, t_1,t_2\ge t_0, \quad \sup_{n\ge
  1}\left\|f^n_{t_1}-f^n_{t_2}\right\|_{L^1_s},\quad
\sup_{n\ge 1}\left\|h^n_{t_1}-h^n_{t_2}\right\|_{L^1_s}\le
C_{t_0,s}|t_1-t_2|.
\end{eqnarray}
Here $N_{\gm}$ is defined by \eqref{(2.47)}, $a=a_{t_0}>0$ is given in
\eqref{(3.11)}, and $C_{t_0}, C_{t_0,s}, C_{t_0,s,n}$ are finite
constants depending only on $N$, $\gm$, the function $b$, $\max\{1,\,
1/t_0\}$, $s$, as well as $n$ in the case of $C_{t_0,s,n}$. In
particular if $t_0\ge 1$, all these constants are independent of $t_0$.
\end{theorem}

\begin{proof} Let
\[
\!|\!|\!|f\!|\!|\!|_{s}=\sup_{t\ge
  t_0}\|f_{t}\|_{L^1_s},\quad s\ge 0.
\]
By using Theorem~\ref{(theo1.0)}, \eqref{(norm)} and the fact that
\[
\|f_0\|_{L^1_0}=1, \quad \|f_0\|_{L^1_2}=1+N
\]
we have with ${\mathcal K}_s={\mathcal K}_s(1,1+N)$ that
\[
\forall\, s\ge 0, \quad \!|\!|\!|f\!|\!|\!|_{s}\le {\mathcal K}_s
\Big(1+\max\{1,\,1/t_0\}\Big)^{(s-2)^{+}/\gm}.
\]
We first prove \eqref{(3.14)} and \eqref{(3.16)}. To do this it
suffices to prove the following estimates
\eqref{(3.17)}-\eqref{(3.21)}:

$\bullet $ For all $s\ge 0$ and all $t\ge t_0$
\begin{equation} \label{(3.17)}
\forall \, n \ge 1, \quad \left\|h^n_{t}\right\|_{L^1_s} \le
\!|\!|\!|f\!|\!|\!|_{s+(2n-1)\gm} ^{2n}(1+t-t_0)^{2n-1}e^{-a(t-t_0)}.
\end{equation}

$\bullet $ If $0<\gm<N-2$, we then define
\[
p_n:=\frac{N-1}{N-1-n\gm}, \quad n=1,2,3,\dots,N_{\gm},
\]
and then for all $s\ge 0$
\begin{eqnarray} \label{(3.18)}&&
\max_{1\le n\le N_{\gm}}\sup_{t\ge t_0}\|f_{t}^{n}\|_{L^{p_{n}}_s}
\le C_{a} \, \!|\!|\!|f\!|\!|\!|_{s+s_1}^{\beta_*},\\
&&\label{(3.19)}
\sup_{n\ge N_{\gm}+1}\sup_{t\ge t_0}\|f_{t}^{n}\|_{L^{\infty}_s}\le
C_{a} \, \!|\!|\!|f\!|\!|\!|_{s+s_1}^{\beta^*},
\end{eqnarray}
where $s_1=N_{\gm}+\gm$ and $0<\beta_*, \beta^*<\infty$ depend only on
$N$ and $\gm$.

$\bullet $ If $\gm\ge N-2$ and $1<p<N$, then for all $s\ge 0$,
\begin{eqnarray}\label{(3.20)} &&
\sup_{t\ge t_0}\|f_{t}^{1}\|_{L^p_s}\le C_{a,p} \,
\!|\!|\!|f\!|\!|\!|_{s+2\gm-N/q} ^3,\\
&&\label{(3.21)}
\sup_{n\ge 2}\sup_{t\ge t_0}\|f_{t}^{n}\|_{L^{\infty}_s}\le
C_{a} \, \!|\!|\!|f\!|\!|\!|_{s+s_1}^{3+4/N},
\end{eqnarray}
where $s_2=3\gm+2-3N/2$.

In the following we denote by $C, C_{*}, C_{*,*}$ the finite positive
constants (larger than $1$) that only depend on
$N,\gm, A_2,\|b\|_{L^{\infty}},$ and on the arguments ``$*,*$''; they
may have different values in different places. 

By definition of
$E_{s}^{t}$ (see \eqref{(3.1)}) and Lemma \ref{(lem3.2)} we have
\begin{equation} \label{(3.22)}
\forall \, t_0\le t_2\le t_1\le t, \quad
\begin{cases} \displaystyle
  E_{t_1}^{t} \le  e^{-a(t-t_1)}, \vspace{0.2cm} \\ \displaystyle
  E_{t_2}^{t_1} E_{t_1}^{t} \le e^{-a(t-t_2)}.
\end{cases}
\end{equation}
We then deduce from \eqref{(3.22)} and $0\le f^{n}_t\le f_t $ that
\begin{eqnarray}\label{(3.23)}&&  h_t^{n}(v)\le
h_t^{1}(v)+ 2\int_{t_0}^{t}\int_{t_0}^{t_1}
e^{-a(t-t_2)}Q^{+}\left(f_{t_1},\,Q^{+}( f_{t_2},\,h^{n-1}_{t_2})
\right)(v)\, {\rm d}t_2 \, {\rm d}t_1\\
&&\nonumber
\qquad \,\,\,\, +\int_{t_0}^{t}\int_{t_0}^{t_1}
e^{-a(t-t_2)}Q^{+}\left(h^{n-1}_{t_1},\,Q^{+}(f_{t_2}, f_{t_2})\right)(v)\, {\rm
  d}t_2 \, {\rm d}t_1,\\
\label{(3.24)} && f^{n}_t(v) \le \int_{t_0}^{t}\int_{t_0}^{t_1}
e^{-a(t-t_2)} Q^{+}\left(f^{n-1}_{t_1},\,Q^{+}( f^{n-1}_{t_2},
f^{n-1}_{t_2})\right)(v)\, {\rm d}t_2 \, {\rm d}t_1.
\end{eqnarray}

Next by definition of $h_t^{1}$ in \eqref{(3.4)}
and using \eqref{(3.22)} and $\|f_{t}\|_{L^1_s}\ge \|f_t\|_{L^1}=1$  we have
$$
\|h_t^{1}\|_{L^1_s} \le
\!|\!|\!|f\!|\!|\!|_{s+\gm}^2(1+t-t_0)e^{-a(t-t_0)}.
$$
Suppose \eqref{(3.17)} holds for some $n\ge 1$. Using \eqref{(3.23)}
for $h^{n+1}_t$ we have
\begin{multline*}
\|h^{n+1}_t\|_{L^1_s} \le
  e^{-a(t-t_0)} \, \!|\!|\!|f\!|\!|\!|_{s+\gm}^2(1+t-t_0)+2 \,
  \!|\!|\!|f\!|\!|\!|_{s+2\gm}^2
  \int_{t_0}^{t}\int_{t_0}^{t_1}\|h^{n}_{t_2}\|_{L^1_{s+2\gm}}
  e^{-a(t-t_2)} \, {\rm d}t_2 \, {\rm d}t_1
  \\
  + \!|\!|\!|f\!|\!|\!|_{s+2\gm}^2
  \int_{t_0}^{t}\|h^{n}_{t_1}\|_{L^1_{s+\gm}}\int_{t_0}^{t_1}
  e^{-a(t-t_2)} \, {\rm d}t_2 \, {\rm d}t_1
  \\
  = \!|\!|\!|f\!|\!|\!|_{s+2\gm}^2e^{-a(t-t_0)}\left(1+t-t_0+
    \int_{t_0}^{t} \int_{t_0}^{t_1}e^{a(t_2-t_0)}\left(2\|h^{n}_{t_2}\|_{L^1_{s+2\gm}}+\|h^{n}_{t_1}\|_{L^1_{s+\gm}}
    \right)\, {\rm d}t_2 \, {\rm
      d}t_1 \right)
\end{multline*}
and by inductive hypothesis on $h^{n}_t$ we have for all $t_0\le
t_2\le t_1$,
\begin{multline*}
2\left\|h^{n}_{t_2}\right\|_{L^1_{s+2\gm}}+\left\|h^{n}_{t_1}\right\|_{L^1_{s+\gm}}
 \le 2\left(\|f\|_{s+2\gm+(2n-1)\gm}\right)^{2n}(1+t_2-t_0)^{2n-1}e^{-a(t_2-t_0)}
 \\
+ \left(\|f\|_{s+\gm+(2n-1)\gm}\right)^{2n}(1+t_1-t_0)^{2n-1}e^{-a(t_1-t_0)}\\
\le
3 \, \!|\!|\!|f\!|\!|\!|_{s+(2n+1)\gm}
^{2n}(1+t_1-t_0)^{2n-1}e^{-a(t_2-t_0)}.
\end{multline*}
So
$$
\|h^{n+1}_t\|_{L^1_s} \le
\!|\!|\!|f\!|\!|\!|_{s+(2n+1)\gm}^{2(n+1)}e^{-a(t-t_0)}\left(
  1+t-t_0+3 \int_{t_0}^{t}(1+t_1-t_0)^{2n-1}(t_1-t_0)\, {\rm d}t_1
\right).
$$
It is easily checked that
\[
\forall \, t \ge t_0, \quad 1+t-t_0+3\int_{t_0}^{t}(1+t_1-t_0)^{2n-1}(t_1-t_0) \, {\rm d}t_1\le
(1+t-t_0)^{2n+1}.
\]
Thus
$$
\|h^{n+1}_t\|_{L^1_s} \le
\!|\!|\!|f\!|\!|\!|_{s+(2n+1)\gm}^{2(n+1)}(1+t-t_0)^{2n+1}e^{-a(t-t_0)}.
$$
This proves \eqref{(3.17)}.

Now we are going to prove \eqref{(3.18)}-\eqref{(3.21)}. First of all
by \eqref{(3.24)} and the inequality
$$\int_{t_0}^{t}\int_{t_0}^{t_1}e^{-a(t-t_2)}\, {\rm d}t_2 \, {\rm d}t_1 \le \frac{1}{a^2}$$ we have
\begin{equation}\label{(3.25)}
  \sup_{t\ge t_0}\|f^n_t\|_{L^p_s}\le \frac{1}{a^2} \left( \sup_{t_1\ge t_2\ge t_0}
    \left\|Q^{+}(f^{n-1}_{t_1},\,Q^{+}( f^{n-1}_{t_2},
      f^{n-1}_{t_2}) )\right\|_{L^p_s} \right)
\end{equation}
for all $s\ge 0, 1\le p\le \infty$, provided that the right hand side makes sense.
\smallskip

\noindent
{\it Case 1: $0<\gm<N-2$.} We first prove that
\begin{equation}\label{(3.26)}
\forall\,s\ge 0,\quad \sup_{t\ge t_0}\|f^n_t\|_{L^{p_{n}}_s}
\le C_{a,n}\left( \!|\!|\!|f\!|\!|\!|_{s+n-1+\gm_1}\right)^{\beta_n},\qquad
n=1,2,\dots, N_{\gm}
\end{equation}
where $\gm_1=\max\{\gm, 1\}$ and
$$
\beta_n:=2(N+1)\left(1+\frac{1}{N}\right)^{n-1} +1-2N.
$$

By part (i) of Theorem \ref{(theo2.1)} we have
$$
\forall \, t_1\ge t_2\ge t_0, \quad
\left\|Q^{+}\left(f_{t_1},\,Q^{+}(f_{t_2}, f_{t_2})\right)\right\|_{L^{p_1}_s} \le
C_{1} \|f_{t_1}\|_{L^1_{s}}\|f_{t_2}\|_{L^1_{s+\gm}}^2 \le
C_{1} \, \!|\!|\!|f\!|\!|\!|_{s+\gm_1}^3.
$$
Using \eqref{(3.25)} with $p=p_1$ and $n=1$ (recalling
$f^{(0)}_t(v)=f_t(v)$) gives
$$
\forall \, s\ge 0, \quad \sup_{t\ge t_0}\left\|f^1_t\right\|_{L^{p_{1}}_s} \le
C_{a,1} \, \!|\!|\!|f\!|\!|\!|_{s+\gm_1}^3.
$$
Since $\beta_1=3$, this proves that the inequality in \eqref{(3.26)} holds for $n=1$.

Suppose the inequality in \eqref{(3.26)} holds for some $1\le n\le
N_{\gm}-1$. Then we compute using $0\le f^n_t\le f_t$ and part (I) of
Theorem \ref{(theo2.1)} that, for all $s\ge 0$,
\begin{multline}\label{(3.27)}
\left\|Q^{+}\left(f^{n}_{t_1},\,Q^{+}( f^n_{t_2},
f^n_{t_2})\right)\right\|_{L^{p_{_{n+1}}}_s}
\le  C_{n}\|f^n_{t_1}\|_{L^{p_n}_{s+1}}
\|f^n_{t_2}\|_{L^1_{s+\gm_1}}
\|f^n_{t_2}\|_{L^1_{s+\gm_1}}^{1-\theta_n} \|f^n_{t_2} \|_{L^{p_n}_{s+1}}^{\theta_n} \\
\le  C_{n} \, \!|\!|\!|f\!|\!|\!|_{s+\gm_1}^{2-\theta_n} \left(\sup_{t\ge
t_0}\|f_{t}^{n}\|_{L^{p_n}_{s+1}}\right)^{1+\theta_n}.
\end{multline}
By inductive hypothesis on $f^n_t$ we compute
\begin{equation}\label{(3.28)}
\!|\!|\!|f\!|\!|\!|_{s+\gm_1}^{2-\theta_n} \left(\sup_{t\ge
t_0}\|f_{t}^{n}\|_{L^{p_n}_{s+1}}\right)^{1+\theta_n} \le C_{a,n} \,
\!|\!|\!|f\!|\!|\!|_{s+n+\gm_1}^{2-\theta_n+\beta_n(1+\theta_n)}.
\end{equation}
Also by definition of $\theta_n$ and $\beta_n$ it is easily checked
that $2-\theta_n+\beta_n(1+\theta_n)<\beta_{n+1}$. It then follows
from \eqref{(3.25)}, \eqref{(3.27)} and \eqref{(3.28)} that
\[
\forall\, s\ge 0, \quad \sup_{t\ge t_0}\|f^{n+1}_t\|_{L^{p_{_{n+1}}}_s}\le C_{a,n+1} \,
\!|\!|\!|f\!|\!|\!|_{s+n+\gm_1}^{\beta_{n+1}}.
\]
This proves that the inequality in \eqref{(3.26)} holds for all
$n=1,2,\dots,N_{\gm}$.  From \eqref{(3.26)} and $N_{\gm}-1+\gm_1<
N_{\gm}+\gm=s_1$, we obtain \eqref{(3.18)} with
$\beta_*=\beta_{_{N_{\gm}}}$.

Next let us prove \eqref{(3.19)}. By Theorem \ref{(theo2.1)} (see
\eqref{(2.51)},\eqref{(2.52)}) and using \eqref{(3.18)} with
$n=N_{\gm}$ we have
\begin{multline*} \left\|Q^{+}(f^{N_{\gm}}_{t_1},\,Q^{+}(
    f^{N_{\gm}}_{t_2}, f^{N_{\gm}}_{t_2})
    )\right\|_{L^{\infty}_s}
  \le C\|f^{N_{\gm}}_{t_1}\|_{L^1_s}^{1-\alpha_1}
  \|f^{N_{\gm}}_{t_1}\|_{L^{p_{_{N_{\gm}}}}_s}^{\alpha_1}
  \|f^{N_{\gm}}_{t_2}\|_{L^1_{s+\gm}}
  \|f^{N_{\gm}}_{t_2}\|_{L^1_{s+\gm}}^{1-\alpha_2}\|f^{N_{\gm}}_{t_2}
  \|_{L^{p_{_{N_{\gm}}}}_s}^{\alpha_2}
  \\
  \le
  C\left(\!|\!|\!|f\!|\!|\!|_{s+s_1}\right)^{3+(\beta_{_{N_{\gm}}}-1)(\alpha_1+\alpha_2)}.\end{multline*}
This together with \eqref{(3.25)} gives
\begin{equation}\label{(3.29)}
\sup_{t\ge t_0}\left\|f^{{N_{\gm}}+1}_t\right\|_{L^{\infty}_s} \le
C_{a}\, \!|\!|\!|f\!|\!|\!|_{s+s_1} ^{\eta},\quad
\eta:=3+(\beta_{_{N_{\gm}}}-1)(\alpha_1+\alpha_2).
\end{equation}
Using \eqref{(3.25)} with $p=\infty$, Theorem \ref{(theo2.1)}, and
$$\|f^{N_{\gm}+k}_t\|_{L^{p_{_{N_{\gm}}}}_s}
\le  \!|\!|\!|f\!|\!|\!|_{s}^{1/p_{N_{\gm}}}
\left\|f^{N_{\gm}+k}_t\right\|_{L^{\infty}_s}^{1/q_{N_{\gm}}}$$
together with the $L^{\infty}_s$-boundedness \eqref{(3.29)} for $k=1$,
we deduce by induction on $k$ that, for all $s\ge 0$,
\begin{multline} \label{(3.30)}
\sup_{t\ge t_0}\left\|f^{N_{\gm}+k+1}_t\right\|_{L^{\infty}_s} \le
\frac{1}{a^2}\sup_{t_1\ge t_2\ge t_0}\left\|Q^{+}\left(f^{N_{\gm}+k}_{t_1},\,Q^{+}\left( f^{N_{\gm}+k}_{t_2},
f^{N_{\gm}+k}_{t_2}\right)\right)\right\|_{L^{\infty}_s}
\\
\le C_a\sup_{t_1\ge t_2\ge t_0}\Bigg\{\left\|f^{N_{\gm}+k}_{t_1}\right\|_{L^1_s}^{1-\alpha_1}
\left\|f^{N_{\gm}+k}_{t_1}\right\|_{L^{p_{_{N_{\gm}}}}_s}^{\alpha_1}\\
\times
\left\|f^{N_{\gm}+k}_{t_2}\right\|_{L^1_{s+\gm}}
\left\|f^{N_{\gm}+k}_{t_2}\right\|_{L^1_{s+\gm}}^{1-\alpha_2}
\left\|f^{N_{\gm}+k}_{t_2}\right\|_{L^{p_{N_{\gm}}}_s}^{\alpha_2}\Bigg\}
\\
\le C_a \, \!|\!|\!|f\!|\!|\!|_{s+\gm}^{3-(\alpha_1+\alpha_2)/{q_{N_{\gm}}}} \left(\sup_{t\ge
t_0}\|f^{N_{\gm}+k}_t\|_{L^{\infty}_s}\right)^{(\alpha_1+\alpha_2)/{q_{N_{\gm}}}}
\\
 =C_a \, \!|\!|\!|f\!|\!|\!|_{s+\gm} ^{3-\dt} \left(\sup_{t\ge
t_0}\|f^{N_{\gm}+k}_t\|_{L^{\infty}_s}\right)^{\dt}<\infty,\quad
k=1,2,3,\dots
\end{multline}
where
$$
\dt:=\frac{\alpha_1+\alpha_2}{q_{_{N_{\gm}}}}=\frac{N-1-\gm+(1-\gm)^{+}}{N}
\,\,\,\,(<1\,).
$$
Now fix any $s\ge 0$ and let us define
\[
A=C_a \, \!|\!|\!|f\!|\!|\!|_{s+\gm} ^{3-\dt} \quad \mbox{ and } \quad  Y_k=\sup_{t\ge
t_0}\left\|f^{N_{\gm}+k}_t\right\|_{L^{\infty}_s}.
\]
Then, from \eqref{(3.30)},
\[
Y_{k+1}\le AY_k^{\dt},\quad k=1,2,\dots
\]
which gives
\[
Y_{k+1}\le A^{1+\dt+\cdots+\dt^{k-1}}Y_{1}^{\dt^{k}}=
A^{\frac{1-\dt^{k}}{1-\dt}}Y_{1}^{\dt^{k}}\le A^{\frac{1}{1-\dt}}Y_{1},\quad
k=1,2,\dots
\]
It follows from \eqref{(3.29)} and $\gm<s_1$ that
$$ \sup_{t\ge
  t_0}\left\|f^{N_{\gm}+k+1}_t\right\|_{L^{\infty}_s}=Y_{k+1} \le
\left(C_{a} \, \!|\!|\!|f\!|\!|\!|_{s+\gm}^{3-\dt}\right)^{\frac{1}{1-\dt}}
C_{a} \, \!|\!|\!|f\!|\!|\!|_{s+s_1}^{\eta} \le
C_{a} ^{1+1/(1-\delta)} \, \!|\!|\!|f\!|\!|\!|_{s+s_1}^{\frac{3-\dt}{1-\dt}+\eta}$$ for
all $k=1,2,3,\dots$ This gives \eqref{(3.19)} with
$\beta^*= (3-\dt)/(1-\dt)+\eta$.
\smallskip

\noindent {\it Case 2: $\gm\ge N-2$.} By Theorem \ref{(theo2.1)} we
have for any $1<p<N$ and $s\ge 0$
\begin{equation*}
\left\|Q^{+}\left(f_{t_1},\,Q^{+}( f_{t_2}, f_{t_2})\right)\right\|_{L^p_s}\le C_{p}
\|f_{t_1}\|_{L^1_{s+\gm-N/q}} \|f_{t_2}\|_{L^1_{s+2\gm-N/q}}^2 \le C_{p}
\, \!|\!|\!|f\!|\!|\!|_{s+2\gm-N/q}^3.
\end{equation*}
This together with \eqref{(3.25)} with $n=1$ implies that
\[
\forall\,s\ge 0, \quad \sup\limits_{t\ge t_0}\|f^1_t\|_{L^p_s}\le C_{a,p}
\, \!|\!|\!|f\!|\!|\!|_{s+2\gm-N/q}^3.
\]
This proves \eqref{(3.20)}.  In particular for $p=2$ we have
\begin{equation}\label{(3.31)}
\forall\, s\ge 0, \quad \sup_{t\ge t_0}\|f^1_t\|_{L^2_s}\le C_{a}
\, \!|\!|\!|f\!|\!|\!|_{s+2\gm-N/2}^3.
\end{equation}
Then using \eqref{(3.25)} with $p=\infty$, Theorem \ref{(theo2.1)}
with
\[
p=q=2\in \left(N/(N-1),\, N\right),
\]
and induction on $n$ starting from $n=1$ with the $L^2_s$-boundedness
\eqref{(3.31)} we have, for all $s\ge 0$,
\begin{eqnarray}&&\label{(3.32)} \sup_{t\ge
    t_0}\left\|f^{n+1}_t\right\|_{L^{\infty}_s}\le \frac{1}{a^2} \sup_{t_1\ge
    t_2\ge t_0} \left\|Q^{+}\left(f^{n}_{t_1},\,Q^{+}\left(f^{n}_{t_2},f^{n}_{t_2}\right)\right)
  \right\|_{L^{\infty}_s} \\
  &&\le C_a \sup_{t\ge t_1\ge
    t_0}\left(\left\|f^{n}_{t_1}\right\|_{L^1_{s+\gm+2-N}}^{1-2/N}
  \left\|f^{n}_{t_1}\right\|_{L^2_{s+\gm+2-N}}^{2/N}\left\|f^{n}_{t_2}\right\|_{L^1_{s+2\gm+1-N}}^2\right)
  \nonumber \\
  &&\nonumber\le C_a \, \!|\!|\!|f\!|\!|\!|_{s+2\gm+1-N})^{3-2/N}
  \left(\sup_{t\ge t_0}\left\|f^{n}_{t}\right\|_{L^2_{s+\gm+2-N}}\right)^{2/N}<\infty,\quad
  n=1,2,3,\dots
\end{eqnarray}
Taking $n=1$ and using \eqref{(3.31)} and $2\gm+1-N< 3\gm+2-3N/2=:s_2$
we obtain
\begin{equation} \label{(3.33)}
\sup_{t\ge t_0}\left\|f^2_t\right\|_{L^{\infty}_s}\le C_{a} \,
  \!|\!|\!|f\!|\!|\!|_{s+s_2}^{3+4/N}.
\end{equation}
Further, using
\[
\forall \, n\ge 2, \quad \left\|f^{n}_{t}\right\|_{L^2_{s+\gm+2-N}} \le
\!|\!|\!|f\!|\!|\!|_{s+2\gm+4-2N}^{1/2}
\left\|f^{n}_{t}\right\|_{L^{\infty}_s} ^{1/2}
\]
and $2\gm+4-2N\le 2\gm+1-N\le s_2$ (because $\gm\ge N-2\ge 1$) we get
from \eqref{(3.32)} that
$$
\sup_{t\ge t_0}\left\|f^{n+1}_t\right\|_{L^{\infty}_s}\le C_{a}
\, \!|\!|\!|f\!|\!|\!|_{s+s_2}^{3-1/N} \left(\sup_{t\ge
t_0}\|f^{n}_{t}\|_{L^{\infty}_{s}}\right)^{1/N},\quad
n=2,3,\dots
$$ By
iteration we deduce, as shown above with $\dt=1/N$, and using
\eqref{(3.33)} that \begin{eqnarray*}&&
\sup_{t\ge t_0}\|f^{n+1}_t\|_{L^{\infty}_s}\le \left(C_{a}
  \, \!|\!|\!|f\!|\!|\!|_{s+s_2}^{3-1/N}\right)^{\frac{1-\dt^{n-1}}{1-\dt}}
\left(\sup_{t\ge t_0}\|f^2_t\|_{L^{\infty}_s}\right)^{\dt^{n-1}}\\
&&\le C_a \left(\!|\!|\!|f\!|\!|\!|_{s+s_2}\right)^{
(3-\frac{1}{N})\frac{1-\dt^{n-1}}{1-\dt}+(3+4/N)\dt^{n-1}}\le
C_{a}\left(\!|\!|\!|f\!|\!|\!|_{s+s_2}\right)^{3+4/N},\quad
n=2,3,\dots
\end{eqnarray*}
This proves \eqref{(3.21)}.

Now let us prove the $H^1$-regularity \eqref{(3.15)} of $f^{n}_t$ for
$n\ge N_{\gm}+2$.  For notation convenience we denote
$$
Q_{t_1, t_2}^{n-1}(v):=Q^{+}\left(f^{n-1}_{t_1},\,Q^{+}\left(f^{n-1}_{t_2},
f^{n-1}_{t_2}\right)E_{t_2}^{t_1}\right)(v).
$$
The iteration formula \eqref{(3.3)} is then written
\begin{equation} \label{(3-iteration)} \forall \, t\ge t_0, \quad
  f_{t}^{n}(v) =\int_{t_0}^{t}E_{t_1}^{t}(v)\int_{t_0}^{t_1}Q_{t_1,
    t_2}^{n-1}(v)\, {\rm d}t_2 \, {\rm d}t_1.
\end{equation}
Applying Theorem \ref{(theo2.1)} and the $L^{\infty}_s$ estimate in
\eqref{(3.14)} we have
$$
\left\|Q_{t_1, t_2}^{n-1} \right\|_{L^{\infty}} \le
e^{-a(t_1-t_2)}\left\|Q^{+}\left(f^{n-1}_{t_1},\,Q^{+}(f^{n-1}_{t_2},
    f^{n-1}_{t_2})\right)\right\|_{L^{\infty}} \le
C_{t_0}e^{-a(t_1-t_2)}.
$$
Also by $f^{n-1}_t\le f_t$ we have
$$
\left\|Q_{t_1, t_2}^{n-1}\right\|_{L^1_2} \le
e^{-a(t_1-t_2)}\left\|Q^{+}\left(f^{n-1}_{t_1},\,Q^{+}\left(f^{n-1}_{t_2},
  f^{n-1}_{t_2}\right)\right)\right\|_{L^1_2} \le C_{t_0}e^{-a(t_1-t_2)}.
$$
And using Lemma \ref{(lemBD)}, \eqref{(3.13-2)} and the $L^{\infty}_s$
estimate in \eqref{(3.14)} we have
 \begin{multline*}
\left\|Q_{t_1,t_2}^{n-1}\right\|_{H^1}\le
\left\|Q_{t_1,t_2}^{n-1}\right\|_{H^{\frac{N-1}{2}}} \le C
\left\|f^{n-1}_{t_1}\right\|_{L^2_{N+\gm}} \left\|Q^{+}(f^{n-1}_{t_2},
f^{n-1}_{t_2})E_{t_2}^{t_1}\right\|_{L^2_{N+\gm}}\\
 \le C\left\|f^{n-1}_{t_1}\right\|_{L^2_{N+\gm}}\left\|f^{n-1}_{t_2}\right\|_{L^2_{2(N+\gm)}}^2e^{-a(t_1-t_2)}
\le C_{t_0}e^{-a(t_1-t_2)}.
\end{multline*}
Thus we conclude from Lemma \ref{(lem3.2)} that
$Q_{t_1,t_2}^{n-1}E_{t_1}^t\in H^1({\bRN})$ and
\begin{equation}\label{(3-Q-E)}
  \left\|Q_{t_1,t_2}^{n-1}E_{t_1}^t\right\|_{H^1}\le C_{t_0}e^{-a(t_1-t_2)} e^{-a(t-t_1)}(1+t-t_1)
  = C_{t_0}e^{-a(t-t_2)}(1+t-t_1).
\end{equation}
Using Minkowski inequality to \eqref{(3-iteration)} we then conclude
from \eqref{(3-Q-E)} and the above estimates that $f^{n}_t\in
H^{1}({\bRN})$ and
$$\left\|f^{n}_t\right\|_{H^1}\le\int_{t_0}^{t}\int_{t_0}^{t_1}
\left\|Q_{t_1,t_2}^{n-1} E_{t_1}^{t}\right\|_{H^1}\, {\rm d}t_2 \, {\rm d}t_1\le
C_{t_0}\int_{t_0}^{t}\int_{t_0}^{t_1}e^{-a(t-t_2)}(1+t-t_1)\, {\rm
  d}t_2 \, {\rm d}t_1\le C_{t_0}.
$$
This proves \eqref{(3.15)}.

Finally let us prove \eqref{(3.15*)}. To do this we rewrite $f^n_t$ as
follows (recall definition of $f^n_t$ in \eqref{(3.3)})
$$f_{t}^{n}(v)
=E_{t_0}^t(v)\int_{t_0}^{t}E_{t_1}^{t_0}(v) \int_{t_0}^{t_1}
Q^{+}\left( f^{n-1}_{t_1},\,Q^{+}\left(f^{n-1}_{t_2}, f^{n-1}_{t_2}\right)E_{t_2}^{t_1}
\right)(v)\, {\rm d}t_2 \, {\rm d}t_1
$$
and recall that
$$
E_{s}^t(v) =\exp\left(-\int_{s}^{t}L_{\gm}(f_{\tau})(v)\, {\rm d}\tau
\right),\quad
L_{\gm}(f_{\tau})(v):=\int_{{\bRN}}|v-v_*|^{\gm}f_\tau(v_*) \, {\rm
  d}v_*.
$$
Then it is easily seen that the function $t\mapsto f^n_t(v)$ is
absolutely continuous on every bounded subinterval of $[t_0,\infty)$
and
$$\frac{\p}{\p t}f_{t}^{n}(v)
=\int_{t_0}^{t}  Q^{+}\left(f^{n-1}_{t},\,Q^{+}\left(
f^{n-1}_{t_2}, f^{n-1}_{t_2}\right)E_{t_2}^{t}\right)(v)\, {\rm d}t_2
-L_{\gm}(f_t)(v)f^n_t(v),\quad {\rm a.e.}\,\,\, t\ge t_0.
$$
Since
\begin{multline*}
  \left\| Q^{+}\left( f^{n-1}_{t},\,Q^{+}\left(
    f^{n-1}_{t_2}, f^{n-1}_{t_2}\right)E_{t_2}^{t}\right)\right\|_{L^1_s} \le
  \left\| Q^{+}(f_{t},\,Q^{+}( f_{t_2},
    f_{t_2}))\right\|_{L^1_s}e^{-a(t-t_2)}  \\
  \le
  \left\|f_{t}\right\|_{L^1_{s+\gm}}
  \left\|f_{t_2}\right\|_{L^1_{s+2\gm}}^2 e^{-a(t-t_2)}
  \le C_{t_0,s}e^{-a(t-t_2)}
\end{multline*}
and
\begin{equation*}
\left\|L_{\gm}(f_t)f^n_t\right\|_{L^1_s}\le
    \left\|f_t\right\|_{\gm}\left\|f_t\right\|_{L^1_{s+\gm}} \le
        C_{t_0,s}
\end{equation*}
it follows that
\[
\left\|\frac{\p}{\p t}f_{t}^{n}\right\|_{L^1_s}\le
  C_{t_0,s} \quad \mbox{ a.e. } t\ge t_0.
\]
Thus, by the absolute continuity of $t\mapsto f^n_t(v)$, we deduce
that
\[
\forall \, t_1, t_2 \ge t_0, \quad \left\|f^n_{t_1}-f^n_{t_2}\right\|_{L^1_s}\le
C_{t_0,s}|t_1-t_2|.
\]

On the other hand, from
$$
\forall\,t\ge
t_0, \quad f_t(v)=f_{t_0}(v)+\int_{t_0}^{t}\left[Q^{+}(f_{\tau},f_{\tau})(v)
-L_{\gm}(f_{\tau})(v)f_{\tau}(v)\right]\, {\rm d}\tau
$$
we also have $\left\|f_{t_1}-f_{t_2}\right\|_{L^1_s}\le
C_{t_0,s}|t_1-t_2|$ for all $t_1,t_2\ge t_0$. Thus the function
$t\mapsto h^n_t=f_t-f^n_t$ satisfies the same estimate.  This proves
\eqref{(3.15*)} and completes the proof of the theorem.
\end{proof}

\begin{corollary}\label{(corol3.2)}  Suppose $N\ge 3$ and
 let  $B(z,\sg)$ be defined in \eqref{(1.4)} with $\gamma \in
  (0,2]$ and with the conditions  \eqref{(Grad-1)},\eqref{(b-upperbd)}.    Let $F_t\in B^{+}_{1,0,1}({\bRN})$ be a conservative measure
  strong solution of equation~\eqref{(B)}. Then for any $t_0>0$, $F_t$
  can be decomposed as
\begin{equation} \label{(3.45)}
\forall \, t\ge t_0, \quad \, {\rm d}F_t(v)=f_t(v)\, {\rm d}v+\,
  {\rm d}\mu_t(v),
\end{equation}
with
\[
0\le f_t\in \bigcap_{s\ge 0}L^{\infty}_s\cap H^1({\bRN}), \quad
\mu_t\in \bigcap_{s\ge 0}{\mathcal B}^{+}_s({\bRN}),
\]
satisfying for all $s\ge 0$
\begin{eqnarray}&&\label{(3.47)}
  \sup_{t\ge t_0}\left\|f_t\right\|_{L^{\infty}_s}\le C_{t_0, s}, \qquad  \sup_{t\ge t_0}\left\|f_t\right\|_{H^{1}}\le C_{t_0}\\
  &&\label{(3.48)} \forall\, t\ge t_0, \quad \left\|\mu_t\right\|_s\le
  C_{t_0, s}e^{-\frac{a}{2}(t-t_0)},
  \\
  &&\label{(3.46)} \forall\, t_1,t_2\in[t_0,\infty), \quad
  \left\|f_{t_1}-f_{t_2}\right\|_{L^1_s},\,\,\left\|\mu_{t_1}-\mu_{t_2}\right\|_s\le
  C_{t_0,s}|t_1-t_2|,
\end{eqnarray}
where $a=a_{t_0}>0$ is given in \eqref{(3.11)} and $C_{t_0},
C_{t_0,s}$ are finite constant depending only on $N$, $\gm$, the function
$b$, $\max\{1,\, 1/t_0\}$ and $s$.
\end{corollary}

\begin{proof} By Theorem \ref{(theo1.0)}, there is a sequence
  $\{f_{k,t}\}_{k=1}^{\infty}\subset L^1_{1,0,1}({\bRN})$ of
  conservative mild solutions of equation~\eqref{(B)} such that
\begin{equation} \label{(3.49)}
\forall\,\vp\in C_c({\bRN}),\quad
\forall\,t\ge 0, \quad \lim_{k\to\infty}\int_{{\bRN}}\vp(v) f_{k,t}(v)\, {\rm d}v
=\int_{{\bRN}}\vp(v)\, {\rm d}F_t(v).
\end{equation}
Let $n_{\gm}=N_{\gm}+2$ with $N_{\gm}$ defined in \eqref{(2.47)} and
consider the positive decompositions of $f_{k,t}$
$$
\forall \, k=1,2,3,\dots \ \forall \, t\ge t_0, \quad
f_{k,t}(v)=f_{k,t}^{n_{\gm}}(v)+h_{k,t}^{n_{\gm}}(v),\quad
$$
given by \eqref{(3.1)}-\eqref{(3.6)}. By Theorem \ref{(theo3.1)} we
have for all $s\ge 0$
\begin{eqnarray}\label{(3.50)}&&\sup_{k\ge 1, t\ge t_0}\left\|f_{k,t}^{n_{\gm}}\right\|_{L^{\infty}_s}\le
  C_{t_0, s},\quad \sup_{k\ge 1,\,t\ge t_0}\left\|f_{k,t}^{n_{\gm}}\right\|_{H^1}\le
  C_{t_0},\\
  &&\label{(3.51)}
  \forall \, t\ge t_0, \quad \sup_{k\ge 1}\left\|h_{k,t}^{n_{\gm}}\right\|_{L^1_s}\le C_{t_0, s}
  e^{-\frac{a}{2}(t-t_0)}, \\
  &&\label{(3.52)}\forall\, t_1, t_2\ge t_0, \quad
  \sup_{k\ge 1}\left\|f^{n_{\gm}}_{k,t_1}-f^{n_{\gm}}_{k,t_2}\right\|_{L^1_s},\,\,
  \sup_{k\ge 1}\left\|h^{n_{\gm}}_{k,t_1}-h^{n_{\gm}}_{k,t_2}\right\|_{L^1_s}\le
  C_{t_0,s}|t_1-t_2|.\end{eqnarray}
From \eqref{(3.50)}, it is easily seen that for every $t\ge t_0$,
$\{f^{n_{\gm}}_{k,t}\}_{k=1}^{\infty}$ is relatively compact in
$L^1({\bRN})$. Moreover by using the density of rational times, a
diagonal process and \eqref{(3.52)}, one can prove that
there is a common subsequence $\{f^{n_{\gm}}_{k_j,t}\}_{j=1}^{\infty}$
(where $\{k_j\}_{j=1}^{\infty}$ is independent of $t$) and a function $0\le f_t\in
L^1({\bRN})$, 
such that
\begin{equation} \label{(3.53)}
\forall\,t\ge t_0, \quad \left\|f^{n_{\gm}}_{k_j,t}-f_t\right\|_{L^1}
\xrightarrow[j\to\infty]{} 0.
\end{equation}
Since $h^{n_{\gm}}_{k_j,t}=f_{k_j,t}-f^{n_{\gm}}_{k_j,t}$, it follows
from \eqref{(3.53)} and the weak convergence \eqref{(3.49)} that for
every $t\ge t_0$, $h^{n_{\gm}}_{k_j,t}$ converges weakly to some
$\mu_t\in{\mathcal B}_0^{+}({\bRN})$ as $j\to\infty$, i.e.
\begin{equation} \label{(3.54)} \forall\,\vp\in C_c({\bRN}),\
  \forall\,t\ge t_0, \quad \lim_{j\to\infty}\int_{{\bRN}}\vp(v)
  h_{k_j,t}^{n_{\gm}}(v)\, {\rm d}v =\int_{{\bRN}}\vp(v)\, {\rm
    d}\mu_t(v).\end{equation} This leads to the decomposition
\eqref{(3.45)}. The inequalities \eqref{(3.47)}, \eqref{(3.48)},
\eqref{(3.46)} follow easily from \eqref{(3.50)}, \eqref{(3.51)},
\eqref{(3.52)}, \eqref{(3.53)}, \eqref{(3.54)} and the equivalent
version \eqref{(dual)} of measure norm $\|\cdot\|_s$.
\end{proof}

\section{Rate of Convergence to Equilibrium }
\label{sec4}

This section is devoted to the proof of Theorem \ref{(theo1.1)}.  We
first recall the results in \cite{Mcmp} on the exponential rate
of convergence to equilibrium for $L^1$ mild solutions.

\begin{theorem}[Cf. Theorem 1.2 of \cite{Mcmp}]\label{(theo4.1)}
  Suppose $N\ge 3$ and let $B(z,\sg)$ be defined in \eqref{(1.4)} with
  $\gm\in (0, \min\{2, N-2\}]$ and with the conditions
  \eqref{(Grad-1)}-\eqref{(b-upperbd)}-\eqref{(b-lowerbd)}. Let
  $\ld=S_{b,\gm}(1,0, 1)>0$ be the spectral gap of the linear operator
  $L_M$ in \eqref{(1-linear operator)} associated with $B(z,\sg)$ and
  the Maxwellian $M(v)=(2\pi)^{-N/2}e^{-|v|^2/2}$ in
  $L^1_{1,0,1}({\bRN})$.

  Let $f_0\in L^1_{1,0,1}({\bRN})\cap L^2({\bRN})$ and let $f_t\in
  L^1_2({\bRN}) $ be the unique conservative solution of equation
  \eqref{(B)} with the initial datum $f_0$. Then there is a constant
  $0<C<\infty$, which depends  only on $N$,
  $\gm$, the function $b$, and on $\left\|f_0\right\|_{L^2}$, such that
$$
\forall\, t\ge 0, \quad \left\|f_t-M\right\|_{L^1}\le Ce^{-\ld t}.
$$
In the important case of hard sphere model (i.e. $N=3$, $\gm=1$, and
$b={\rm const.}$), the assumption ``$ f_0\in L^1\cap L^2$" can be
relaxed to ``$ f_0\in L^1$" and the same result holds with the
constant $C$ depending  only on $N$, $\gm$, and the function $b$.
\end{theorem}

\begin{lemma}[Cf. Lemma 4.6 of \cite{Mcmp}]\label{(lem4.1)}
  Suppose $N\ge 3$ and let $B(z,\sg)$ be defined in \eqref{(1.4)} with
  $\gm\in (0, \min\{2, N-2\}]$, and with the conditions
  \eqref{(Grad-1)}-\eqref{(b-upperbd)}-\eqref{(b-lowerbd)}. Let
  $\ld=S_{b,\gm}(1,0, 1)>0$ be the spectral gap of the linear operator
  $L_M$ in \eqref{(1-linear operator)} associated with $B(z,\sg)$ and
  the Maxwellian $M(v)=(2\pi)^{-N/2}e^{-|v|^2/2}$ in
  $L^1_{1,0,1}({\bRN})$.  Let $\alpha>0$,
  $m(v)=e^{-\alpha |v|^{\gm}}$.

  Then there are some explicitable finite constants $\vep>0, C>0$
  depending  only on $N$, $\gm$, the function $b$, and
  $\alpha$, such that if $f_t$ with the initial datum $f_0\in
  L^1_{1,0,1}({\bRN})\cap L^1({\bRN}, m^{-2})$ is a conservative
  solution to equation \eqref{(B)} satisfying
$$
\forall\, t\in[0,\infty), \quad \left\|f_t-M\right\|_{L^1(m^{-2})}\le
\vep
$$
then
$$
\forall\, t\in[0,\infty), \quad \left\|f_t-M\right\|_{L^1(m^{-1})}\le
C\left\|f_0-M\right\|_{L^1(m^{-1})}e^{-\ld t}.
$$
\end{lemma}
\vskip2mm

\begin{remarks}\label{remark4.1}
\begin{enumerate}
\item The original version of Theorem \ref{(theo4.1)} and Lemma
  \ref{(lem4.1)} in \cite{Mcmp} were proved for the class
  $L^1_{\pi^{N/2}, 0,1/2}({\bRN})$, i.e. for the Maxwellian
  $M(v)=M_{\pi^{N/2},0,1/2}(v)=e^{-|v|^2}$. According to Proposition
  \ref{(prop1.1)} (normalization), these are equivalent to the present
  version. In fact let $g_t\in L^1_{\pi^{N/2},0,1/2}({\bRN}) $,
  $f_{t}\in L^1_{1,0,1}({\bRN})$ have the relation
$$f_t(v)=(2\pi)^{-N/2}g_{t/c}(v/\sqrt{2}),\quad{\rm i.e.}\quad
g_t(v)=(2\pi)^{N/2}f_{ct}(\sqrt{2}\,v),\quad t\ge 0$$ where
$c=\pi^{N/2}2^{-\gm/2}$.  Then $f_t$ is a conservative solution of
equation~\eqref{(B)} if and only if $g_t$ is a conservative solution
of equation~\eqref{(B)}.  And we have
\begin{equation*}
  \forall \, t\ge 0, \quad
  \left\|f_{t}-M_{1,0,1}\right\|_{L^1}=\pi^{-N/2}\left\|g_{t/c}-M_{\pi^{N/2},0,1/2}\right\|_{L^1},
\end{equation*}
and
\begin{equation*}
S_{b,\gm}(\pi^{N/2},0,1/2)=S_{b,\gm}(1,0,1)\pi^{N/2}2^{-\gm/2}.
\end{equation*}

\item In order to prove the exponential rate of convergence to
  equilibrium, it was introduced in \cite{Mcmp} the modified
  linearized collision operator
$${\mathcal L}_m(h)=m^{-1}
ML_M(mM^{-1}h),\quad m(v)=e^{-a|v|^s}$$ with $M(v)=e^{-|v|^2}$, $a>0$
and $0<s<2$. It is proved in \cite{Mcmp} that ${\mathcal L}_m$ and
$L_M$ has the same spectrum, but ${\mathcal L}_m$ is available to
connect the exponential moment estimates of solutions.  The proof of
the original version of Theorem \ref{(theo4.1)} in \cite{Mcmp} used
additional technical assumptions: the angular function $b$ is convex
and nondecreasing in $(-1,1)$, and the constant $s$ in
$m(v)=e^{-a|v|^s}$ satisfies $0<s<\gm/2$.  These assumptions were only
used to prove the exponential moment estimate of the form
\eqref{(1.13)} (see Lemma 4.7 of \cite{Mcmp}). According to Theorem
\ref{(theo1.0)} in Section 1, these additional assumptions on the
function $b$ can be removed and the restriction $0<s<\gm/2$ can be
relaxed to $0<s\le \gm$. In particular one can choose $s=\gm$.

\item In \cite{Mcmp} it was actually assumed that $\gm\in(0,1]$;
  however this assumption was only used three times: 
  \begin{itemize}
  \item The first and second times are in \cite[Proof of
    Proposition~2.3, pp.643-645]{Mcmp}:
    \begin{itemize}
    \item first in
      ``$|v-v_*|^{\gm}\sin^{N-2}\theta/2\le
      (|v-v_*|\sin\theta/2)^{\gm}=|v-v'|^{\gm}$
      since $N-2\ge 1\ge \gm$'' [a key step in obtaining basic
      properties of the linearized collision operator $L_M$], but here
      the condition $0<\gm\le 1$ can be relaxed to $0<\gm\le N-2$ for
      any $N \ge 3$;
      \item second in
      ``$\|{\mathcal I}_{\dt}L^{+}\|_{L^2(M)}=O(\dt^{2-\gm})\to 0$ as
      $\dt\to 0^+$'', with
      ${\mathcal I}_{\dt}=\wt{\Theta}_{\dt}* {\bf 1}_{\{|\cdot|\le
        \dt^{-1}\}}$
      as defined in \cite[p.639]{Mcmp}. In this place, if we assume
      $0<\gm\le 2$ and let ${\bf 1}_{\{|\cdot|\le \dt^{-1}\}}$ be
      modified as ${\bf 1}_{\{|\cdot|\le \dt^{-1/2}\}}$, i.e.  we
      redefine
      ${\mathcal I}_{\dt}=\wt{\Theta}_{\dt}* {\bf 1}_{\{|\cdot|\le
        \dt^{-1/2}\}}$,
      then the same proof in \cite{Mcmp} also yields
      $\|{\mathcal I}_{\dt}L^{+}\|_{L^2(M)}=O(\dt^{2-\gm/2})\to 0$ as
      $\dt\to 0^+$.
    \end{itemize}

  \item The third time was in applying a regularity result from
    \cite{MR2081030}: in the latter paper the condition $0<\gm<2$ was
    used to ensure the existence and the uniqueness of the mild
    solution of equation~\eqref{(B)} constructed in
    \cite{MR1697562}. However since by Theorem~\ref{(theo1.0)}, the
    existence and the uniqueness of the mild solution of
    equation~\eqref{(B)} have been proven for all $0<\gm\le 2$, the
    results of \cite{MR2081030} mentioned above holds also for
    $\gm=2$.
  \end{itemize}

  Therefore the present assumption $0<\gm\le \min\{2, N-2\}$ satisfies
  all requirements and so the above Theorem \ref{(theo4.1)} and Lemma
  \ref{(lem4.1)} hold true.  Of course in the physical case, $N=3$,
  there is no improvement on $\gm$.

\item Let $f_t$ be a conservative mild solution of
  equation~\eqref{(B)} on $[\tau,\infty)$ with the initial datum
  $f_{\tau}$. Applying Theorem \ref{(theo4.1)} and Lemma
  \ref{(lem4.1)} to the solution $t\mapsto f_{t+\tau}$ on
  $[0,\infty)$, the exponential terms for the decay estimates of
  $\left\|f_t-M\right\|_{L^1}$ on the time interval $[\tau,\infty)$ is
  given by $e^{-\ld(t-\tau)}$.
\end{enumerate}
\end{remarks}

To prove Theorem \ref{(theo1.1)}, we need further preparation.

\begin{lemma}\label{(lem4.2)} Let  $0\le f\in L^1_{k+l}({\bRN})\cap L^2_1({\bRN})\cap
  H^s({\bRN})$ with $k\ge 0$, $l>0$, $s>0$. Let
\[
{\mathcal N}(f)={\mathcal N}_{\rho,u,T}(f)\in L^1_{1,0,1}({\bRN})
\]
be the normalization of $f$ defined in
\eqref{(1-general)},\eqref{(1-norm-f)}, and suppose that
$|\rho-1|+|u|+|T-1|\le 1/2$. Then
\begin{equation}\label{(4.2)}
  \left\|f-{\mathcal N}(f)\right\|_{L^1_k}\le
  C_{N,k,l,s}(f) \left(
    |\rho-1|+|u|+|T-1|\right)^{\frac{sl}{(1+s)(k+N+2l)}}
\end{equation}
where
\[
 C_{N,k,l,s}(f):=C_{N,k,l}\max\left\{\|f\|_{L^1_{k+l}},\,
\|f\|_{L^2_1},\, \|f\|_{H^s}\right\}
\]
and $C_{N,k,l}<\infty$ only depends on $N,k,l$.
\end{lemma}

\begin{proof} Recall that ${\mathcal N}(f)=\rho^{-1}T^{N/2}f(\sqrt{T}\, v+u)$.
Let ${\mathcal N}_1(f)=T^{N/2}f(\sqrt{T}\, v+u)$. Then
$$
\left\|f-{\mathcal N}(f)\right\|_{L^1_k}\le\left\|f-{\mathcal N}_1(f)\right\|_{L^1_k}
+2|\rho-1| \left\|{\mathcal N}_1(f)\right\|_{L^1_k}
$$
where we used $|1-\rho^{-1}|\le 2|\rho-1|$ because $1/2\le \rho\le
3/2$. We need to prove that
\begin{equation} \label{(lem4.1)}\left\|f-{\mathcal
      N}_1(f)\right\|_{L^1_k}\le C_{N} \left\|f-{\mathcal
        N}_1(f)\right\|_{L^1_{k+l}} ^{\frac{k+N}{k+N+2l}}
    \left\|\widehat{f}-\widehat{{\mathcal
          N}_1(f)}\right\|_{L^2}^{\frac{2l}{k+N+2l}}.
  \end{equation}
Let $h=f-{\mathcal N}_1(f)$, $R\in(0,\infty)$. We have
$$
\|h\|_{L^1_k}= \int_{\la v\ra<R}\la v\ra^k |h(v)|\, {\rm d}v+\int_{\la v\ra\ge
R}\la v\ra^k |h(v)|\, {\rm d}v\le
 C_N \|h\|_{L^2}R^{\frac{k+N}{2}}+\|h\|_{L^1_{k+l}}\frac{1}{R^{l}}.
$$
Minimizing the right hand side with respect to $R\in (0,\infty)$ leads to
$$
\|h\|_{L^1_k}\le 2 C_N ^{\frac{2 l}{k+N+2l}}  \|h\|_{L^1_{k+l}}^{\frac{k+N}{k+N+2l}}
\|h\|_{L^2}^{\frac{2l }{k+N+2l}}
$$
which gives \eqref{(lem4.1)} by Plancherel theorem
$\|h\|_{L^2}=(2\pi)^{-N/2}\|\widehat{h}\|_{L^2}$.

Since $1/2\le T\le 3/2$ and $|u|\le 1/2$ imply
\[
1+ \left|\frac{v-u}{\sqrt{T}}\right|^2\le 4(1+|v|^2),
\]
it follows that
$$
\left\|{\mathcal
    N}_1(f)\right\|_{L^1_{k+l}}=\int_{{\bRN}}\left(1+\left|\frac{v-u}{\sqrt{T}}\right|^2\right)^{(k+l)/2}
f(v)\, {\rm d}v \le 2^{k+l}\|f\|_{L^1_{k+l}}
$$
and thus
\begin{equation}\label{(4.3)}
  \left\|f-{\mathcal N}(f)\right\|_{L^1_k}\\
  \le C_{N,k,l}  \|f\|_{L^1_{k+l}} ^{\frac{k+N
    }{k+N+2l}} \left\|\widehat{f}-\widehat{{\mathcal N}_1(f)}\right\|_{L^2}^{\frac{2l
    }{k+N+2l}}+2^{k+1}|\rho-1|\|f\|_{L^1_k}.
\end{equation}
Next we compute
\begin{eqnarray*}&& \widehat{{\mathcal N}_1(f)}(\xi)
  =e^{{\rm i}T^{-1/2}\xi\cdot u}\widehat{f}(T^{-1/2}\xi),\\ \\
  && \left|1-e^{{\rm i}T^{-1/2}\xi\cdot u}\right|\le
  2\left(1+|T^{-1/2}\xi|^2\right)^{s/2}
 \max\{|u|,\, |u|^s\},\\ \\
  && \left|\widehat{f}(\xi)-\widehat{{\mathcal N}_1(f)}(\xi)\right|\le
  \left|\widehat{f}(\xi)-\widehat{f}(T^{-1/2}\xi)\right|+2\left|\widehat{f}(T^{-1/2}\xi)\right|
  \left(1+|T^{-1/2}\xi|^2\right)^{s/2}\max\{|u|,\,
  |u|^s\},
\end{eqnarray*}
hence
\begin{equation}\label{(4.7)}
  \left\|\widehat{f}-\widehat{{\mathcal N}_1(f)}\right\|_{L^2}  \le
  \left\|\widehat{f}-\widehat{f}\left(T^{-1/2}\cdot\right)\right\|_{L^2}+2^{1+N/4} \|f\|_{H^s}\max\{|u|,\,
  |u|^s\}.
\end{equation}
Write $\xi=(\xi_1,\xi_2, \dots,\xi_N), v=(v_1,v_2,\dots,v_N)$, and
$f_j(v)=v_j f(v), j=1,2,\dots,N$. Then
$$
\widehat{f}(\xi)-\widehat{f}\left(T^{-1/2}\xi\right)
=-{\rm
  i}\int_{\frac{1}{\sqrt{T}}}^{1}\sum_{j=1}^{N}\widehat{f_j}(t\xi)
\xi_j \, {\rm d}t.
$$
By Cauchy-Schwarz inequality and
\[
1/2\le T\le 3/2 \quad \Longrightarrow \quad
\left|\frac{1}{\sqrt{T}}-1\right|\le |T-1|
\]
we have
$$
\left|\widehat{f}(\xi)-\widehat{f}\left(T^{-1/2}\xi\right)\right| \le
|T-1|^{1/2}\left(\int_{1\wedge \frac{1}{\sqrt{T}}}^{1\vee
    \frac{1}{\sqrt{T}}}\sum_{j=1}^{N} |\widehat{f_j}(t\xi)|^2 \, {\rm
    d}t\right)^{1/2}\,|\xi|$$ where $a \wedge b=\min\{a,b\}$ and
$a\vee b=\max\{a,b\}$. Let $1/p+1/q=1$ with
$p=(1+s)/s$ and $q=1+s$. Then
\begin{multline*}
\left|\widehat{f}(\xi)-\widehat{f}\left(T^{-1/2}\xi\right)\right|^2=
\left|\widehat{f}(\xi)-\widehat{f}\left(T^{-1/2}\xi\right)\right|^{2/p}
\left|\widehat{f}(\xi)-\widehat{f}\left(T^{-1/2}\xi\right)\right|^{2/q}\\
\le |T-1|^{1/p}\left(\int_{1\wedge \frac{1}{\sqrt{T}}}^{1\vee \frac{1}{\sqrt{T}}}\sum_{j=1}^{N}
\left|\widehat{f_j}(t\xi)\right|^2 \, {\rm d}t\right)^{1/p}\,|\xi|^{2/p}
\left|\widehat{f}(\xi)-\widehat{f}\left(T^{-1/2}\xi\right)\right|^{2/q}.
\end{multline*}
It follows from H\"{o}lder  inequality and $q/p=s$ that
\begin{multline*}
\left\|\widehat{f}-\widehat{f}\left(T^{-1/2}\cdot\right)\right\|_{L^2}^2\\
\le |T-1|^{1/p}
\left(\int_{1\wedge \frac{1}{\sqrt{T}}}^{1\vee \frac{1}{\sqrt{T}}}\sum_{j=1}^{N}
\left(\int_{{\bRN}}\left|\widehat{f_j}(t\xi)\right|^2 \, {\rm d}\xi\right)\, {\rm d}t
\right)^{1/p}
\left\|\widehat{f}-\widehat{f}\left(T^{-1/2}\cdot\right)\right\|_{\dot{L}^2_s}
^{2/q}
\end{multline*}
where ${\dot{L}^2_s}$ denotes the weighted $L^2$ space with the
\emph{homogeneous} weight $|\xi|^{2s}$. Since, by Plancherel theorem,
\begin{multline*}
\int_{1\wedge \frac{1}{\sqrt{T}}}^{1\vee \frac{1}{\sqrt{T}}}
\left(\sum_{j=1}^{N} \int_{{\bRN}}\left|\widehat{f_j}(t\xi)\right|^2
\, {\rm d}\xi\right)\, {\rm d}t \\ =
(2\pi)^N\left( \int_{1\wedge \frac{1}{\sqrt{T}}}^{1\vee \frac{1}{\sqrt{T}}}t^{-N}{\rm
d}t \right) \left( \int_{{\bRN}}|v|^2 f(v)^2\, {\rm d}v \right) \le C_N\|f\|_{L^2_1}^2
\end{multline*}
and
$$
\left\|\widehat{f}-\widehat{f}\left(T^{-1/2}\cdot\right)\right\|_{\dot{L}^2_s}\le
\left\|\widehat{f}\right\|_{\dot{L}^2_s}+
\left\|\widehat{f}\left(T^{-1/2}\cdot\right)\right\|_{\dot{L}^2_s}\le
\left(1+2^{(N+2s)/4}\right)\|f\|_{H^s},
$$
it follows that
\begin{equation}\label{(4.8)}
\left\|\widehat{f}-\widehat{f}\left(T^{-1/2}\cdot\right)\right\|_{L^2} \le C_N \|f\|_{
L^2_1}^{1/p} \|f\|_{H^s}^{1/q}|T-1|^{1/(2p)}.
\end{equation}
Thus we get from \eqref{(4.7)}, \eqref{(4.8)}, and
$1/p=s/(1+s)$ that
\begin{equation}\label{(4.9)}
\left\|\widehat{f}-\widehat{{\mathcal N}_1(f)}\right\|_{L^2}\le
C_N\max\left\{ \|f\|_{L^2_1},\ \|f\|_{H^s}\right\}(|u|+
|T-1|)^{\frac{s}{2(1+s)}}
\end{equation}
and so \eqref{(4.2)} follows from \eqref{(4.3)} and \eqref{(4.9)}.
\end{proof}

In order to apply existing results on $L^1$ solutions to our measure
solutions, we shall use the \emph{Mehler transform}, which is defined
as follows:
\begin{definition}
  Let $\rho>0$, $u\in {\bRN}$ and $T>0$.  The \emph{Mehler transform}
  $I_n[F]$ of $F\in {\mathcal B}_{\rho,u,T}^{+}({\bRN})$ is given by
$$
I_n[F](v)=e^{Nn}\int_{{\bRN}}M_{1,0,T}\left(e^n\left(v-u-\sqrt{1-e^{-2n}}\,(v_*-u)\right)\right)
\, {\rm d}F(v_*),\quad n>0,
$$
where $M_{1,0,T}(v)=(2\pi T)^{-N/2}\exp(- |v|^2/(2T))$.
\end{definition}

The following lemma gives some basic properties of the Mehler
transform that we shall use in the proof of Theorem~\ref{(theo1.1)}.
\begin{lemma}\label{(lem4.4)} Given any $\rho>0$, $u\in{\bRN}$ and $T>0$. Let  $F\in{\mathcal B}_{\rho,u,T}^{+}({\bRN})$
  and let $M \in{\mathcal B}_{\rho,u,T}^{+}({\bRN})$ be the Maxwellian
  distribution. Then $I_n[F]\in L^1_{\rho,u,T}({\bRN}) $ and for any
  $0\le s\le 2$
\begin{equation}\label{(4.10*)}\lim_{n\to\infty}\left\|I_n[F]-M\right\|_{L^1_s}=\|F-M\|_s.\end{equation}
\end{lemma}

\begin{proof} Recall the basic formula of $I_n[F]$:
\begin{equation}\label{(4.12)}
\int_{{\bRN}}\psi(v)I_n[F](v)\, {\rm d}v=
\int_{{\bRN}}\left(\int_{{\bRN}}\psi\left(e^{-n} z+u+\sqrt{1-e^{-2n}}\,(v_*-u)\right)M_{1,0,T}(z)
\, {\rm d}z\right)\, {\rm d}F(v_*)\end{equation}
where $\psi$ is any Borel function on ${\bRN}$
satisfying $\sup\limits_{v\in{\bRN}}|\psi(v)|\la v\ra^{-2}<\infty$. 
This formula \eqref{(4.12)} is easily proved by using Fubini theorem and change of variables. From \eqref{(4.12)} it is easily deduced that $I_n[F]\in L^1_{\rho,u,T}({\bRN})$ for all $n>0$  and
\begin{equation}\label{(4.11)}
\lim_{n\to\infty} \int_{{\bRN}}\vp(v)I_n[F](v) \, {\rm
d}v=\int_{{\bRN}}\vp(v) \, {\rm d}F(v)
\end{equation}
for all $\vp\in C({\bRN})$ satisfying
$\sup\limits_{v\in{\bRN}}|\vp(v)|\la v\ra^{-2}<\infty$.

Let $0\le s\le 2$. Applying the dual version \eqref{(norm)} of the
norm $\|\cdot\|_s$ and the convergence \eqref{(4.11)} we have
\begin{equation}\label{(4.11*)}
\|F-M\|_s\le \liminf_{n\to\infty}\left\|I_n[F]-M\right\|_{L^1_s}.
\end{equation}

On the other hand we shall prove that 
\begin{equation}\label{(4.15)}
\limsup_{n\to\infty}\left\|I_n[F]-M\right\|_{L^1_s}\le \|F-M\|_s
\end{equation}
also holds true, and this together with \eqref{(4.11*)} then proves \eqref{(4.10*)}.

To prove \eqref{(4.15)},  we take
$$
\psi_n(v)=\la v\ra^s{\rm sign}\left(I_n[F](v)-M(v)\right).
$$
Then
\begin{multline*}
\left\|I_n[F]-M\right\|_s=
  \int_{{\bRN}}\psi_n(v)I_n[F](v)\, {\rm d}v-
  \int_{{\bRN}}\psi_n(v)M(v)\, {\rm d}v
  \\
  =\int_{{\bRN}}\left(\int_{{\bRN}}\psi_n\left(e^{-n}
    z+u+\sqrt{1-e^{-2n}}\,(v_*-u)\right)M_{1,0,T}(z) \, {\rm d}z\right)\,
  {\rm d}(F-M)(v_*)
  \\
  +\int_{{\bRN}}\left(\int_{{\bRN}}\psi_n\left(e^{-n}
    z+u+\sqrt{1-e^{-2n}}\,(v_*-u)\right)M_{1,0,T}(z) \, {\rm
      d}z\right)M(v_*)\, {\rm d}v_*-\int_{{\bRN}}\psi_n(v)M(v)\, {\rm
    d}v
  \\
  :=I_{n,1}+I_{n,2}.
\end{multline*}
Let $h$ be the sign function of $F-M$, i.e.,
$h:\mathbb{R}^N\to \mathbb{R}$ is a real Borel function satisfying 
\[
{\rm d}(F-M)(v_*)=h(v_*)\, {\rm d}|F-M|(v_*) \quad \mbox{ and }\quad h(v_*)^2\equiv
1
\]
(see e.g. \cite[Chapter~6]{MR0210528}). Then
$$
I_{n,1}\le \int_{{\bRN}}\int_{{\bRN}}\left\langle e^{-n} z+u+\sqrt{1-e^{-2n}}\,(v_*-u)\right\rangle ^sM_{1,0,T}(z)
\, {\rm d}z\, {\rm d}|F-M|(v_*).
$$
Since
$$
\forall\, (z,v_*)\in{\bRN}\times{\bRN}, \quad
\lim_{n\to\infty}\left\langle e^{-n} z+u+\sqrt{1-e^{-2n}}\,(v_*-u)\right\rangle
  ^s=\left\langle v_*\right\rangle ^s
$$
and
\[
\left\langle e^{-n} z+u+\sqrt{1-e^{-2n}}\,(v_*-u)\right\rangle^s \le 3^s\left\langle
u\right\rangle^s\left\langle z\ra^s\la v_*\right\rangle^s
\]
so that
$$\int_{{\bRN}}\int_{{\bRN}}3^s\la u\ra^s\la z\ra^s\la v_*\ra^s M_{1,0,T}(z)
\, {\rm d}z\, {\rm d}|F-M|(v_*)
=3^s\la u\ra^s\left\|M_{1,0,T}\right\|_s\|F-M\|_s<\infty,
$$
it follows from dominated convergence that
\begin{equation}\label{(4.13)}
\limsup_{n\to\infty}I_{n,1}\le\left\|M_{1,0,T}\right\|_{L^1}
\|F-M\|_s=\|F-M\|_s.
\end{equation}

Next we prove that $\limsup_{n\to\infty}I_{n,2}\le 0$.  We
compute by changing variable that
\begin{multline*}
\int_{{\bRN}}\left(\int_{{\bRN}}\psi_n\left(e^{-n}
  z+u+\sqrt{1-e^{-2n}}\,(v_*-u)\right) \, M_{1,0,T}(z)
\, {\rm d}z\right)M(v_*)\, {\rm d}v_*
\\
=\frac{1}{(1-e^{-2n})^{N/2}}
\int_{{\bRN}}M_{1,0,T}(z)
\, {\rm d}z\int_{{\bRN}}\psi_n(v) \, M\left(\frac{v-e^{-n}
    z-(1-\sqrt{1-e^{-2n}}) u}{\sqrt{1-e^{-2n}}}\right)\, {\rm d}v.
\end{multline*}
So we get
\begin{multline*}
 I_{n,2}\\
=
\frac{1}{(1-e^{-2n})^{N/2}}
\int_{{\bRN}}M_{1,0,T}(z)
\, {\rm d}z\int_{{\bRN}}\psi_n(v)\left[M\left(\frac{v-e^{-n} z-(1-\sqrt{1-e^{-2n}})u}{\sqrt{1-e^{-2n}}}\right)-M(v)\right]\, {\rm d}v\\
+\left(\frac{1}{(1-e^{-2n})^{N/2}}-1\right)\int_{{\bRN}}\psi_n(v)M(v)\, {\rm d}v
\\
\le \frac{1}{(1-e^{-2n})^{N/2}}
\intt_{{\bRRN}}\la v\ra^s\left|M\left(\frac{v-e^{-n}
      z-(1-\sqrt{1-e^{-2n}})u}{\sqrt{1-e^{-2n}}}\right)-M(v)\right|\, M_{1,0,T}(z) \, {\rm d}v\, {\rm d}z\\
+\left(\frac{1}{(1-e^{-2n})^{N/2}}-1\right)\int_{{\bRN}}\la v\ra ^sM(v)\, {\rm d}v
\end{multline*}
and finally (since the last last term above clearly converges to zero)
\begin{multline}\label{(4.14)}
\limsup_{n\to\infty} I_{n,2}\\
\le  \limsup_{n\to\infty}
\intt_{{\bRRN}}\la v\ra^s\left|M\left(\frac{v-e^{-n}
    z-(1-\sqrt{1-e^{-2n}})u}{\sqrt{1-e^{-2n}}}\right)-M(v)\right|M_{1,0,T}(z)
\, {\rm d}v\, {\rm d}z.
\end{multline}
It is obvious that the integrand in the right hand side of \eqref{(4.14)} tends to
zero as $n\to\infty$. To find a dominated function for the integrand,
we recall that
\begin{multline*}
M\left(\frac{v-e^{-n}
    z-(1-\sqrt{1-e^{-2n}})u}{\sqrt{1-e^{-2n}}}\right)\\
=\frac{\rho}{(2\pi
  T)^{N/2}}\exp\left(-\frac{1}{2T}\left|\frac{v-e^{-n}
      z-(1-\sqrt{1-e^{-2n}})u}{\sqrt{1-e^{-2n}}}-u\right|^2\right).
\end{multline*}
Elementary calculation shows that
$$
\frac{1}{2T} \left|\frac{v-e^{-n}
    z-(1-\sqrt{1-e^{-2n}})u}{\sqrt{1-e^{-2n}}}-u\right|^2 \ge
\frac{|v-u|^2}{4T}-\frac{|z|^2}{4T}.
$$
This gives
$$
M\left(\frac{v-e^{-n} z-(1-\sqrt{1-e^{-2n}})u}{\sqrt{1-e^{-2n}}}\right)M_{1,0,T}(z)
\le\frac{\sqrt{\rho}}{(2\pi T)^{N/2}}\sqrt{M(v)}\sqrt{M_{1,0,T}(z)}$$
and thus
\begin{multline*}
\la v\ra^s \left|M\left(\frac{v-e^{-n}
        z-(1-\sqrt{1-e^{-2n}})u}{\sqrt{1-e^{-2n}}}\right)-M(v)\right|M_{1,0,T}(z)
  \\
  \le \frac{\sqrt{\rho}}{(2\pi T)^{N/2}}\la
  v\ra^s\sqrt{M(v)}\sqrt{M_{1,0,T}(z)} +\la
  v\ra^sM(v)M_{1,0,T}(z).
\end{multline*}
By dominated convergence we then conclude that the limit in the right hand side of
\eqref{(4.14)} is zero. Therefore $\limsup_{n\to\infty}I_{n,2}\le 0$
and
$$
\lim\sup\limits_{n\to\infty}\left\|I_n[F]-M\right\|_s\le
\lim\sup\limits_{n\to\infty}I_{n,1}+
\lim\sup\limits_{n\to\infty}I_{n,2} \le \|F-M\|_s.
$$
This proves \eqref{(4.15)} and completes the proof of the lemma.
\end{proof}
\vskip2mm

Now we are ready to prove Theorem~\ref{(theo1.1)}.

\begin{proof}[Proof of Theorem \ref{(theo1.1)}]

  We first prove that the theorem holds true for all $L^1$ mild
  solutions in $L^1_{1,0,1}({\bRN})$. We shall then use approximation
  and normalization to extend it to general measure solutions.
  \smallskip

\noindent {\em Step 1}. Let $\ld=S_{b,\gm}(1,0,1)>0$ be the spectral
gap of the linearized operator $L_M$ associated with the kernel
$B(z,\sg)$ and the Maxwellian $M(v)=(2\pi)^{-N/2} e^{-|v|^2/2}$ in
$L^1_{1,0,1}({\bRN})$. Let $f_0\in L^1_{1,0,1}({\bRN})$ and let $f_t$
be the unique conservative mild $L^1$ solution of equation~\eqref{(B)}
with the initial datum $f_0$. We shall prove that
\begin{equation}\label{(4.17)}
\forall\, t\ge
    0, \quad \left\|f_t-M\right\|_{L^1_2}\le
  C_0\left\|f_0-M\right\|_{L^1_2}^{1/2}e^{-\ld t}
\end{equation}
where the constant $0<C_0<\infty$ depends only on $N$, $\gamma$, and the function
$b$.

To do this we use Theorem \ref{(theo3.1)} to consider the positive decomposition:
$$
\forall \, t \ge 1, \quad f_t=g_t+h_t,
$$
where $g_t=f^{n}_t\ge 0$ and $h_t=h^{n}_t\ge 0$ are given in
\eqref{(3.1)}-\eqref{(3.5)} with $n=N_{\gm}+2$ and $t_0=1$.  In the
following we denote by $c_i>0, C_{i}>0\,(i=1,2,\dots)$ some finite
constants that depend  only on $N, \gamma$, and the function $b$. By Theorem \ref{(theo3.1)} with $t_0=1$ we have
\begin{eqnarray}&& \label{(4.20)}
\forall \, t\ge 1, \quad \left\|f_t-g_t\right\|_{L^1_2}=\left\|h_t\right\|_{L^1_2} \le
C_1e^{-\frac{1}{2}at},\\
&&\label{(4.21}
\sup_{t\ge 1}\left\{\left\|g_t\right\|_{L^1_{N+4}},\, \left\|g_t\right\|_{L^2_1},\,
\left\|g_t\right\|_{H^1}\right\} \le C_{2}.\end{eqnarray}
We can assume that $C_1\ge 1$. Let $\tau_0=(2/a)\log(8C_1)\,(>1)$ and let
\begin{equation*}
\begin{cases} \displaystyle
\rho_t=\int_{{\bRN}}g_t(v)\, {\rm d}v,\,\,\,
u_t=\frac{1}{\rho_t}\int_{{\bRN}}v g_t(v)\, {\rm d}v,\,\,\,
T_t=\frac{1}{N\rho_t}\int_{{\bRN}}|v-u_t|^2g_t(v) \, {\rm d}v,
\vspace{0.2cm} \\ \displaystyle
{\mathcal
  N}(g_t)(v)=\frac{T_t^{N/2}}{\rho_t}g_t(\sqrt{T_t}\,v+u_t),\quad t\ge
\tau_0.
\end{cases}
\end{equation*}
Using the relation
\[
T_t=\frac{1}{N\rho_t}\int_{{\bRN}}|v|^2g_t(v)\, {\rm
  d}v-\frac{|u_t|^2}{N}
\]
we compute
\begin{equation} \label{(4.23)}
  \forall \, t\ge \tau_0, \quad \left|\rho_t-1\right|+\left|u_t\right|+\left|T_t-1\right|\le
  4\left\|f_t-g_t\right\|_{L^1_2}\le 4C_1e^{-\frac{1}{2}at}\le
  \frac{1}{2}.
\end{equation}
So by \eqref{(4.21} and applying Lemma \ref{(lem4.2)} (with $k=2,
l=N+2, s=1$) we have
$$
\forall \, t\ge \tau_0, \quad \left\|g_t-{\mathcal N}(g_t)\right\|_{L^1_2}\le C_3 \left(4C_1 e^{-\frac{1}{2}a
t}\right)^{1/6}=C_4 e^{-c_1t}.
$$
This together with \eqref{(4.20)} gives
\begin{equation}\label{(4.24)}
  \forall \, t\ge
  \tau_0, \quad \left\|f_t-{\mathcal N}(g_t)\right\|_{L^1_2}\le\left\|f_t-g_t\right\|_{L^1_2}+
  \left\|g_t-{\mathcal N}(g_t)\right\|_{L^1_2} \le C_5e^{-c_1t}.
\end{equation}
Also by \eqref{(4.23)}, $\sup\limits_{t\ge
  1}\left\|g_t\right\|_{L^2}\le C_2$, $\tau_0>1$ and
${\mathcal N}(g_t)\in L^1_{1,0,1}({\bRN})$ we have
\begin{equation}\label{(4.25)}
  C_6\le \inf_{t\ge \tau_0}\left\|{\mathcal
      N}(g_t)\right\|_{L^2},\quad \sup_{t\ge \tau_0} \left\|{\mathcal
      N}(g_t)\right\|_{L^2}\le C_7.
\end{equation}
The second inequality follows from elementary calculations and the
bounds $1/2\le \rho_t$ and $T_t\le 3/2$.
To prove the first one, we consider some $R>0$ and write
$$
1=\int_{|v|<R}{\mathcal N}(g_t)(v)\, {\rm d}v+\int_{|v|\ge R}{\mathcal
  N}(g_t)(v)\, {\rm d}v\le \left|{\mathbb
  B}^N\right|^{1/2}R^{N/2}\left\|{\mathcal N}(g_t)\right\|_{L^2}
+R^{-2}N,
$$
where $|{\mathbb B}^N|$ is the volume of the unite ball ${\mathbb
  B}^N$. If we now fix $R=\sqrt{2N}$, then
\[
\frac{1}{2}\left|{\mathbb B}^N\right|^{-1/2}(2N)^{-N/4}\le\left\|{\mathcal
    N}(g_t)\right\|_{L^2}
\]
for all $t\ge \tau_0$ so that the first inequality in \eqref {(4.25)}
holds for $C_6=(1/2)|{\mathbb B}^N|^{-1/2}(2N)^{-N/4}$.

To prove \eqref{(4.17)} we use the following technique of ``moving
solutions" as used in \cite{MR1697495} and \cite{Mcmp}. For any
$\tau\ge \tau_0$, let $(t,v)\mapsto f_{t}^{(\tau)}(v)$ be the unique
conservative solution on $[\tau,\infty)\times{\bRN}$ with the initial
datum at time $t = \tau$:
$$
f^{(\tau)}_{t}|_{t=\tau}=f_\tau^{(\tau)}= {\mathcal N}(g_\tau).
$$
On the one hand, by Theorem \ref{(theo4.1)}, we have
$$
\forall\, t\ge \tau, \quad \left\|f_{t}^{(\tau)}-M\right\|_{L^1}\le
C_{f_\tau^{(\tau)}}e^{-\ld (t-\tau)}
$$
where the coefficient $0<C_{f_\tau^{(\tau)}} <\infty$ depends only on
$N$, $\gamma$, the function $b$, and $\|f_\tau^{(\tau)}\|_{L^2}$.
Since \eqref{(4.25)} implies $C_6\le \|f_\tau^{(\tau)}\|_{L^2}\le C_7$
for all $\tau\ge \tau_0$, it follows from Remark
\ref{remark1.UpBd}-(3) that
$\sup\limits_{\tau\ge \tau_0}C_{f_\tau^{(\tau)}}\le C_8$, and thus for
every $\tau\ge \tau_0$ we have
\begin{equation}\label{(4.26)}
  \forall\, t\ge \tau, \quad \left\|f_{t}^{(\tau)}-M\right\|_{L^1}\le
  C_8e^{-\ld (t-\tau)}.
\end{equation}
On the other hand using the stability estimate \eqref{(1.21*)} we have
\begin{equation}\label{(4.27)}
  \forall \, t\ge\tau, \quad \left\|f_t-f^{(\tau)}_t\right\|_{L^1_2}\le \left\| f_{\tau}-f_\tau^{(\tau)}\right\|_{L^1_2} e^{c_2(t-\tau)}.
\end{equation}
Since \eqref{(4.24)} and $\tau\ge \tau_0$ imply
\[
\left\|f_{\tau}-f_\tau^{(\tau)}\right\|_{L^1_2}=\left\|f_{\tau}-{\mathcal
    N}(g_\tau)\right\|_{L^1_2}\le C_5 e^{-c_1\tau},
\]
it follows from \eqref{(4.26)} and \eqref{(4.27)} that
\begin{equation}\label{(4.28)}
  \forall \, t\ge \tau, \quad \left\|f_t-M\right\|_{L^1}\le \left\|f_t-f_t^{(\tau)}\right\|_{L^1}+
  \left\|f_t^{(\tau)}-M\right\|_{L^1}\le
  C_5e^{-c_1\tau+c_2(t-\tau)}+C_8e^{-\ld (t-\tau)}.
\end{equation}
Now for any
\[
t\ge t_1:=\frac{c_1+c_2+\ld}{c_2+\ld}\tau_0, \quad
\mbox{we choose} \quad \tau =\tau(t)=\frac{c_2+\ld}{c_1+c_2+\ld} t.
\]
Then $t>\tau(t)\ge \tau(t_1)=\tau_0$ and
$$
-c_1\tau(t)+c_2(t-\tau(t))=-\frac{c_1\ld }{c_1+c_2+\ld}t:=-c_3
t.
$$
Thus applying \eqref{(4.28)} with  $t>\tau=\tau(t)$ (for all $t\ge t_1$) we obtain
\begin{equation}\label{(4.29)}
  \forall\,t\ge t_1, \quad \left\|f_t-M\right\|_{L^1}\le (C_5+C_8)e^{-c_3 t}.
\end{equation}
Now let
\[
m(v):=\exp\left(-\frac{\alpha(1)}{4}|v|^{\gm}\right)
\]
where $\alpha(t)>0$ is given in Theorem \ref{(theo1.0)} for the
initial datum $F_0\in {\mathcal B}_{1,0,1}^{+}({\bRN})$ defined by
${\rm d}F_0(v)=f_0(v)\, {\rm d}v$. Then
\begin{equation}\label{(4-moment)}
\sup_{t\ge 1}\left\|f_t\right\|_{L^1(m^{-4})} \le 2,\quad
\left\|M\right\|_{L^1(m^{-4})}\le C_9.
\end{equation}
Therefore, using Cauchy-Schwarz inequality and \eqref{(4.29)},
we get
\begin{equation}\label{(4.30)}
  \forall\, t\ge t_1, \quad
  \left\|f_t-M\right\|_{L^1(m^{-2})}\le\left\|f_t-M\right\|_{L^1(m^{-4})}^{1/2}\left\|f_t-M\right\|_{L^1}^{1/2}
  \le C_{10}e^{-c_4 t}.
\end{equation}
Let $\vep>0$ be the constant in Theorem 4.1 corresponding to $m(v)$,
and let us choose
$$
t_2=\max\left\{t_1,\,\frac{1}{c_4}\log\left(\frac{C_{10}}{\vep}\right)\right\}.
$$
Then we deduce from \eqref{(4.30)} that
$$
\forall\, t\ge t_2, \quad \left\|f_t-M\right\|_{L^1(m^{-2})}\le C_{10}e^{-c_4 t}\le \vep.
$$
It follows from Lemma \ref{(lem4.1)} (see Remark \ref{remark4.1}-(3)) that
\begin{equation}\label{(4.31)}
  \forall\, t\ge t_2, \quad \left\|f_t-M\right\|_{L^1(m^{-1})}\le C_{11}\left\|f_{t_2}-M\right\|_{L^1(m^{-1})}e^{-\ld
    (t-t_2)}.
\end{equation}
Next, applying the elementary inequality
$$
1+|v|^2\le 
C e^{\dt |v|^{2\eta}},\qquad
\eta:=\frac{\gm}{2},\quad \dt:=\frac{\alpha(1)}{4}
$$
for some constant $C=C_{\eta,\dt}>0$, we have
\begin{equation} \label{(4.31*)}
\left\|f_t-M\right\|_{L^1_2}\le
C_{12}\left\|f_t-M\right\|_{L^1(m^{-1})}.
\end{equation}
On the other hand, using the bound in \eqref{(4-moment)}, we have
\begin{equation}\label{(4.33)}
\left\|f_{t_2}-M\right\|_{L^1(m^{-1})}\le
\left\|f_{t_2}-M\right\|_{L^1(m^{-2})}^{1/2}
\left\|f_{t_2}-M\right\|_{L^1_2}^{1/2}\le
C_{13}\left\|f_{t_2}-M\right\|_{L^1_2}^{1/2}.
\end{equation}
It follows from \eqref{(4.31*)}, \eqref{(4.31)} and \eqref{(4.33)} that
\begin{equation}\label{(4.34)}
   \forall \, t\ge t_2, \quad \left\|f_t-M\right\|_{L^1_2}\le
  C_{14}\left\|f_{t_2}-M\right\|_{L^1_2}^{1/2} e^{-\ld t}.
\end{equation}

It remains to estimate $\|f_t-M\|_{L^1_2}$ in terms of
$\|f_0-M\|_{L^1_2}$ for $t\in[0,t_2]$. To do this we use the estimate
in \cite[p. 3359 line 4]{partI} for the measure $H_t:=F_t-G_t$ where $F_t,
G_t$ are measure solutions of the equation~\eqref{(B)}. Here we define
more precisely $F_t, G_t$ to be
\[
{\rm d}F_t(v)=M(v)\, {\rm d}v \quad \mbox{ and } \quad {\rm
  d}G_t(v)=f_t(v)\, {\rm d}v.
\]
Then $\|H_t\|_2=\|M-f_t\|_{L^1_2}, \|f_t\|_{L^1_2}=\|M\|_{L^1_2}$, and
thus (recalling $A_0=1$)
\begin{multline*}
\forall \, t \in [r,\infty), \quad \left\|M-f_t\right\|_{L^1_2}\le \\
2\left\|(M-f_r)^{+}\right\|_{L^1_2}+
4\left(\|M\|_{L^1_{2+\gm}}+\|M\|_{L^1_2}\right)
\int_{r}^{t}\left\|M-f_s\right\|_{L^1_{\gm}}\, {\rm d}s.
\end{multline*}
Since $t\mapsto f_t(v)\ge 0$ is continuous on $[0,\infty)$ for
a.e. $v\in {\mathbb R}^N$, it follows from dominated convergence that
$$
2\left\|(M-f_r)^{+}\right\|_{L^1_2}\xrightarrow[r \to 0^+]{}
2\left\|(M-f_0)^{+}\right\|_{L^1_2}=\left\|f_0-M\right\|_{L^1_2},
$$
where the last equality follows from \[
|f_0-M|=f_0-M+2(M-f_0)^{+} \quad \mbox{ and } \quad
\|f_0\|_{L^1_2}=\|M\|_{L^1_2}.
\]
Thus letting $r\to 0^+$ gives
$$
\forall \, t \in [0,\infty), \quad \left\|f_t-M\right\|_{L^1_2}\le
\left\|f_0-M\right\|_{L^1_2}+c_5\int_{0}^{t}\left\|f_s-M\right\|_{L^1_{\gm}}\,
{\rm d}s.
$$
Since $\|f_s-M\|_{L^1_{\gm}}\le \|f_s-M\|_{L^1_2}$, it follows from
Gronwall lemma that
\begin{equation}\label{(4.32)}
  \forall \,  t\ge 0, \quad \|f_t-M\|_{L^1_2}\le \|f_0-M\|_{L^1_2}
  e^{c_5t}.
\end{equation}

Inserting this estimate with $t=t_2$ into the right hand side of
\eqref{(4.34)} gives
\begin{equation}\label{(4.34*)}
\forall \, t\ge t_2, \quad \left\|f_t-M\right\|_{L^1_2}\le
C_{15}\left\|f_0-M\right\|_{L^1_2}^{1/2} e^{-\ld t}.
\end{equation}
Also because of \eqref{(4.32)} and $\left\|f_{0}-M\right\|_{L^1_2}\le
2(1+N)$, we have
\begin{equation}\label{(4.35)}
\forall \, t \in [0,t_2], \quad \left\|f_t-M\right\|_{L^1_2}
\le C_{16}\left\|f_0-M\right\|_{L^1_2}^{1/2} e^{-\ld t}.
\end{equation}
Combining \eqref{(4.34*)}, \eqref{(4.35)} we then obtain \eqref{(4.17)} with
$C_0=\max\{C_{15},\, C_{16}\}$.
\smallskip

\noindent {\em Step 2.} Let us now prove that \eqref{(4.17)} holds also true
for all measure solutions in ${\mathcal B}_{1,0,1}^{+}({\bRN})$. Given any
$F_0\in {\mathcal B}^{+}_{1,0,1}({\bRN})$, let $F_t$ be the unique
conservative measure strong solution of equation~\eqref{(B)} with the
initial datum $F_0$.  By part (e) of Theorem \ref{(theo1.0)} and Lemma
\ref{(lem4.4)}, there is a sequence $f_{k,t}\in L^1_{1,0,1}({\bRN})$
of solutions with the initial data $f_{k,0}:=I_{n_k}[F_0]\in
L^1_{1,0,1}({\bRN})$ such that
\begin{eqnarray}\label{(4.18)}
  && \forall\,\vp\in C_b({\bRN}),\quad
  \forall\,t\ge 0, \quad \lim_{k\to\infty}\int_{{\bRN}}\vp(v)
  f_{k,t}(v)\, {\rm d}v
  =\int_{{\bRN}}\vp(v)\, {\rm d}F_t(v), \\
  &&\label{(4.18-initial)}
  \lim_{k\to\infty}\left\|f_{k,0}-M\right\|_{L^1_2}=
  \left\|F_0-M\right\|_{2}.
\end{eqnarray}
Using the formulation \eqref{(dual)} of the norm
$\left\|\cdot\right\|_2$, we conclude from \eqref{(4.18)} that
\begin{equation} \label{(4.18F)}
\forall \, t\ge 0, \quad \left\|F_t-M\right\|_2\le
\lim\sup_{k\to\infty}\left\|f_{k,t}-M\right\|_{L^1_2}.
\end{equation}
On the other hand, applying \eqref{(4.17)} to $f_{k,t}$, we have
\begin{equation}\label{(4.19)}
\forall \, t\ge
0,\ k=1,2,3,\dots, \quad \left\|f_{k,t}-M\right\|_{L^1_2}
\le C_0\left\|f_{k,0}-M\right\|_{L^1_2}^{1/2}e^{-\ld t}.
\end{equation}
Combining  \eqref{(4.18F)}, \eqref{(4.19)} and \eqref{(4.18-initial)} we obtain
\begin{equation}\label{(4-standard)}
\forall \, t\ge 0, \quad \left\|F_t-M\right\|_2\le
C_0\left\|F_0-M\right\|_2^{1/2}e^{-\ld t}.
\end{equation}
\smallskip

\noindent {\em Step 3.} Finally we show that for any $\rho>0$,
$u\in{\bRN}$ and $T>0$, the theorem holds true for all measure solutions
in ${\mathcal B}^{+}_{\rho,u,T}({\bRN})$. Let $F_0\in {\mathcal
  B}^{+}_{\rho,u,T}({\bRN})$ and let $F_t$ be the unique conservative
measure strong solution with the initial datum $F_0$. Let
$M_{\rho,u,T}\in {\mathcal B}_{\rho,u,T}^{+}({\bRN})$ be the
Maxwellian and let ${\mathcal N}={\mathcal N}_{\rho,u,T}$ be the
normalization operator. By Proposition \ref{(prop1.1)}, the flow
$t\mapsto {\mathcal N}(F_{t/c})$ is the unique conservative measure
strong solution of equation~\eqref{(B)} with the initial datum
${\mathcal N}(F_0)\in{\mathcal B}_{1,0,1}^{+}({\bRN})$. Here $c=\rho
T^{\gm/2}$.  Since ${\mathcal N}(M_{\rho,u,T})\in {\mathcal
  B}_{1,0,1}^{+}({\bRN})$ is the standard Maxwellian, it follows from
the above result \eqref{(4-standard)} that (writing ${\mathcal N}(F_t)={\mathcal
  N}(F_{ct/c})$)
$$
\forall\, t\ge 0, \quad \left\|{\mathcal N}(F_{t})-{\mathcal
    N}(M_{\rho,u,T})\right\|_2 \le C_0\left\|{\mathcal
    N}(F_{0})-{\mathcal N}(M_{\rho,u,T})\right\|_2^{1/2}e^{-\ld ct}.
$$
Then, applying Proposition \ref{(prop1.1)}, we have
\begin{multline*}
\forall\, t\ge 0, \quad \left\|F_t-M_{\rho,u,T}\right\|_2\le
C_{1/\rho,|u|/\sqrt{T},1/T}\left\|{\mathcal N}(F_{t})-{\mathcal N}(M_{\rho,u,T})\right\|_2
\\
\qquad\qquad \qquad\quad  \le C_0C_{1/\rho,|u|/\sqrt{T},1/T}\left\|{\mathcal N}(F_{0})-{\mathcal N}(M_{\rho,u,T})\right\|_2^{1/2}e^{-\ld ct}\\
\le C_0C_{1/\rho,|u|/\sqrt{T},1/T}[C_{\rho,|u|,T}]^{1/2}\left\|F_{0}-M_{\rho,u,T}\right\|_2^{1/2}e^{-\ld ct}.\qquad\,\,
\end{multline*}
Since $\ld c=S_{b,\gm}(1,0,1)\rho T^{\gm/2}=S_{b,\gm}(\rho,u,T)$ is
the spectral gap of the linearized operator $L_{M_{\rho,u,T}}$,
this completes the proof of Theorem \ref{(theo1.1)}.
\end{proof}

\section{Lower Bound of Convergence Rate }
\label{sec5}

In this section we prove Theorem \ref{(theo1.2)}. Recall that we
assume here that $\gm\in(0, 2]$ and that the function $b$ satisfies
only \eqref{(Grad-1)}.

\begin{proof}[Proof of Theorem \ref{(theo1.2)}] We first prove the
  theorem for the standard case, i.e. assuming
  $F_0, M\in{\mathcal B}^{+}_{1,0,1}({\bRN})$.  By
$$
\forall\, t_1,t_2\in[0,\infty), \quad
F_{t_2}=F_{t_1}+\int_{t_1}^{t_2}Q(F_{\tau},F_{\tau})\, {\rm d}\tau
$$
and $Q(M,M)=0$, we have
\begin{equation}\label{(5.1)}
  \forall \, t_1, t_2\in[0,\infty), \quad \Big|\left\|F_{t_2}-M\right\|_0 -\left\|F_{t_1}-M\right\|_0\Big|\le
  \left| \int_{t_1}^{t_2}\left\|Q(F_{\tau},F_{\tau})-Q(M,M)\right\|_0
    \, {\rm d}\tau\right|.
\end{equation}
Using the inequalities in \eqref{(1-Q-differ-bound)}, $0<\gm\le 2$, and the conservation of mass and energy (which implies
$\left\|F_t\right\|_{\gm}\le \left\|F_t\right\|_{2}=1+N$, etc.)
we have
\begin{equation}\label{(5.2)}
\left\|Q(F_{t},F_{t})-Q(M,M)\right\|_0\le
2^{4}(1+N)\left\|F_{t}-M\right\|_{\gm}.
\end{equation}
Since $t\mapsto \left\|F_{t}-M\right\|_{\gm}$ is bounded and,
by H\"{o}lder inequality,
\begin{equation}\label{(5.3)}
  \left\|F_{t}-M\right\|_{\gm}\le
  \left\|F_t-M\right\|_0^{1-\gm/2}\left\|F_t-M\right\|_2^{\gm/2},
\end{equation}
it follows from \eqref{(5.1)}-\eqref{(5.3)} that $t\mapsto
\|F_t-M\|_0$ is Lipschitz continuous and
\begin{equation}\label{(5.4)}
  \left|\frac{\, {\rm d}}{\, {\rm d}t}\left\|F_{t}-M\right\|_0\right|\le
  2^{4}(N+1)\left\|F_t-M\right\|_0^{1-\gm/2}\left\|F_t-M\right\|_2^{\gm/2}
  \quad {\rm a.e.}\quad
  t\in(0,\infty).\end{equation}
Next, thanks to the exponential decay of the Maxwellian,
we show that $\|F_t-M\|_2$ can be controlled by $\|F_t-M\|_0$ (see e.g. \eqref{(5.8)} below). In fact we show that this property holds for all measure $F\in{\mathcal B}^{+}_{1,0,1}({\bRN})$.
To do this,  let $(M-F)^{+}$ be the positive part of $M-F$, i.e., $(M-F)^{+}=\fr{1}{2}(|M-F|+ M-F)$. Then $|M-F|=F-M+2(M-F)^{+}$. Let
$h$ be the sign function of $M-F$, i.e., $h(v)^2\equiv 1$ such that ${\rm d}(M-F)=h{\rm d}|M-F|$.
Then
${\rm d}(M-F)^{+}=\fr{1}{2}(1+h){\rm d}(M-F)$. From these we have
\begin{equation}
{\rm d}|M-F|={\rm d}(F-M)+2{\rm d}(M-F)^{+}\quad {\rm and}\quad {\rm d}(M-F)^{+}\le {\rm d}M
\label{(FM)}
\end{equation}
where the inequality part is due to $F\ge 0$. Now since
$F,M$ have the same mass and energy, it follows from
\eqref{(FM)} and $|F-M|=|M-F|$ that
\begin{equation}\label{(5.5)}
\left\|F-M\right\|_0=2\left\|(M-F)^{+}\right\|_0,\quad
\left\|F-M\right\|_2=2\left\|(M-F)^{+}\right\|_2.
\end{equation}
Let $0<\dt<1$. Applying Jensen inequality to
the convex function $x\mapsto \exp(\dt x/2)$
and the measure $(M-F)^{+}$  and assuming
$\|(M-F)^{+}\|_0>0$ we have
\begin{equation}\label{(5.6)}
\frac{1}{\left\|(M-F)^{+}\right\|_0}\int_{\mathbb{R}^N}\exp(\dt\la v\ra^2/2)\, {\rm d}(M-F)^{+}\ge
\exp\left(\frac{\dt}{2}\cdot\frac{\left\|(M-F)^{+}\right\|_2}{\left\|(M-F)^{+}\right\|_0}\right).\end{equation}
On the other hand we have
\[
{\rm d}(M-F)^{+}(v)\le \, {\rm d}M(v)=\frac{1}{(2\pi)^{N/2}}
\exp(-|v|^2/2)\, {\rm d}v
\]
and
\begin{equation}\label{(5.7)}
  \int_{{\bf R}^N}\exp(\dt\la v\ra^2/2)\, {\rm d}(M-F)^{+}(v)
  \le \int_{{\bf R}^N}\exp(\dt\la v\ra^2/2)\, {\rm d}M(v)=e^{\dt/2}\left(\frac{1}{1-\dt}\right)^{N/2}.\end{equation}
Let us now choose $\dt=\frac{1}{N+1}$. Then
\[
e^{\dt/2}\left(\frac{1}{1-\dt}\right)^{N/2}=e^{\frac{1}{2(N+1)}}
\left(1+1/N\right)^{N/2}<2
\]
and thus, from \eqref{(5.5)}-\eqref{(5.7)}, we deduce
$$
\exp\left(\frac{1}{2(N+1)}\cdot\frac{\left\|F-M\right\|_2}{\left\|F-M\right\|_0}\right)\le
\frac{4}{\left\|F-M\right\|_0},
$$
i.e.
\begin{equation}\label{(5.8)}
\left\|F-M\right\|_2\le 2(N+1)\left\|F-M\right\|_0\log
\left(\frac{4}{\left\|F-M \right\|_0}\right).
\end{equation}
If we adopt the convention $x\log (1/x)=0$ for $x=0$, the inequality
\eqref{(5.8)} also holds for $\|F-M\|_0=0$.

Now let us go back to the solution $F_t$. To avoid discussing the
case $\|F_t-M\|_0=0$ for some $t$, we consider
$$
U_{\vep}(t):=\frac{\|F_t-M\|_0+\vep}{4},\quad 0<\vep<1.
$$
Then
$$
\frac{\|F_t-M\|_0}{4}<U_{\vep}(t)\le \frac{2+\vep}{4}<\frac{2}{3}, \quad
0<\vep <\frac{2}{3}.
$$
Using the inequality
\[
\forall \, 0\le x\le y\le \frac{2}{3}, \quad x\log \frac{1}{x} \le 2
y\log \frac{1}{y}
\]
and \eqref{(5.8)} (with $F=F_t$), we then obtain
$$\frac{\left\|F_t-M\right\|_2}{4}
\le 4(N+1)U_{\vep}(t)\log \left(\frac{1}{U_{\vep}(t)}
\right).$$
Thus by \eqref{(5.4)} we deduce
\begin{equation}\label{(5.9)}
\left|\frac{\, {\rm d}}{\, {\rm d}t}U_{\vep}(t)\right|\le
A U_{\vep}(t)\left [\log \left(\frac{1}{U_{\vep}(t)}\right)\right ]^{\gm/2}
\quad {\rm a.e.}\quad
t\in(0,\infty)\end{equation}
where $A=2^{6}(N+1)^2$.
\smallskip

\noindent {\em Case 1: $0<\gm<2$.} In this case we have, by
\eqref{(5.9)},
\begin{equation*}
\frac{\, {\rm d}}{\, {\rm d}t}
\left[\log\left(\frac{1}{U_{\vep}(t)}\right)\right]^{1-\gm/2}
=-(1-\gm/2)\left[\log\left(\frac{1}{U_{\vep}(t)}\right)\right]^{-\gm/2}\frac{1}{U_{\vep}(t)}\cdot
\frac{\, {\rm d}}{\, {\rm d}t}U_{\vep}(t) \le (1-\gm/2)A
\end{equation*}
for almost every $t\in(0,\infty)$. Observe that the function
$$t\mapsto \left[\log\left(\frac{1}{U_{\vep}(t)}\right)\right]^{1-\gm/2}$$
is absolutely continuous on every bounded interval of $[0,\infty)$.
It follows that
\begin{equation}\label{(5.10)}
  \forall\, t\ge 0, \quad
  \left[\log\left(\frac{1}{U_{\vep}(t)}\right)\right]^{1-\gm/2}
  \le
  \left[\log\left(\frac{1}{U_{\vep}(0)}\right)\right]^{1-\gm/2}+(1-\gm/2)A
  t.
\end{equation}
Next, using the convexity inequality
$$
\forall \, x,y\ge 0, \quad \left(x+\left(1-\frac{\gamma}{2}
  \right)y\right)^{\frac{1}{1-\gm/2}} \le \frac{\gamma}{2} \left(
  \frac{2 x}{\gm}
\right)^{\frac{1}{1-\gm/2}}+\left(1-\frac{\gm}{2}\right)
y^{\frac{1}{1-\gm/2}},$$
we have
$$
\left\{\left[\log\left(\frac{1}{U_{\vep}(0)}\right)\right]^{1-\gm/2}+\left(1-\frac{\gm}{2}\right)A
  t \right\}^{\frac{1}{1-\gm/2}}\le \alpha
\log\left(\frac{1}{U_{\vep}(0)}\right) +\beta_1 t^{\frac{2}{2-\gm}}
$$
where
$$
\alpha=\left(\frac{2}{\gm}\right)^{\frac{\gm}{2-\gm}},\quad
\beta_1=\left(1-\frac{\gm}{2}\right)
A^{\frac{2}{2-\gm}}=\left(1-\frac{\gm}{2}\right)
\left(2^{6}(N+1)^2\right)^{\frac{2}{2-\gm}}.
$$
Thus, from \eqref{(5.10)}, we obtain
$$
\forall\, t\ge 0, \quad U_{\vep}(t)\ge
U_{\vep}(0)^{\alpha} \exp\left(-\beta_1 t^{\frac{2}{2-\gm}}\right).
$$
Using the definition of $U_{\vep}(t)$ and letting $\vep\to 0+$, we get
finally
$$
\forall\, t\ge 0, \quad \frac{\|F_t-M\|_0}{4}\ge \left
  (\frac{\|F_0-M\|_0}{4}\right)^{\alpha} \exp\left(-\beta_1
t^{\frac{2}{2-\gm}}\right).
$$
This concludes the proof of the standard case for $0<\gm<2$.
\smallskip

\noindent {\em Case 2: $\gm=2$.} In this case we have by \eqref{(5.9)}
with $\gm=2$ that
$$\frac{\, {\rm d}}{\, {\rm d}t}
\log\left(\log\left(\frac{1}{U_{\vep}(t)}\right)\right)=-
\left[\log\left(\frac{1}{U_{\vep}(t)}\right)\right]^{-1}\frac{1}{U_{\vep}(t)}\cdot
\frac{\, {\rm d}}{\, {\rm d}t}U_{\vep}(t)\le A,\quad {\rm a.e.}\quad
t\in(0,\infty).$$ Since the function
$$
t \mapsto \log\left(\log\left(\frac{1}{U_{\vep}(t)}\right)\right)
$$
is absolutely continuous on every bounded interval of $[0,\infty)$, it follows that for all $t>0$
$$
\log\left(\log\left(\frac{1}{U_{\vep}(t)}\right)\right)
\le \log\left(\log\left(\frac{1}{U_{\vep}(0)}\right)\right)+A t,
\qquad {\rm i.e.}\quad
 U_{\vep}(t)
\ge (U_{\vep}(0))^{e^{At}}.
$$
Letting $\vep\to 0+$ leads to
$$
\forall\, t\ge 0, \quad \frac{\|F_t-M\|_0}{4}\ge \left
  (\frac{\|F_0-M\|_0}{4}\right)^{e^{At}}.
$$
This prove the standard case for $\gm=2$.
\smallskip

\noindent {\em General non-normalized setting.} The general case can
be reduced to the standard case by using normalization. Let $F_0\in
{\mathcal B}_{\rho,u,T}^+({\bRN})$, $c=\rho T^{\gm/2}$ and let $M\in
{\mathcal B}_{\rho,u,T}^+({\bRN})$ be the Maxwellian.  Then, according to
Proposition \ref{(prop1.1)}, the normalization $t\mapsto {\mathcal
  N}(F_{t/c})\in {\mathcal B}^{+}_{1,0,1}({\bRN})$ is a conservative
measure strong solution of equation~\eqref{(B)} with the initial datum
${\mathcal N}(F_0)$.  Applying the above estimates and
$\|F_{t}-M\|_0=\rho\|{\mathcal N}(F_{t})-{\mathcal N}(M)\|_0$ we obtain
that if $0<\gm<2$ then
\begin{multline*}
 \forall\, t\ge 0, \quad  \|F_t-M\|_0 =\rho\|{\mathcal N}(F_{c^{-1}ct})-{\mathcal N}(M)\|_0
\\
\ge  4\rho\left(\frac{\|{\mathcal N}(F_0)-{\mathcal N}(M)\|_0}{4}\right)^{\alpha}
\exp\left(-\beta_1(ct)^{\frac{2}{2-\gm}}\right)\\
=4\rho\left(\frac{\|F_0-M\|_0}{4\rho}\right)^{\alpha}\exp\left(-\beta\,t^{\frac{2}{2-\gm}}
\right)\qquad \qquad \qquad \qquad \quad \qquad \qquad\end{multline*}
with
\[
\beta=\beta_1c^{\frac{2}{2-\gm}}=
(1-\frac{\gm}{2}) \left(2^{6}(1+N)^2\rho
  T^{\gm/2}\right)^{\frac{2}{2-\gm}}.
\]

Similarly  if $\gm=2$, then
$$\|F_t-M\|_0\ge
4\rho\left (\frac{\|{\mathcal N}(F_{0})-{\mathcal N}(M)\|_0}{4}\right)^{e^{Act}}
= 4\rho\left (\frac{\|F_{0}-M\|_0}{4\rho}\right)^{e^{\kappa\,t}}
$$
with $\kappa=Ac=2^{6}(N+1)^2\rho T$.  This completes the
proof. \end{proof} \vskip2mm

\section{Global in Time Stability Estimate}
\label{sec6}

In the last section we prove the the global in time strong stability
of the measure strong solutions of equation~\eqref{(B)}.

\begin{proof}[Proof of Theorem \ref{(theo1.3)}]
  Let $F_t$ be a conservative measure strong solution of
  equation~\eqref{(B)} with the initial datum $F_0\in {\mathcal
    B}_{\rho_0,u_0,T_0}^{+}({\bRN})$, and let $G_t$ be any
  conservative measure strong solution of equation~\eqref{(B)} with
  the initial datum $G_0$.  Let
\begin{equation}\label{(6.1)}
  D_0:=\min\left\{\frac{\rho_0}{2},\,
    \left(
      \frac{4\left\|F_0\right\|_2}{N\rho_0^2}
      +
      \frac{6}{N}\left(\frac{\left\|F_0\right\|_2}{\rho_0^2}\right)^2\right)^{-1}\frac{T_0}{2}
  \right\}.
\end{equation}
If $\left\|F_0-G_0\right\|_2\ge D_0$,
then by conservation of mass and energy we have  for all $t\ge 0$,
\begin{equation}\label{(6.2)}
  \left\|F_t-G_t\right\|_2\le
  \left\|F_0\right\|_2+\left\|G_0\right\|_2
  \le 2\left\|F_0\right\|_2+\left\|G_0-F_0\right\|_2\le
  \left(\frac{2\left\|F_0\right\|_2}{D_0}+1\right)\left\|G_0-F_0\right\|_2.
\end{equation}

In the following we assume that $\left\|F_0-G_0\right\|_2< D_0$. By
the uniqueness theorem, we can also assume that
$\left\|F_0-G_0\right\|_2>0$. Due to our choice of $D_0$, we see that
$G_0$ is non-zero and is not a Dirac distribution. Therefore let
$\rho>0$, $u\in{\bRN}$, $T>0$ be the mass, mean velocity and
temperature corresponding to $G_0$, i.e., $G_0\in {\mathcal
  B}_{\rho,u,T}^{+}({\bRN})$. Using the condition
$\left\|F_0-G_0\right\|_2< D_0$ and elementary estimates we have
\begin{equation*}
\begin{cases} \displaystyle
  |\rho-\rho_0|\le \left\|G_0-F_0\right\|_2, \quad
  0<\frac{\rho_0}{2}\le \rho\le \frac{3\rho_0}{2}, \vspace{0.2cm} \\ \displaystyle
  |u-u_0|\le
  \frac{2\left\|F_0\right\|_2}{\rho_0^2}\left\|G_0-F_0\right\|_2,
  \vspace{0.2cm} \\ \displaystyle
  |T-T_0|\le\left( \frac{4\left\|F_0\right\|_2}{N\rho_0^2} +
    \frac{6}{N}\left(\frac{\left\|F_0\right\|_2}{\rho_0^2}\right)^2\right)\left\|G_0-F_0\right\|_2,
  \vspace{0.2cm} \\ \displaystyle
  0<\frac{T_0}{2}\le T\le \frac{3T_0}{2}.
\end{cases}
\end{equation*}

Let $M_{F_0}, M_{G_0}$ be the Maxwellians associated with $F_0, G_0$
respectively, i.e. $M_{F_0}\in{\mathcal
  B}_{\rho_0,u_0,T_0}^{+}({\bRN})$, $M_{G_0}\in {\mathcal
  B}_{\rho,u,T}^{+}({\bRN})$.  In the following calculations the
constants $0<C_i<\infty$ $(i=1,2,\dots,9)$ only depend on $N$, the
function $b$, $\gm$, $\rho_0$, $u_0$ and $T_0$, and we recall that
\[
\left\|F_0\right\|_2=\rho_0(1+NT_0+|u_0|^2).
\]

We need to estimate
$\left\|M_{G_0}-M_{F_0}\right\|_2$. Let us define
$$
{\mathcal M}(\rho,u,T; v)=(2\pi)^{-N/2}\rho T^{-N/2}\exp\left(-\frac{|v-u|^2}{2T}\right)
$$
and let us compute
\begin{equation*}
\begin{cases}\displaystyle
\frac{\p }{\p \rho}{\mathcal M}(\rho,u,T; v)
={\mathcal M}(1,u,T; v), \vspace{0.2cm} \vspace{0.2cm} \\
\displaystyle
\nabla_{u}{\mathcal M}(\rho,u,T; v)
={\mathcal M}(\rho,u,T; v) \, \frac{v-u}{T},
\vspace{0.2cm} \\ \displaystyle
\frac{\p }{\p T}{\mathcal M}(\rho,u,T; v)
=\left(-\frac{(N/2+1)}{T}+\frac{|v-u|^2}{2T^2}\right)
{\mathcal M}(\rho,u,T; v).
\end{cases}
\end{equation*}
If we set
$$
\forall \, \theta \in [0,1], \quad \rho(\theta)=\theta\rho+(1-\theta)\rho_0,
\,\,\,u(\theta)=\theta u+(1-\theta)u_0,\,\,\,
T(\theta)=\theta T+(1-\theta)T_0,
$$
then
\begin{multline*}
\left|{\mathcal M}(\rho,u,T; v)-{\mathcal M}(\rho_0,v_0,T_0; v)
\right|
\le  |\rho-\rho_0|
\int_{0}^{1}{\mathcal M}(1,u(\theta),T(\theta); v)\, {\rm d}\theta\\
+|u-u_0|\int_{0}^{1}{\mathcal M}(\rho(\theta),u(\theta),T(\theta); v)
\frac{|v-u(\theta)|}
{T(\theta)}
\, {\rm d}\theta \\
+|T-T_0|\int_{0}^{1}{\mathcal M}(\rho(\theta),u(\theta),T(\theta); v)
\left(\frac{(N/2+1)}{T(\theta)}+\frac{|v-u(\theta)|^2}{2T(\theta)^2}\right)
\, {\rm d}\theta.
\end{multline*}
We then deduce
\begin{multline*}
\left\|M_{G_0}-M_{F_0}\right\|_2=\int_{{\bRN}}\la v\ra^2
\left|{\mathcal M}(\rho,u,T; v)-{\mathcal M}(\rho_0,u_0,T_0; v)
\right|\, {\rm d}v\\
\le C_1
\left(|\rho-\rho_0|+|u-u_0|+|T-T_0|
\right)
\end{multline*}
and thus using the above estimates for $\rho-\rho_0$, $u-u_0$ and
$T-T_0$, we obtain
\begin{equation}\label{(6.3)}
\left\|M_{G_0}-M_{F_0}\right\|_2\le C_2\left\|G_0-F_0\right\|_2.
\end{equation}

Next from the above estimates we have
$$
\ld_0=S_{b,\gm}(1,0,1) \rho_0 T_0^{\gm/2},\quad \ld=S_{b,\gm}(1,0,1)
\rho T^{\gm/2} \ge 2^{-1-\gm/2}\ld_0.
$$
Then using the convergence estimate
\eqref{(4.36)} and recalling that
\[
C_{1/\rho,|u|/\sqrt{T},1/T}=\rho\max\left\{ 1+|u|^2+\sqrt{T}|u|,\,
  T+\sqrt{T}|u|\right\},
\]
we have
\begin{equation*}
\begin{cases} \displaystyle
\left\|F_t-M_{F_0}\right\|_2\le  C_3e^{-\ld_0t}, \vspace{0.2cm} \\ \displaystyle
\left\|G_t-M_{G_0}\right\|_2\le C_0C_{1/\rho,|u|/\sqrt{T},1/T}e^{-\ld t}
 \le C_{4} \exp\left(- 2^{-1-\gm/2}\ld_0 t\right).
\end{cases}
\end{equation*}
Thus
$$
\forall\, t\ge 0, \quad
\left\|F_t-M_{F_0}\right\|_2+\left\|G_t-M_{G_0}\right\|_2 \le
C_5e^{-C_6 t},
$$
and it follows from \eqref{(6.3)} that
\begin{equation}\label{(6.4)}
  \forall\, t\ge 0, \quad \left\|F_t-G_t\right\|_2\le C_5e^{-C_6 t}+C_2\left\|F_0-G_0\right\|_2.
\end{equation}
On the other hand by the stability estimate \eqref{(1.21)} we have
\begin{equation}\label{(6.5)}
  \forall\,t\ge 0, \quad \left\|F_t-G_t\right\|_2\le
  \Psi_{F_0}\left(\left\|F_0-G_0\right\|_2\right)e^{C_{7}(1+t)}.
\end{equation}

The remaining of the proof is concerning with balancing properly
\eqref{(6.4)} and \eqref{(6.5)}.  \smallskip

\noindent {\em Case 1: $\Psi_{F_0}(\left\|F_0-G_0\right\|_2)<1$.} Note
that $\left\|F_0-G_0\right\|_2>0$ implies
$\Psi_{F_0}(\left\|F_0-G_0\right\|_2>0$. Let
$$
t_0=
\log\left\{\left(\frac{1}{\Psi_{F_0}(\left\|F_0-G_0\right\|_2)}\right)^{\frac{1}{C_6+C_7}}\right\}.
$$
For every $t\ge 0$, if $t\le t_0$, then, using \eqref{(6.5)},
\begin{equation*}
\left\|F_t-G_t\right\|_2\le \Psi_{F_0}\left(\left\|F_0-G_0\right\|_2\right)e^{C_{7}(1+t_0)}
=e^{C_{7}}\left[\Psi_{F_0}\left(\left\|F_0-G_0\right\|_2\right)\right]^{\frac{C_{6}}{C_6+C_{7}}}.
\end{equation*}
If  $t\ge t_0$, then, using \eqref{(6.4)},
\begin{equation*}
\left\|F_t-G_t\right\|_2\le C_5\Big[\Psi_{F_0}(\left\|F_0-G_0\right\|_2)\Big]^{\frac{C_6}{C_6+C_7}}
+C_2\left\|F_0-G_0\right\|_2.
\end{equation*}
Thus
\begin{equation}\label{(6.6)}
\forall\, t\ge 0, \quad \left\|F_t-G_t\right\|_2\le C_{8}
\Big[\Psi_{F_0}(\left\|F_0-G_0\right\|_2)\Big]^{\frac{C_6}{C_6+C_{7}}}
+C_2\left\|F_0-G_0\right\|_2.
\end{equation}
\smallskip

\noindent
{\em Case 2: $\Psi_{F_0}(\left\|F_0-G_0\right\|_2)\ge 1$.}
In this case we have, by conservation of mass and energy and
$\left\|F_0-G_0\right\|_2\le \rho_0/2\le \left\|F_0\right\|_2/2$ that
\begin{equation}\label{(6.7)}
  \forall\, t\ge 0, \quad \left\|F_t-G_t\right\|_2\le
  \frac{5}{2}\left\|F_0\right\|_2\le
  \frac{5}{2}\left\|F_0\right\|_2\Big[\Psi_{F_0}(\left\|F_0-G_0\right\|_2)\Big]^{\frac{C_6}{C_6+C_{7}}}.
\end{equation}

Combining \eqref{(6.6)}, \eqref{(6.7)}, and \eqref{(6.2)}, we obtain
$$
\forall\, t\ge 0, \quad \left\|F_t-G_t\right\|_2\le C_{9}\left(
  \left[\Psi_{F_0}\left(\left\|F_0-G_0\right\|_2\right)\right]^{\frac{C_6}{C_6+C_{7}}}
+\left\|F_0-G_0\right\|_2 \right).
$$
This proves Theorem
\ref{(theo1.3)}
\end{proof}

\vskip3mm \medskip\par\noindent\emph{Acknowledgements.} This work was
started while the first author was visiting the University
Paris-Dauphine as an invited professor during the autumn 2006, and the
support of this university is acknowledged. The first author also
acknowledges support of National Natural Science Foundation of China,
Grant No.10571101, No.11171173. The second author's research was supported by the
ERC Starting Grant MATKIT.

\par\noindent{\scriptsize\copyright\ 2015 by the authors. This paper
  may be reproduced, in its entirety, for non-commercial purposes.}

\bibliographystyle{acm}
\bibliography{Biblio-LM}

\begin{thebibliography}{10}

\bibitem{MR1697495}
{\sc Abrahamsson, F.}
\newblock Strong {$L^1$} convergence to equilibrium without entropy conditions
  for the {B}oltzmann equation.
\newblock {\em Comm. Partial Differential Equations 24}, 7-8 (1999),
  1501--1535.

\bibitem{Baranger-Mouhot}
{\sc Baranger, C., and Mouhot, C.}
\newblock Explicit spectral gap estimates for the linearized {B}oltzmann and
  {L}andau operators with hard potentials.
\newblock {\em Rev. Mat. Iberoamericana 21}, 3 (2005), 819--841.

\bibitem{Bobylev1975}
{\sc Bobyl{\"e}v, A.~V.}
\newblock The method of the {F}ourier transform in the theory of the
  {B}oltzmann equation for {M}axwell molecules.
\newblock {\em Dokl. Akad. Nauk SSSR 225}, 6 (1975), 1041--1044.

\bibitem{Boby:maxw:88}
{\sc Bobyl{\"e}v, A.~V.}
\newblock The theory of the nonlinear spatially uniform {B}oltzmann equation
  for {M}axwell molecules.
\newblock In {\em Mathematical physics reviews, Vol.\ 7}, vol.~7 of {\em Soviet
  Sci. Rev. Sect. C Math. Phys. Rev.} Harwood Academic Publ., Chur, 1988,
  pp.~111--233.

\bibitem{MR1639275}
{\sc Bouchut, F., and Desvillettes, L.}
\newblock A proof of the smoothing properties of the positive part of
  {B}oltzmann's kernel.
\newblock {\em Rev. Mat. Iberoamericana 14}, 1 (1998), 47--61.

\bibitem{Carleman}
{\sc Carleman, T.}
\newblock {\em Probl\`emes math\'ematiques dans la th\'eorie cin\'etique des
  gaz}.
\newblock Publ. Sci. Inst. Mittag-Leffler. 2. Almqvist \& Wiksells Boktryckeri
  Ab, Uppsala, 1957.

\bibitem{CC92}
{\sc Carlen, E.~A., and Carvalho, M.~C.}
\newblock Strict entropy production bounds and stability of the rate of
  convergence to equilibrium for the {B}oltzmann equation.
\newblock {\em J. Statist. Phys. 67}, 3-4 (1992), 575--608.

\bibitem{MR1263387}
{\sc Carlen, E.~A., and Carvalho, M.~C.}
\newblock Entropy production estimates for {B}oltzmann equations with
  physically realistic collision kernels.
\newblock {\em J. Statist. Phys. 74}, 3-4 (1994), 743--782.

\bibitem{CC94}
{\sc Carlen, E.~A., and Carvalho, M.~C.}
\newblock Entropy production estimates for {B}oltzmann equations with
  physically realistic collision kernels.
\newblock {\em J. Statist. Phys. 74}, 3-4 (1994), 743--782.

\bibitem{MR1725612}
{\sc Carlen, E.~A., Carvalho, M.~C., and Gabetta, E.}
\newblock Central limit theorem for {M}axwellian molecules and truncation of
  the {W}ild expansion.
\newblock {\em Comm. Pure Appl. Math. 53}, 3 (2000), 370--397.

\bibitem{MR2546739}
{\sc Carlen, E.~A., Carvalho, M.~C., and Lu, X.}
\newblock On strong convergence to equilibrium for the {B}oltzmann equation
  with soft potentials.
\newblock {\em J. Stat. Phys. 135}, 4 (2009), 681--736.

\bibitem{MR1991033}
{\sc Carlen, E.~A., and Lu, X.}
\newblock Fast and slow convergence to equilibrium for {M}axwellian molecules
  via {W}ild sums.
\newblock {\em J. Statist. Phys. 112}, 1-2 (2003), 59--134.

\bibitem{MR2765747}
{\sc Desvillettes, L., Mouhot, C., and Villani, C.}
\newblock Celebrating {C}ercignani's conjecture for the {B}oltzmann equation.
\newblock {\em Kinet. Relat. Models 4}, 1 (2011), 277--294.

\bibitem{Grad1}
{\sc Grad, H.}
\newblock Principles of the kinetic theory of gases.
\newblock In {\em Handbuch der {P}hysik (herausgegeben von {S}. {F}l\"ugge),
  {B}d. 12, {T}hermodynamik der {G}ase}. Springer-Verlag, Berlin, 1958,
  pp.~205--294.

\bibitem{Grad2}
{\sc Grad, H.}
\newblock Asymptotic theory of the {B}oltzmann equation. {II}.
\newblock In {\em Rarefied {G}as {D}ynamics ({P}roc. 3rd {I}nternat. {S}ympos.,
  {P}alais de l'{UNESCO}, {P}aris, 1962), {V}ol. {I}}. Academic Press, New
  York, 1963, pp.~26--59.

\bibitem{GMM}
{\sc Gualdani, M., Mischler, S., and Mouhot, C.}
\newblock Factorization of non-symmetric operators and applications.
\newblock Tech. rep., preprint, 2013.

\bibitem{MR1511713}
{\sc Hilbert, D.}
\newblock Begr\"undung der kinetischen {G}astheorie.
\newblock {\em Math. Ann. 72}, 4 (1912), 562--577.

\bibitem{MR0056184}
{\sc Hilbert, D.}
\newblock {\em Grundz\"uge einer allgemeinen {T}heorie der linearen
  {I}ntegralgleichungen}.
\newblock Chelsea Publishing Company, New York, N.Y., 1912 (1953).

\bibitem{MR1284432}
{\sc Lions, P.-L.}
\newblock Compactness in {B}oltzmann's equation via {F}ourier integral
  operators and applications. {I}, {II}.
\newblock {\em J. Math. Kyoto Univ. 34}, 2 (1994), 391--427, 429--461.

\bibitem{MR1663589}
{\sc Lu, X.}
\newblock A direct method for the regularity of the gain term in the
  {B}oltzmann equation.
\newblock {\em J. Math. Anal. Appl. 228}, 2 (1998), 409--435.

\bibitem{MR2946633}
{\sc Lu, X.}
\newblock On backward solutions of the spatially homogeneous {B}oltzmann
  equation for {M}axwelian molecules.
\newblock {\em J. Stat. Phys. 147}, 5 (2012), 991--1006.

\bibitem{partI}
{\sc Lu, X., and Mouhot, C.}
\newblock On measure solutions of the {B}oltzmann equation, part {I}: moment
  production and stability estimates.
\newblock {\em J. Differential Equations 252}, 4 (2012), 3305--3363.

\bibitem{MR1697562}
{\sc Mischler, S., and Wennberg, B.}
\newblock On the spatially homogeneous {B}oltzmann equation.
\newblock {\em Ann. Inst. H. Poincar\'e Anal. Non Lin\'eaire 16}, 4 (1999),
  467--501.

\bibitem{MR2254617}
{\sc Mouhot, C.}
\newblock Explicit coercivity estimates for the linearized {B}oltzmann and
  {L}andau operators.
\newblock {\em Comm. Partial Differential Equations 31}, 7-9 (2006),
  1321--1348.

\bibitem{Mcmp}
{\sc Mouhot, C.}
\newblock Rate of convergence to equilibrium for the spatially homogeneous
  {B}oltzmann equation with hard potentials.
\newblock {\em Comm. Math. Phys. 261}, 3 (2006), 629--672.

\bibitem{MR2301289}
{\sc Mouhot, C.}
\newblock Quantitative linearized study of the {B}oltzmann collision operator
  and applications.
\newblock {\em Commun. Math. Sci.}, suppl. 1 (2007), 73--86.

\bibitem{MR2081030}
{\sc Mouhot, C., and Villani, C.}
\newblock Regularity theory for the spatially homogeneous {B}oltzmann equation
  with cut-off.
\newblock {\em Arch. Ration. Mech. Anal. 173}, 2 (2004), 169--212.

\bibitem{MR0210528}
{\sc Rudin, W.}
\newblock {\em Real and complex analysis}.
\newblock McGraw-Hill Book Co., New York, 1966.

\bibitem{MR0304972}
{\sc Stein, E.~M., and Weiss, G.}
\newblock {\em Introduction to {F}ourier analysis on {E}uclidean spaces}.
\newblock Princeton University Press, Princeton, N.J., 1971.
\newblock Princeton Mathematical Series, No. 32.

\bibitem{Villani-handbook}
{\sc Villani, C.}
\newblock A review of mathematical topics in collisional kinetic theory.
\newblock In {\em Handbook of mathematical fluid dynamics, {V}ol. {I}}.
  North-Holland, Amsterdam, 2002, pp.~71--305.

\bibitem{Vi03a}
{\sc Villani, C.}
\newblock Cercignani's conjecture is sometimes true and always almost true.
\newblock {\em Comm. Math. Phys. 234}, 3 (2003), 455--490.

\bibitem{WCUh:LBE:70}
{\sc Wang~Chang, C.~S., Uhlenbeck, G.~E., and de~Boer, J.}
\newblock In {\em Studies in Statistical Mechanics, Vol. V}. North-Holland,
  Amsterdam, 1970.

\end{thebibliography}

\signxl \signcm

\end{document}